\documentclass[12pt,reqno]{amsart}
\usepackage{amssymb}
\usepackage{amsfonts}
\usepackage{hyperref}

\def\equationautorefname~#1\null{(#1)\null}

\theoremstyle{plain}
\newtheorem*{acknowledgment}{Acknowledgment}

\newtheorem{algorithm}{Algorithm}[section]
\newtheorem{thm}{Thm}

\newtheorem{corollary}[algorithm]{Corollary}

\newtheorem{definition}[algorithm]{Definition}
\newtheorem*{Remark-delta}{Remark on all things $\protect\delta$}

\newtheorem{lemma}[algorithm]{Lemma}

\newtheorem{theorem} [algorithm] {Theorem}
\newtheorem{theoremlet}[thm]{Theorem}
\newtheorem{corollarylet}[thm]{Corollary}
\newtheorem{lemmalet}[thm]{Lemma}

\newtheorem*{theoremnonum}{Theorem}
\newtheorem*{TNST}{Tubular Neighborhood Stability Theorem}

\newtheorem{definitionlet}[thm]{Definition}
\newtheorem*{remarknonum}{Remark}

\newtheorem{keylemma}[algorithm]{Key Lemma}

\newtheorem{proposition}[algorithm]{Proposition}
\newtheorem{remark}[algorithm]{Remark}

\numberwithin{equation}{algorithm}

\input{tcilatex}

\begin{document}
\title{Diffeomorphism Stability and Codimension Three}
\author{Curtis Pro}
\address{{Department of Mathematics, }California State University, Stanislaus%
}
\email{cpro@csustan.edu}
\author{Frederick Wilhelm}
\thanks{This work was supported by a grant from the Simons Foundation
(\#358068, Frederick Wilhelm)}
\address{{Department of Mathematics, University of California, Riverside}}
\email{fred@math.ucr.edu}
\urladdr{https://sites.google.com/site/frederickhwilhelmjr/home}

\begin{abstract}
Given $k\in \mathbb{R},$ $v,$ $D>0,$ and $n\in \mathbb{N},$ let $\left\{
M_{\alpha }\right\} _{\alpha =1}^{\infty }$ be a Gromov-Hausdorff convergent
sequence of Riemannian $n$--manifolds with sectional curvature $\geq k,$
volume $>v,$ and diameter $\leq D.$ Perelman's Stability Theorem implies
that all but finitely many of the $M_{\alpha }$s are homeomorphic. The
Diffeomorphism Stability Question asks whether all but finitely many of the $%
M_{\alpha }$s are diffeomorphic.

We answer this question affirmatively in the special case when all of the
singularities of the limit space occur along Riemannian manifolds of
codimension $\leq 3$. We then describe several applications. For instance,
if the limit space is an orbit space whose singular strata are of
codimension $\leq 3,$ then all but finitely many of the $M_{\alpha }$s are
diffeomorphic.
\end{abstract}

\maketitle

\pdfbookmark[1]{Introduction}{Introduction}

Let $\mathcal{M}_{k,v,d}^{K,V,D}\left( n\right) $ denote the class of closed
Riemannian $n$--manifolds $M$ with 
\begin{equation*}
\begin{array}{cclccc}
k & \leq & \sec \,M & \leq & K, &  \\ 
v & \leq & \mathrm{vol}\,M & \leq & V, & \mathrm{and} \\ 
d & \leq & \mathrm{diam}\,M & \leq & D, & 
\end{array}%
\end{equation*}%
where $\sec \,M$ is the sectional curvature of $M$, $\mathrm{vol}\,M$ is the
volume of $M$, and $\mathrm{diam}\,M$ is the diameter of $M$.

Let $\left\{ M_{\alpha }\right\} _{\alpha =1}^{\infty }$ $\subset \mathcal{M}%
_{k,v,0}^{\infty ,\infty ,D}\left( n\right) $ converge in the
Gromov-Hausdorff topology to $X.$ Perelman's Stability Theorem implies that
all but finitely many of the $M_{\alpha }$s are homeomorphic to $X$ (\cite%
{Perel1}, \cite{Kap2}). Motivated by this it is natural to ask the\vspace*{%
0.15in}

\noindent \textbf{Diffeomorphism Stability Question.} \emph{Given }$k\in 
\mathbb{R},$\emph{\ }$v,D>0,$\emph{\ and }$n\in N,$\emph{\ let }$\left\{
M_{\alpha }\right\} _{\alpha =1}^{\infty }\subset M_{k,v,0}^{\infty ,\infty
,D}\left( n\right) $\emph{\ be a convergent sequence. Are all but finitely
many of the }$M_{\alpha }$\emph{s diffeomorphic?}\vspace*{0.15in}

If $\left\{ M_{\alpha }\right\} _{\alpha =1}^{\infty }$ happens to lie in $%
\mathcal{M}_{k,v,0}^{K,\infty ,D}\left( n\right) $ for some $K\in \mathbb{R}%
, $ then by Gromov's Compactness Theorem, $X$ is a $C^{1,\alpha }$
Riemannian manifold, and all but finitely many of the $M_{\alpha }$s are $%
C^{1}$--diffeomorphic to $X$ (\cite{GLP}, \cite{Nik}).

An affirmative answer to the Diffeomorphism Stability Question would provide
a simultaneous generalization of the Finiteness Theorems of Cheeger (\cite%
{Cheeg2}) and Grove-Petersen-Wu (\cite{GrovPetWu}). In addition, Grove and
the second author proved the following.

\begin{theoremnonum}
(\cite{GrovWilh2}) If the answer to the Diffeomorphism Stability Question is
\textquotedblleft yes\textquotedblright , then every Riemannian $n$%
--manifold $M$ with $\sec \,M\geq 1$ and $\mathrm{diam}\,M>\frac{\pi }{2}$
is diffeomorphic to $S^{n}.$
\end{theoremnonum}

We answer the Diffeomorphism Stability Question affirmatively in the special
case when all the singularities of $X$ occur along Riemannian manifolds of
codimension $\leq 3$. Before stating the result, we define the concept of a
space being diffeomorphically stable.

\begin{definitionlet}
A space $X\in \mathrm{closure}\left( \mathcal{M}_{k,v,0}^{\infty ,\infty
,D}\left( n\right) \right) $ is diffeomorphically stable if for any sequence 
$\left\{ M_{\alpha }\right\} _{\alpha =1}^{\infty }\subset \mathcal{M}%
_{k,v,0}^{\infty ,\infty ,D}\left( n\right) $ with $M_{\alpha
}\longrightarrow X,$ in the Gromov--Hausdorff topology, all but finitely
many of the $M_{\alpha }$s are diffeomorphic.
\end{definitionlet}

The definition of a non-singular point we use was introduced in \cite{BGP},
where it is called an \textquotedblleft $\left( n,\delta \right) $--burst
point\textquotedblright . Elsewhere in the literature, $\left( n,\delta
\right) $--burst points are called $\left( n,\delta \right) $-strained
points (see also Definition \autoref{strainer} below). Since Alexandrov
spaces have singular points, we define a notion of isometric embeddings that
generalizes the usual definition in Riemannian Geometry that is formulated
in terms of pull-backs.

\begin{definitionlet}
\label{Dnf sm and isom}Let $X$ be an Alexandrov space and $\left( S,g\right) 
$ a Riemannian manifold. Let $\mathrm{dist}^{X}$ be the distance of $X$ and $%
\mathrm{dist}^{S}$ the distance on $S$ induced by $g.$ An embedding $\iota
:\left( S,g\right) \hookrightarrow X$ is infinitesimally isometric if and
only if for every $\varepsilon >0$ there is a $\delta >0$ so that for
distinct $a,b\in S$ with $\mathrm{dist}^{S}\left( a,b\right) <\delta ,$ 
\begin{equation}
\left\vert \frac{\mathrm{dist}^{S}\left( a,b\right) }{\mathrm{dist}%
^{X}\left( \iota \left( a\right) ,\iota \left( b\right) \right) }%
-1\right\vert <\varepsilon .  \label{bi lip dfn}
\end{equation}
\end{definitionlet}

\begin{theoremlet}
\label{dif stab- dim 4}There is a $\delta \left( k,v,D,n\right) >0$ so that $%
X\in \mathrm{closure}\left( \mathcal{M}_{k,v,0}^{\infty ,\infty ,D}\left(
n\right) \right) $ is diffeomorphically stable provided $X$ contains a
finite collection $\mathcal{S}\equiv \left\{ S_{i}\right\} _{i\in I}$ of
infinitesimally isometrically embedded, pairwise disjoint, Riemannian
manifolds $S_{i}$ without boundary that have the following properties.

\begin{enumerate}
\item Every point of $X\setminus \left\{ \cup _{i\in I}S_{i}\right\} $ is $%
\left( n,\delta \right) $-strained$.$

\item No point of any $S\in \mathcal{S}$ is $\left( \dim \left( S\right)
+1,\delta \right) $-strained.

\item $\mathcal{S}$ is the union of two subcollections $\mathcal{K}$ and $%
\mathcal{N}.$

\item Elements of $\mathcal{K}$ are compact and have codimension $\leq 3.$

\item Elements of $\mathcal{N}$ are not compact and have codimension $\leq 2$%
.

\item The closure of an element of $N\in \mathcal{N}$ is a union of elements
of $\mathcal{S}.$
\end{enumerate}
\end{theoremlet}

\begin{remarknonum}
Of course an isometric embedding $\iota :\left( S,g\right) \hookrightarrow M$
of Riemannian manifolds is an example of Definition \ref{Dnf sm and isom}.
In general, Definition \ref{Dnf sm and isom} implies that at points of $S$
the space of directions of $X$ contains a euclidean unit sphere of dimension 
$\mathrm{\dim }\left( S\right) -1$ (see Proposition \ref{isom embedd prop}
below). It also implies that the intrinsic metrics on $S$ induced by $%
\mathrm{dist}^{X}$ and $g$ coincide$,$ though the converse is false. For
example, the boundary of a square with the intrinsic metric induced from $%
\mathbb{R}^{2}$ does not satisfy (\ref{bi lip dfn}). On the other hand, if $%
X $\ is the $n$--dimensional cube $\left[ 0,1\right] ^{n},$ then the open
faces of $X$ are infinitesimally isometrically embedded submanifolds.
\end{remarknonum}

From here on we identify $S$ with $\iota \left( S\right) $ and write $%
\mathrm{dist}^{X}\left( \cdot ,\cdot \right) $ for $\iota ^{\ast }\mathrm{%
dist}^{X}.$ Adopting the language of orbit spaces, we call the elements of $%
\mathcal{S}$ the \textquotedblleft strata\textquotedblright\ of $X,$ and we
call $X\setminus \left\{ \cup _{i\in I}S_{i}\right\} $ the \textquotedblleft
top strata\textquotedblright . It was shown in \cite{BGP} that for all
sufficiently small $\delta >0,$ the set, $X_{n,\delta },$ of $\left(
n,\delta \right) $-strained points is a topological manifold that is open
and dense in $X.$ In general, $X\setminus X_{n,\delta }$ can be rather wild,
so the hypothesis that the singularities occur along Riemannian manifolds is
rather special. Nevertheless this special situation occurs in all orbit
spaces, so Theorem \ref{dif stab- dim 4} has the following corollary.

\begin{corollarylet}
If $X\in \mathrm{closure}\left( \mathcal{M}_{k,v,0}^{\infty ,\infty
,D}\left( n\right) \right) $ is the quotient of an isometric group action on
a Riemannian manifold, then $X$ is diffeomorphically stable provided all of
its singular strata have codimension $\leq 3.$
\end{corollarylet}

Theorem \ref{dif stab- dim 4} generalizes Theorem 6.1 in \cite{KMS}, where
the same conclusion is obtained under the hypothesis that $\mathcal{S}%
=\emptyset .$ Theorem \ref{dif stab- dim 4} also provides an alternative
proof of the main theorems in \cite{PSW1} and \cite{PSW2}. The first author
has observed that another consequence of Theorem \ref{dif stab- dim 4} is
that Theorem 1 in \cite{Pro} holds with \textquotedblleft
homeomorphic\textquotedblright\ replaced with \textquotedblleft
diffeomorphic\textquotedblright $.$ In other words, the following is a
corollary of Theorem \ref{dif stab- dim 4} and Theorem 1 in \cite{Pro}.

\begin{theoremlet}
\label{sag thm}Let $\mathcal{S}_{k}^{n}$ be the complete, simply connected
Riemannian manifold with constant curvature $k.$ Given $k,h,r\in \mathbb{R}$
and $n\in \mathbb{N}$ with $h,r\in \left( 0,\frac{1}{2}\mathrm{diam}\mathcal{%
S}_{k}^{n}\right] $ and $h\leq r,$ there is an integer $c$ with the
following property.

If $M$ is a complete Riemannian $n$--manifold with 
\begin{eqnarray}
\mathrm{sec}M &\geq &k,  \notag \\
\mathrm{Radius}\left( M\right) &\leq &r,  \label{3 bnds inequal} \\
\mathrm{Sag}_{r}\left( M\right) &\leq &h,  \notag
\end{eqnarray}%
and almost maximal volume, then $M$ is diffeomorphic either to $\mathbb{S}%
^{n},$ to $\mathbb{R}P^{n},$ or to a Lens space $\mathbb{S}^{n}/\mathbb{Z}%
_{m}$ with $m\in \left\{ 3,4,\ldots ,c\right\} .$
\end{theoremlet}

We refer the reader to \cite{Pro} for the definition of $\mathrm{Sag}%
_{r}\left( M\right) $ and the meaning of almost maximal volume with respect
to the bounds in \autoref{3 bnds inequal}.

Here are some examples that illustrate the smoothness condition and the
possibilities for the strata inclusions in Theorem \ref{dif stab- dim 4}.

\noindent \textbf{Examples. }\emph{Let }$\mathbb{D}^{n}$\emph{\ be a disk in 
}$\mathbb{R}^{n}$\emph{\ with boundary }$\mathbb{S}^{n-1}$ \emph{and
interior }$\mathbb{B}^{n}.$\emph{\ The double of }$\mathbb{D}^{n}\times 
\mathbb{D}^{p}\times \mathbb{D}^{r}$\emph{\ satisfies the hypotheses of
Theorem \ref{dif stab- dim 4} with }%
\begin{eqnarray*}
\mathcal{N} &=&\left\{ \mathbb{S}^{n-1}\times \mathbb{S}^{p-1}\times \mathbb{%
B}^{r},S^{n-1}\times \mathbb{B}^{p}\times \mathbb{S}^{r-1},\mathbb{B}%
^{n}\times \mathbb{S}^{p-1}\times \mathbb{S}^{r-1}\right\} \text{ and} \\
\mathcal{K} &\mathcal{=}&\left\{ \mathbb{S}^{n-1}\times \mathbb{S}%
^{p-1}\times \mathbb{S}^{r-1}\right\} .
\end{eqnarray*}%
\emph{Thus the double of }$\mathbb{D}^{n}\times \mathbb{D}^{p}\times \mathbb{%
D}^{r}$\emph{\ is diffeomorphically stable. For similar reasons, the double
of }$\mathbb{D}^{n}\times \mathbb{D}^{q}$ \emph{is diffeomorphically stable.
More generally, in all the above examples, we may replace any of the disks }$%
\mathbb{D}$\emph{\ by any closed, convex subset }$K$\emph{\ of any
Riemannian manifold, provided the boundary of }$K$\emph{\ is smooth and }$%
\dim \left( K\right) =\dim \left( \mathbb{D}\right) .$

To explain our strategy to prove Theorem \ref{dif stab- dim 4}, let $\left\{
M_{a}\right\} _{\alpha }\subset \mathcal{M}_{k,v,0}^{\infty ,\infty
,D}\left( n\right) $ converge to $X,$ and let $G$ be a precompact open
subset of the top stratum, $X\setminus \left\{ \cup _{i\in I}S_{i}\right\} .$
It follows from Theorem 6.1 in \cite{KMS} that for all sufficiently large $%
\alpha ,\beta ,$ there is an open $G_{\alpha }\subset M_{\alpha }$ that is
close to $G$ and admits a smooth embedding $\Phi _{\beta ,\alpha }:G_{\alpha
}\longrightarrow M_{\beta }$ that is also a Gromov-Hausdorff approximation.
The goal is to reconstruct $\Phi _{\beta ,\alpha }$ in a manner that extends
to a diffeomorphism $M_{a}\longrightarrow M_{\beta }$ . The next two results
are the tools that allow us to do this. The first is a consequence of the
fact that the diffeomorphism group of the $n$--sphere deformation retracts
to the orthogonal group when $n=1,2,$ or $3$ (see \cite{Hat}, \cite{Sm}).

\begin{lemmalet}
\label{bundle extension}\emph{(Bundle Extension Lemma, cf Lemma 3.18 in \cite%
{GrovWilh1})} Let $\pi _{1}:E_{1}{\longrightarrow }B$ and $\pi _{2}:E_{2}{%
\longrightarrow }B$ be bundles with fiber $\mathbb{D}^{n}$ and structure
group $\mathrm{Diff}\left( \mathbb{D}^{n}\right) ,$ where $\mathbb{D}^{n}$%
{\Huge \ }is the closed disk in{\Huge \ }$\mathbb{R}^{n}.$ Let $\pi
_{1}:S\left( E_{1}\right) {\longrightarrow }B$ and $\pi _{2}:S\left(
E_{2}\right) {\longrightarrow }B$ be sphere bundles obtained by removing the
interior of each fiber from $\pi _{1}:E_{1}{\longrightarrow }B$ and $\pi
_{2}:E_{2}{\longrightarrow }B$.

If 
\begin{equation*}
\Phi :S\left( E_{1}\right) \longrightarrow S\left( E_{2}\right)
\end{equation*}%
is a diffeomorphism so that 
\begin{equation}
\pi _{1}=\pi _{2}\circ \Phi ,  \label{resp pi eqn}
\end{equation}%
then $\Phi $ extends to a diffeomorphism%
\begin{equation*}
\hat{\Phi}:E_{1}\longrightarrow E_{2}
\end{equation*}%
so that $\pi _{1}=\pi _{2}\circ \hat{\Phi}$, provided $n\leq 3.$
\end{lemmalet}

\begin{proof}
The main theorems of \cite{Hat} and \cite{Sm} give homotopy equivalences%
\begin{equation*}
\mathrm{Diff}\left( \mathbb{D}^{n}\right) \simeq O\left( n\right) \simeq 
\mathrm{Diff}\left( S^{n-1}\right)
\end{equation*}%
provided $n\leq 3$ (see also \cite{Hat 3}). Moreover, $\mathrm{Diff}\left(
S^{n-1}\right) \simeq O\left( n\right) $ is realized by a deformation
retraction.

In particular, the structure groups of the $E_{i}$ reduce to $O\left(
n\right) ,$ and therefore the $E_{i}$ admit Euclidean metrics. Using
Equation (\ref{resp pi eqn}) we view $\Phi :S\left( E_{1}\right)
\longrightarrow S\left( E_{2}\right) $ as a family of diffeomorphisms of $%
S^{n-1}$, parameterized by $B.$ Let $A\left( E_{i}\right) $ be annulus
subbundles of $E_{i}$ whose outer boundary is $S\left( E_{i}\right) .$ Use
the deformation retraction of $\mathrm{Diff}\left( S^{n-1}\right) $ to $%
O\left( n\right) $ to extend $\Phi $ to a diffeomorphism 
\begin{equation*}
\tilde{\Phi}:A\left( E_{1}\right) \longrightarrow A\left( E_{2}\right) ,
\end{equation*}%
that satisfies $\pi _{1}=\pi _{2}\circ \Phi ,$ and which is orthogonal on
the inner boundary sphere bundles of the $A\left( E_{i}\right) .$ Since
orthogonal maps extend to $\mathbb{R}^{n},$ $\tilde{\Phi}$ extends to the
desired diffeomorphism 
\begin{equation*}
\hat{\Phi}:E_{1}\longrightarrow E_{2},
\end{equation*}%
which moreover satisfies, $\pi _{1}=\pi _{2}\circ \hat{\Phi}.$
\end{proof}

There are two main difficulties with the proposal to extend $\Phi _{\beta
,\alpha }$ over successively lower dimensional strata: We do not have any
canonical tubular neighborhoods around the strata to serve as the disk
bundles of the Bundle Extension Lemma, and, even granting the existence of
these disk bundles, we do not know that $\Phi _{\beta ,\alpha }$ satisfies %
\autoref{resp pi eqn}. We resolve these problems via the next result, which
is called the Tubular Neighborhoods Stability Theorem, or TNST, for short.

Before stating the TNST, we fix some terminology to describe how strata are
related to each other. Set $\mathcal{S}^{\mathrm{ext}}\equiv \mathcal{S}\cup
\left( X\setminus \cup _{S\in \mathcal{S}}S\right) ,$ and partially order
the $S\in \mathcal{S}^{\mathrm{ext}}$ by declaring that $S<S^{\prime }$ if $%
S\subsetneq \bar{S}^{\prime }$, where $\bar{S}^{\prime }$ is the closure of $%
S^{\prime }.$ We call $a\in \mathbb{N}$ the \label{ansc numb page}\emph{%
Ancestor Number} of $S\in \mathcal{S}^{\mathrm{ext}}$ if $a$ is the length
of the largest chain 
\begin{equation*}
S_{a}<\cdots <S_{1}<S_{0}
\end{equation*}%
with $S=S_{a}$ and $S_{0}=X\setminus \left\{ \cup _{i\in I}S_{i}\right\} $
(cf \cite{SearW}). We say that $N$ is a \emph{parent} of $S$ if $S\subset 
\bar{N},$ and if $T\in \mathcal{S}^{\mathrm{ext}}$ satisfies $S\subset \bar{T%
}\subset \bar{N},$ then either $T=S$ or $T=N.$ Thus if $X$ satisfies the
hypotheses of Theorem \ref{dif stab- dim 4}, and $N\in \mathcal{N},$ then
its only parent is the top stratum. If $S\in \mathcal{K},$ then either the
top stratum is its only parent or the parents of $S$ are a finite subset of $%
\mathcal{N}$.

\bigskip

\noindent \textbf{Example. }\emph{Let }$X$\emph{\ be the double of the }$5$%
\emph{--dimensional cube }$\left( \left[ 0,1\right] ^{5}\right) _{-}\amalg
_{\partial }\left( \left[ 0,1\right] ^{5}\right) _{+}.$\emph{\ Then the
strata and their Ancestor Numbers are given by the following table. }%
\begin{equation*}
\begin{tabular}{|l|c|c|}
\hline
Submanifold & \textrm{Ancestor Number} & \textrm{Strata Type} \\ \hline
Interiors of the cubes and their $4$--dimensional faces & 0 & \textrm{Top}
\\ \hline
Interiors of the $3$--dimensional faces & 1 & $\mathcal{N}$ \\ \hline
Interiors of the $2$--dimensional faces & 2 & $\mathcal{N}$ \\ \hline
Interiors of the $1$--dimensional faces & 3 & $\mathcal{N}$ \\ \hline
Vertices & 4 & $\mathcal{K}$ \\ \hline
\end{tabular}%
\end{equation*}

\bigskip

The fact that the $4$--dimensional faces in this example are part of the top
stratum illustrates a more general phenomenon: Since $\partial X=\emptyset $%
, it follows from Corollary 12.8 of \cite{BGP} that the $\left( n-1\right) $%
--strained points of $X$ are $n$--strained. If $X$ is as in Theorem \ref{dif
stab- dim 4}, then all of the $S_{i}$s are of codimension $\leq 3.$ It
follows that the only possible Ancestor Numbers are $0,$ $1,$ and $2.$

\begin{TNST}
\emph{(TNST)}Let $X,$ $\mathcal{S}\equiv \left\{ S_{i}\right\} _{i\in I},$ $%
\mathcal{K},$ and $\mathcal{N}$ be as in Theorem \ref{dif stab- dim 4}. Let $%
\left\{ M_{\gamma }\right\} _{\gamma =1}^{\infty }\subset $ $\mathcal{M}%
_{k,v,0}^{\infty ,\infty ,D}\left( n\right) $ converge to $X.$ For all but
finitely many $\gamma \in \mathbb{N}$, $M_{\gamma }$ has a finite open cover 
$\left\{ G_{\gamma },\mathrm{interior}\left( \mathcal{U}_{\gamma
}^{S_{i}}\right) \right\} _{i\in I}$ with the following properties.

\begin{enumerate}
\item \label{T1}For $i\neq j,$ $\mathcal{U}_{\gamma }^{S_{i}}\cap \mathcal{U}%
_{\gamma }^{S_{j}}=\emptyset $ unless $S_{i}\subset \bar{S}_{j}$ or $%
S_{j}\subset \bar{S}_{i}$.

\item \label{T2}There are $C^{1}$--disk bundles 
\begin{equation*}
P_{\gamma }^{S_{i}}:\mathcal{U}_{\gamma }^{S_{i}}\longrightarrow
O_{S_{i}}\subset S_{i}
\end{equation*}%
with $O_{S_{i}}=P_{\gamma }^{S_{i}}\left( \mathcal{U}_{\gamma
}^{S_{i}}\right) $ an open subset of $S_{i}.$ Moreover, if $S_{i}\in 
\mathcal{K},$ then $O_{S_{i}}=S_{i}.$

\item \label{T3}For $t\in \left[ 1,10\right] ,$ there is an extension of $%
P_{\gamma }^{S_{i}}$ to a $1$--parameter family of disk bundles $\left( 
\mathcal{U}_{\gamma }^{S_{i}}\left( t\right) ,P_{\gamma }^{S_{i}}\right) $
so that 
\begin{equation*}
\mathcal{U}_{\gamma }^{S_{i}}=\mathcal{U}_{\gamma }^{S_{i}}\left( 1\right) ,
\end{equation*}%
and for each $x\in O_{S_{i}}$ and $1\leq s<t\leq 10,$ 
\begin{equation*}
\left( P_{\gamma }^{S_{i}}\right) ^{-1}\left( x\right) \cap \mathcal{U}%
_{\gamma }^{S_{i}}\left( s\right) \subset \mathrm{interior}\left( \left(
P_{\gamma }^{S_{i}}\right) ^{-1}\left( x\right) \cap \mathcal{U}_{\gamma
}^{S_{i}}\left( t\right) \right) .
\end{equation*}

\item \label{T4}For all but finitely many $\alpha ,\beta \in \mathbb{N},$
there is a $C^{1}$--diffeomorphism 
\begin{equation*}
\Phi _{\beta ,\alpha }:G_{\alpha }\longrightarrow \Phi _{\beta ,\alpha
}\left( G_{\alpha }\right) \subset G_{\beta }
\end{equation*}%
so that for all $S\in \mathcal{S},$ 
\begin{equation}
P_{\alpha }^{S}=P_{\beta }^{S}\circ \Phi _{\beta ,\alpha }
\label{top stratum resp eqn}
\end{equation}%
on their common domains. Moreover, for all $S\in \mathcal{S}$ with ancestor
number $1,$%
\begin{equation}
\Phi _{\beta ,\alpha }\left( \partial \mathcal{U}_{\alpha }^{S}\left(
3\right) \right) =\partial \mathcal{U}_{\beta }^{S}\left( 3\right) ,
\label{containment}
\end{equation}%
and for all $S\in \mathcal{S}$ with ancestor number $2,$%
\begin{equation}
\Phi _{\beta ,\alpha }\left( \partial \mathcal{U}_{\alpha }^{S}\left(
3\right) \cap G_{\alpha }\right) =\partial \mathcal{U}_{\beta }^{S}\left(
3\right) \cap G_{\beta }.  \label{schn 2}
\end{equation}

\item \label{inter strata resp T6}Let $S\in \mathcal{S}$ have ancestor
number 2. For each parent $N$ of $S,$ there is a neighborhood $\mathcal{V}%
^{S}$ of $S$ in $\bar{N}$ and a $C^{1}$--submersion 
\begin{equation*}
Q^{S}:\mathcal{V}^{S}\setminus S\longrightarrow S
\end{equation*}%
so that 
\begin{equation}
P_{\gamma }^{S}=Q^{S}\circ P_{\gamma }^{N}  \label{gen resp
EEEEqn}
\end{equation}%
wherever both expressions are defined. Moreover, 
\begin{equation*}
\bar{N}\subset O_{N}\cup \dbigcup\limits_{S}\mathcal{V}^{S},
\end{equation*}%
where the union is over all $S\in \mathcal{S}\setminus \left\{ N\right\} $
with $S\subset \bar{N}.$

\item \label{T6}Let $S\in \mathcal{S}$ have ancestor number $2$. Let $N\in 
\mathrm{Parents}\left( S\right) .$ For $z\in \mathcal{U}_{\gamma }^{N}\cap
\partial \mathcal{U}_{\gamma }^{S}\left( 3\right) ,$ 
\begin{equation*}
\left( P_{\gamma }^{N}\right) ^{-1}\left( P_{\gamma }^{N}\left( z\right)
\right) \subset \partial \mathcal{U}_{\gamma }^{S}\left( 3\right) .
\end{equation*}
\end{enumerate}
\end{TNST}

We prove Theorem \ref{dif stab- dim 4} by successively extending the
diffeomorphism $\Phi _{\beta ,\alpha }:G_{\alpha }\longrightarrow \Phi
_{\beta ,\alpha }\left( G_{\alpha }\right) $ of the TNST to the lower
dimensional strata. This is done by combining the Bundle Extension Lemma
with the TNST.

Let $\mathcal{A}\left( 2\right) :=\left\{ \left. S\in \mathcal{S}\text{ }%
\right\vert \text{ the ancestor number of }S\text{ is }2\right\} .$ Set 
\begin{equation*}
G_{\alpha }^{\mathrm{1}}:=M_{\gamma }\setminus \left\{ \dbigcup\limits_{S\in 
\mathcal{A}\left( 2\right) }\mathcal{U}_{\alpha }^{S}\left( 1\right)
\right\} .
\end{equation*}%
Suppose $S\in \mathcal{S}$ has ancestor number $1.$ Equation (\ref%
{containment}) implies that $\Phi _{\beta ,\alpha }$ takes the boundary
sphere bundle $\mathcal{U}_{\alpha }^{S}\left( 3\right) $ to the boundary
sphere bundle of $\mathcal{U}_{\beta }^{S}\left( 3\right) .$ Equation (\ref%
{top stratum resp eqn}) gives us Equation \autoref{resp pi eqn} with $%
P_{\alpha }^{S}$, $P_{\beta }^{S},$ and $\Phi _{\beta ,\alpha }$ playing the
roles of $\pi _{1}$, $\pi _{2},$ and $\Phi ,$ respectively. Thus by the
Bundle Extension Lemma, $\Phi _{\beta ,\alpha }$ extends to an embedding 
\begin{equation*}
\Phi _{\beta ,\alpha }^{1}:G_{\alpha }^{\mathrm{1}}\longrightarrow M_{\beta
},
\end{equation*}%
so that 
\begin{equation}
P_{\alpha }^{S}=P_{\beta }^{S}\circ \Phi _{\beta ,\alpha }^{1}
\label{resp Phi^a+1 a+1 eqn}
\end{equation}%
for all $S$ with ancestor number $1.$

To check that Equation \autoref{resp Phi^a+1 a+1 eqn} holds when the
Ancestor Number of $S$ is $2,$ suppose that $N$ is a parent of $S.$ Then
Equation \autoref{resp Phi^a+1 a+1 eqn} holds for $N,$ so%
\begin{equation*}
P_{\alpha }^{N}=P_{\beta }^{N}\circ \Phi _{\beta ,\alpha }^{1}.
\end{equation*}%
Applying $Q^{S}$ to both sides of this equation and using Part \ref{inter
strata resp T6} of the TNST we get 
\begin{equation}
P_{\alpha }^{S}=Q^{S}\circ P_{\alpha }^{N}=Q^{S}\circ P_{\beta }^{N}\circ
\Phi _{\beta ,\alpha }^{1}=P_{\beta }^{S}\circ \Phi _{\beta ,\alpha }^{1}.
\label{resp step 2 eqn}
\end{equation}%
Thus Equation \autoref{resp Phi^a+1 a+1 eqn} holds for all $S\in \mathcal{S}%
. $

The final step is to extend $\Phi _{\beta ,\alpha }^{1}$ to each $\mathcal{U}%
_{\alpha }^{S}\left( 3\right) $ for which $S\in \mathcal{A}\left( 2\right) $%
. By combining Part \ref{T1} of the TNST with the fact that $M_{\gamma
}\subset G_{\gamma }\cup \dbigcup\limits_{i\in I}\mathcal{U}_{\gamma
}^{S_{i}}\left( 1\right) ,$ we see that for such $S$, 
\begin{equation*}
\partial \mathcal{U}_{\gamma }^{S}\left( 3\right) \subset G_{\gamma
}\dbigcup\limits_{N\in \mathrm{Parents}\left( S\right) }\mathcal{U}_{\gamma
}^{N}.
\end{equation*}%
Thus 
\begin{eqnarray*}
\partial \mathcal{U}_{\gamma }^{S}\left( 3\right) &=&\left( G_{\gamma }\cap
\partial \mathcal{U}_{\gamma }^{S}\left( 3\right) \right)
\dbigcup\limits_{N\in \mathrm{Parents}\left( S\right) }\mathcal{U}_{\gamma
}^{N}\cap \partial \mathcal{U}_{\gamma }^{S}\left( 3\right) \\
&\subset &\left( G_{\gamma }\cap \partial \mathcal{U}_{\gamma }^{S}\left(
3\right) \right) \dbigcup\limits_{N\in \mathrm{Parents}\left( S\right)
}\dbigcup\limits_{z\in \mathcal{U}_{\gamma }^{N}\cap \partial \mathcal{U}%
_{\gamma }^{S}\left( 3\right) }\left( P_{\gamma }^{N}\right) ^{-1}\left(
P_{\gamma }^{N}\left( z\right) \right) .
\end{eqnarray*}%
But by Part \ref{T6} of the TNST, if $N\in \mathrm{Parents}\left( S\right) $
and $z\in \mathcal{U}_{\gamma }^{N}\cap \partial \mathcal{U}_{\gamma
}^{S}\left( 3\right) ,$ then 
\begin{equation*}
\left( P_{\gamma }^{N}\right) ^{-1}\left( P_{\gamma }^{N}\left( z\right)
\right) \subset \partial \mathcal{U}_{\gamma }^{S}\left( 3\right) ,
\end{equation*}%
so 
\begin{equation}
\partial \mathcal{U}_{\gamma }^{S}\left( 3\right) =\left( G_{\gamma }\cap
\partial \mathcal{U}_{\gamma }^{S}\left( 3\right) \right)
\dbigcup\limits_{N\in \mathrm{Parents}\left( S\right) }\dbigcup\limits_{z\in 
\mathcal{U}_{\gamma }^{N}\cap \partial \mathcal{U}_{\gamma }^{S}\left(
3\right) }\left( P_{\gamma }^{N}\right) ^{-1}\left( P_{\gamma }^{N}\left(
z\right) \right) .  \label{boundary decomp}
\end{equation}%
It follows from (\ref{schn 2}) that $\Phi _{\beta ,\alpha }^{1}\left(
\partial \mathcal{U}_{\alpha }^{S}\left( 3\right) \cap G_{\alpha }\right)
=\partial \mathcal{U}_{\beta }^{S}\left( 3\right) \cap G_{\beta }.$
Combining this with $P_{\alpha }^{N}=P_{\beta }^{N}\circ \Phi _{\beta
,\alpha }^{1}\ $and (\ref{boundary decomp}) we see that 
\begin{equation*}
\Phi _{\beta ,\alpha }^{1}\left( \partial \mathcal{U}_{\alpha }^{S}\left(
3\right) \right) =\partial \mathcal{U}_{\beta }^{S}\left( 3\right) .
\end{equation*}

So $\Phi _{\beta ,\alpha }^{1}$ takes the boundary sphere bundle of $%
\mathcal{U}_{\alpha }^{S}\left( 3\right) $ to the boundary sphere bundle of $%
\mathcal{U}_{\beta }^{S}\left( 3\right) $ while mapping fibers of $\partial 
\mathcal{U}_{\alpha }^{S}\left( 3\right) $ to fibers of $\partial \mathcal{U}%
_{\beta }^{S}\left( 3\right) .$ Via a final application of the Bundle
Extension Lemma, we extend $\Phi _{\beta ,\alpha }^{1}$ to the desired
diffeomorphism 
\begin{equation*}
\Phi _{\beta ,\alpha }^{2}:M_{\alpha }\longrightarrow M_{\beta }.
\end{equation*}%
Thus Theorem \ref{dif stab- dim 4} follows from the TNST.

The table below lists the main milestones in the remainder of the paper and
their roles in the overall proof. \bigskip

\begin{tabular}{|l|p{10.5cm}|}
\hline
\textbf{Theorem \ref{magnum cover thm}} & constructs a cover of $X$ by
strained neighborhoods on which local Alexandrov models of the disk bundles
of the TNST are defined. \\ \hline
\textbf{Theorem \ref{Perel cap thm}} & constructs local approximate versions
of the $P_{\gamma }^{S_{i}}$s$,$ $Q^{S_{j}}$s, and $\Phi _{\alpha ,\beta }$s
that satisfy local versions of Equations \autoref{top stratum resp eqn} and %
\autoref{gen resp EEEEqn}. \\ \hline
\textbf{Propositions \ref{overlaps C^1 close prop} and \ref{generic cover
cor}} & show that the local maps from Theorem \ref{Perel cap thm} are $C^{1}$%
--close on their overlaps. \\ \hline
\textbf{Corollary \ref{gluing addendum}} & allows us to define the $%
P_{\gamma }^{S_{i}}$s$,$ $Q^{S_{j}}$s, and $\Phi _{\alpha ,\beta }$s by
gluing together the local approximate versions of the $P_{\gamma }^{S_{i}}$s$%
,$ $Q^{S_{j}}$s, and $\Phi _{\alpha ,\beta }$s in a manner that preserves
Inequailities \autoref{top stratum resp eqn} and \autoref{gen resp EEEEqn}.
\\ \hline
\textbf{Proposition \ref{submersions are glued prop}} & constructs the disk
bundles of the TNST. \\ \hline
\end{tabular}%
\vspace*{0.15in}

The first subsection of Section 1 reviews basic concepts of Alexandrov
geometry, and the second subsection uses these to derive several results
that we use to prove the TNST. The rest of the paper is devoted to proving
the TNST. The main project is the construction of the disk bundles of Part %
\ref{T2} of the TNST. For this we glue together locally defined disk bundles
whose projection mappings are $C^{1}$--close. We obtain Alexandrov models of
these local disk bundles in Section \ref{Riem vs Alex}, wherein we study
isometric embeddings of Riemannian manifolds in Alexandrov spaces in greater
detail. In Section \ref{Perelman Section}, we construct the local disk
bundles by combining strainers with Perelman's concavity construction.
Section \ref{tools to arrange section} shows that the locally defined
submersions from Section \ref{Perelman Section} are $C^{1}$--close$.$

The gluing result we use, Corollary\textbf{\ }\ref{gluing addendum}, is
stated in Section \ref{Gluing Section}. Since it is similar to other results
in the literature, we defer its proof to Appendix A (Section \ref{Appendix
section}). We complete the proof of the TNST in Section \ref{TNST Section}.
For the convenience of the reader, we list notations and conventions in
Appendix B (Section \ref{Not and con}).

\begin{remarknonum}
In light of the main theorem of \cite{Hat} (cf also \cite{BalmKlin}), it is
natural to ask if the hypotheses of Theorem \ref{dif stab- dim 4} can be
weakened to allow for strata of codimension 4. In fact, an earlier version
of this paper asserted the veracity of this result.

A serious hurdle must be cleared before the ideas in this paper can prove
the stronger result: the homotopy type of $\mathrm{Diff}\left( \mathbb{D}%
^{4}\right) $ is not known{\large . }Thus the reduction of the structure
groups of the disk bundles in the Bundle Extension Lemma is not possible
through any topological arguments that we are aware of.{\large \ }

A potential path around this difficulty would be to reduce the structure
groups of the disk bundles of Part \ref{T2} of the TNST to orthogonal groups
through geometric means. In particular, it follows from the TNST that $X\in 
\mathrm{closure}\left( \mathcal{M}_{k,v,0}^{\infty ,\infty ,D}\left(
n\right) \right) $ is diffeomorphically stable, provided its singularities
occur along smooth Riemannian manifolds of codimension $\leq 4$ and the
codimension $\ 4$ disk bundles of the TNST are trivial.

We are profoundly grateful to a referee for pointing out a gap in an earlier
draft that is related to this remark.
\end{remarknonum}

\begin{remarknonum}
While there are many examples of Alexandrov spaces whose singularities occur
along Riemannian manifolds as in Theorem \ref{dif stab- dim 4}, we imagine
that a generic Alexandrov space does not satisfy this condition. In fact,
there are constructions of Alexandrov spaces in \cite{Li} and \cite{LiNab}
whose singularities occur along cantor sets.
\end{remarknonum}

\begin{acknowledgment}
We are grateful to Paula Bergen for copy editing the manuscript, to Vitali
Kapovitch for several extensive conversations about this problem over the
years, to Catherine Searle and Maree Jaramillo for comments on the
manuscript, to Jim Kelliher for several discussions relevant to the proof of
Theorem \ref{submersion gluing}, to a referee of \cite{PSW1} for proposing a
form of the Tubular Neighborhood Stability Theorem, to Julie Bergner and
Pedro Sol\'{o}rzano for discussions on the classification of vector bundles,
and to Michael Sill and Nan Li for multiple discussions on and valuable
criticisms of this manuscript. Special thanks go to Notre Dame for hosting a
stay by the second author during which this work was completed. We are
profoundly grateful to the referees for valuable mathematical and expository
criticisms.
\end{acknowledgment}

\section{Basic Tools of Alexandrov Geometry}

The notion, from \cite{BGP}, of strainers in an Alexandrov space forms the
core of the calculus arguments we use. In the next subsection, we review
this notion and its relevant consequences. The exposition borrows freely
from \cite{PSW1} and \cite{PSW2}.

\addtocounter{algorithm}{1}

\subsection{Strainers and their Consequences{\protect\Huge \ }}

\begin{definition}
\label{strainer}Let $X$ be an Alexandrov space. A point $x\in X$ is said to
be $\left( n,\delta ,r\right) $--strained by the strainer $\left\{ \left(
a_{i},b_{i}\right) \right\} _{i=1}^{n}\subset X\times X$ provided that for
all $i\neq j$ we have%
\begin{equation*}
\begin{array}{ll}
\tilde{\sphericalangle}\left( a_{i},x,b_{i}\right) >\pi -\delta , & \tilde{%
\sphericalangle}\left( a_{i},x,b_{j}\right) >\frac{\pi }{2}-\delta , \\ 
\tilde{\sphericalangle}\left( b_{i},x,b_{j}\right) >\frac{\pi }{2}-\delta ,
& \tilde{\sphericalangle}\left( a_{i},x,a_{j}\right) >\frac{\pi }{2}-\delta ,%
\text{ and} \\ 
\multicolumn{2}{c}{\min_{i=1,\ldots ,n}\left\{ \mathrm{dist}%
(\{a_{i},b_{i}\},x)\right\} >r.}%
\end{array}%
\end{equation*}

We say $B\subset X$ is an $(n,\delta ,r)$--strained set with strainer $%
\{a_{i},b_{i}\}_{i=1}^{n}$ provided every point $x\in B$ is $(n,\delta ,r)$%
--strained by $\{a_{i},b_{i}\}_{i=1}^{n}$. When there is no need to specify, 
$r$ we say that $x$ is $\left( n,\delta \right) $--strained.
\end{definition}

Next we state a powerful lemma from \cite{BGP} which shows that for a $%
(1,\delta ,r)$ strained neighborhood, angle and comparison angle almost
coincide for geodesic hinges with one side in this neighborhood and the
other reaching a strainer.

\begin{lemma}
\label{5.6}(\cite{BGP}, Lemma $5.6$) Let $B\subset X$ be $\left( 1,\delta
,r\right) $--strained by $(a,b).$ For any $x,z\in B$,%
\begin{eqnarray}
\left\vert \sphericalangle \left( a,x,z\right) -\tilde{\sphericalangle}%
\left( a,x,z\right) \right\vert &<&\tau \left( \delta \right) +\tau \left( 
\mathrm{dist}\left( x,z\right) |r\right) ,\text{ and}  \notag \\
\left\vert \sphericalangle \left( b,x,z\right) -\tilde{\sphericalangle}%
\left( b,x,z\right) \right\vert &<&\tau \left( \delta \right) +\tau \left( 
\mathrm{dist}\left( x,z\right) |r\right) .  \label{5.6 baby}
\end{eqnarray}%
In addition, 
\begin{equation}
\left\vert \tilde{\sphericalangle}\left( a,x,z\right) +\tilde{\sphericalangle%
}\left( b,x,z\right) -\pi \right\vert <\tau \left( \delta \right) +\tau
\left( \mathrm{dist}\left( x,z\right) |r\right) .  \label{5.6 proof}
\end{equation}
\end{lemma}

The importance of the previous result cannot be overstated. As we will see
next, Lemma $5.6$ of \cite{BGP} gives us two-sided bounds for both the angle
and the comparison angle of a strained point to its strainer. The tremendous
synergy this creates is due to the fact that comparison angles are
continuous and angles determine derivatives of distance functions.

\begin{lemma}
\label{angle convergence}Let $B\subset X$ be $\left( l,\delta ,r\right) $%
--strained by $\left\{ \left( a_{i},b_{i}\right) \right\} _{i=1}^{l}.$ For
any $x\in B$ and $i\neq j,$ 
\begin{equation*}
\begin{array}{ll}
\pi -\delta <\tilde{\sphericalangle}\left( a_{i},x,b_{i}\right) \leq \pi , & 
\frac{\pi }{2}-\delta <\tilde{\sphericalangle}\left( a_{i},x,b_{j}\right) <%
\frac{\pi }{2}+\tau \left( \delta \right) , \\ 
\frac{\pi }{2}-\delta <\tilde{\sphericalangle}\left( b_{i},x,b_{j}\right) <%
\frac{\pi }{2}+\tau \left( \delta \right) , & \frac{\pi }{2}-\delta <\tilde{%
\sphericalangle}\left( a_{i},x,a_{j}\right) <\frac{\pi }{2}+\tau \left(
\delta \right) ,%
\end{array}%
\end{equation*}%
\begin{equation*}
\begin{array}{ll}
\pi -\delta <\sphericalangle \left( a_{i},x,b_{i}\right) \leq \pi , & \frac{%
\pi }{2}-\delta <\sphericalangle \left( a_{i},x,b_{j}\right) <\frac{\pi }{2}%
+\tau \left( \delta \right) ,\text{ and} \\ 
\frac{\pi }{2}-\delta <\sphericalangle \left( b_{i},x,b_{j}\right) <\frac{%
\pi }{2}+\tau \left( \delta \right) , & \frac{\pi }{2}-\delta
<\sphericalangle \left( a_{i},x,a_{j}\right) <\frac{\pi }{2}+\tau \left(
\delta \right) .%
\end{array}%
\end{equation*}
\end{lemma}

\begin{proof}
Since angles are bigger than comparison angles, it follows from the
definition of strainer that we need only prove the last three angle upper
bounds.

Since angles are limits of comparison angles, our lower curvature bound
gives us that%
\begin{equation*}
\sphericalangle \left( a_{i},x,b_{i}\right) +\sphericalangle \left(
b_{i},x,b_{j}\right) +\sphericalangle \left( b_{j},x,a_{i}\right) \leq 2\pi
\end{equation*}%
(see \cite{BGP}, 2.3(D)). Since angles are bigger than comparison angles,
the definition of strainer gives%
\begin{equation*}
\frac{\pi }{2}-\delta <\sphericalangle \left( b_{j},x,a_{i}\right) ,\frac{%
\pi }{2}-\delta <\sphericalangle \left( b_{i},x,b_{j}\right) ,\text{ and }%
\pi -\delta <\sphericalangle \left( a_{i},x,b_{i}\right) .
\end{equation*}%
Together, the previous two displays give 
\begin{equation*}
\sphericalangle \left( b_{j},x,a_{i}\right) \leq \frac{\pi }{2}+\tau \left(
\delta \right) \text{ and }\sphericalangle \left( b_{i},x,b_{j}\right) \leq 
\frac{\pi }{2}+\tau \left( \delta \right) ,
\end{equation*}%
and by a similar argument, $\sphericalangle \left( a_{i},x,a_{j}\right) <%
\frac{\pi }{2}+\tau \left( \delta \right) .$
\end{proof}

\begin{proposition}
\label{angggl conv prop}Suppose $\left\{ M_{\alpha }\right\} _{\alpha }$ is
a sequence of $n$--dimensional Alexandrov spaces with curvature $\geq k$
that converge in the Gromov-Hausdorff topology to $X.$ Suppose $\left\{
\left( a_{i},b_{i}\right) \right\} _{i=1}^{l}$ is an $\left( l,\delta
,r\right) $--strainer for $y\in X.$ Let $\left\{ \left( a_{i}^{\alpha
},b_{i}^{\alpha }\right) \right\} _{i=1}^{l}\subset M_{\alpha }\times
M_{\alpha }$ converge to $\left\{ \left( a_{i},b_{i}\right) \right\}
_{i=1}^{l},$ and let $c^{\alpha }\in M_{\alpha }$ converge to $c\in X.$

Then for $y^{\alpha }\in M^{\alpha }$ with $y^{\alpha }\rightarrow y,$ 
\begin{equation*}
\left\vert \sphericalangle \left( \Uparrow _{y^{\alpha }}^{a_{i}^{\alpha
}},\Uparrow _{y^{\alpha }}^{c^{\alpha }}\right) -\sphericalangle \left(
\Uparrow _{y}^{a_{i}},\Uparrow _{y}^{c}\right) \right\vert <\tau \left( 
\frac{1}{\alpha }|r\right) +\tau \left( \delta \right) .
\end{equation*}
\end{proposition}

\begin{proof}
In general, semi-continuity of angles gives 
\begin{eqnarray}
\lim_{\alpha \rightarrow \infty }\inf \sphericalangle \left( \Uparrow
_{y^{\alpha }}^{a_{i}^{\alpha }},\Uparrow _{y^{\alpha }}^{c^{\alpha
}}\right) &\geq &\sphericalangle \left( \Uparrow _{y}^{a_{i}},\Uparrow
_{y}^{c}\right) \text{ and }  \notag \\
\lim_{\alpha \rightarrow \infty }\inf \sphericalangle \left( \Uparrow
_{y^{\alpha }}^{b_{i}^{\alpha }},\Uparrow _{y^{\alpha }}^{c^{\alpha
}}\right) &\geq &\sphericalangle \left( \Uparrow _{y}^{b_{i}},\Uparrow
_{y}^{c}\right) .  \label{semi II}
\end{eqnarray}%
Since $\left\{ \left( a_{i}^{\alpha },b_{i}^{\alpha }\right) \right\}
_{i=1}^{l}$ and $\left\{ \left( a_{i},b_{i}\right) \right\} _{i=1}^{l}$ are
strainers, 
\begin{eqnarray}
\pi -\delta &<&\sphericalangle \left( \Uparrow _{y^{\alpha }}^{a_{i}^{\alpha
}},\Uparrow _{y^{\alpha }}^{b_{i}^{\alpha }}\right) \leq \sphericalangle
\left( \Uparrow _{y^{\alpha }}^{a_{i}^{\alpha }},\Uparrow _{y^{\alpha
}}^{c^{\alpha }}\right) +\sphericalangle \left( \Uparrow _{y^{\alpha
}}^{c^{\alpha }},\Uparrow _{y^{\alpha }}^{b_{i}^{\alpha }}\right) <\pi +\tau
\left( \delta \right) +\tau \left( \frac{1}{\alpha }|r\right) ,\text{ and} 
\notag \\
\pi -\delta &<&\sphericalangle \left( \Uparrow _{y}^{a_{i}},\Uparrow
_{y}^{b_{i}}\right) \leq \sphericalangle \left( \Uparrow
_{y}^{a_{i}},\Uparrow _{y}^{c}\right) +\sphericalangle \left( \Uparrow
_{y}^{c},\Uparrow _{y}^{b_{i}}\right) <\pi +\tau \left( \delta \right) ,
\label{5.6 MMMooreee}
\end{eqnarray}%
where the last upper bound on each line comes from Inequality \autoref{5.6
proof} and the fact that angles are limits of comparison angles.

Combining Inequalities \autoref{semi II} and \autoref{5.6 MMMooreee}, 
\begin{equation*}
\left\vert \sphericalangle \left( \Uparrow _{y^{\alpha }}^{a_{i}^{\alpha
}},\Uparrow _{y^{\alpha }}^{c^{\alpha }}\right) -\sphericalangle \left(
\Uparrow _{y}^{a_{i}},\Uparrow _{y}^{c}\right) \right\vert <\tau \left( 
\frac{1}{\alpha }|r\right) +\tau \left( \delta \right) .
\end{equation*}
\end{proof}

If $x$ is $\left( l,\delta ,r\right) $--strained by $\left\{ \left(
a_{i},b_{i}\right) \right\} _{i=1}^{l},$ we get an analogy with linear
algebra by thinking of $\left\{ \Uparrow _{x}^{a_{i}}\right\} _{i=1}^{l}$ as
an almost orthonormal subset in $\Sigma _{x}.$ This leads to

\begin{proposition}
\label{sim str cor}Suppose that $x\in X$ is $\left( l,\delta \right) $%
--strained by $\left\{ \left( a_{i},b_{i}\right) \right\} _{i=1}^{l}$ and $%
\left\{ \left( c_{i},d_{i}\right) \right\} _{i=1}^{l},$ and that $\tilde{x}%
\in \tilde{X}$ is $\left( l,\delta \right) $--strained by $\left\{ \left( 
\tilde{a}_{i},\tilde{b}_{i}\right) \right\} _{i=1}^{l}$ and $\left\{ \left( 
\tilde{c}_{i},\tilde{d}_{i}\right) \right\} _{i=1}^{l}.$ In addition,
suppose both sets of strainers \textquotedblleft almost span the same
subspace\textquotedblright ,\ in the sense that 
\begin{equation}
\left\vert \left\vert \det \left( \cos \sphericalangle \left( \Uparrow
_{x}^{a_{i}},\Uparrow _{x}^{c_{j}}\right) \right) _{i,j}\right\vert
-1\right\vert <\tau \left( \delta \right)  \label{area proj 1}
\end{equation}%
and 
\begin{equation}
\left\vert \left\vert \det \left( \cos \sphericalangle \left( \Uparrow _{%
\tilde{x}}^{\tilde{a}_{i}},\Uparrow _{\tilde{x}}^{\tilde{c}_{j}}\right)
\right) _{i,j}\right\vert -1\right\vert <\tau \left( \delta \right) .
\label{area proj 2}
\end{equation}%
Suppose further that in each space we have \textquotedblleft almost the same
change of basis matrix\textquotedblright ,\ in the sense that for all $i,j$
and for some $\varepsilon >0,$ 
\begin{equation}
\left\vert \sphericalangle \left( \Uparrow _{x}^{a_{i}},\Uparrow
_{x}^{c_{j}}\right) -\sphericalangle \left( \Uparrow _{\tilde{x}}^{\tilde{a}%
_{i}},\Uparrow _{\tilde{x}}^{\tilde{c}_{j}}\right) \right\vert <\varepsilon .
\label{ch bas ineq}
\end{equation}

Then given $W\in \Sigma _{x}\left( X\right) $ with 
\begin{equation}
\left\vert \dsum\limits_{i=1}^{l}\cos \sphericalangle \left( W,\Uparrow
_{x}^{a_{i}}\right) -1\right\vert <\tau \left( \delta \right) ,
\label{proj
of Y}
\end{equation}%
there is a $\tilde{W}\in \Sigma _{\tilde{x}}\left( \tilde{X}\right) $ so
that for all $i,$ 
\begin{equation}
\left\vert \sphericalangle \left( W,\Uparrow _{x}^{a_{i}}\right)
-\sphericalangle \left( \tilde{W},\Uparrow _{\tilde{x}}^{\tilde{a}%
_{i}}\right) \right\vert <\tau \left( \delta \right) +\tau \left(
\varepsilon \right)  \label{alm same coeffs
1}
\end{equation}%
and 
\begin{equation}
\left\vert \sphericalangle \left( W,\Uparrow _{x}^{c_{i}}\right)
-\sphericalangle \left( \tilde{W},\Uparrow _{\tilde{x}}^{\tilde{c}%
_{i}}\right) \right\vert <\tau \left( \delta \right) +\tau \left(
\varepsilon \right) .  \label{alm same coeff
2}
\end{equation}
\end{proposition}

\begin{proof}
When $\delta =0,$ the statement can be interpreted as a linear algebra fact.
Indeed, if $\delta =0,$ then $\left\{ \Uparrow _{x}^{a_{i}}\right\}
_{i=1}^{l}$ and $\left\{ \Uparrow _{x}^{c_{j}}\right\} _{i=1}^{l}$ lie in
subsets $V_{a}$ and $V_{c}$ of $T_{x}X$ that are isometric to $\mathbb{R}%
^{l},$ in which $\left\{ \Uparrow _{x}^{a_{i}}\right\} _{i=1}^{l}$ and $%
\left\{ \Uparrow _{x}^{c_{j}}\right\} _{i=1}^{l}$ are orthonormal bases.
Inequality \autoref{area proj 1} with $\delta =0,$ implies that $V_{a}$ and $%
V_{c}$ are the same, since the projection $V_{a}$ onto $V_{c}$ carries the
cube spanned by $\left\{ \Uparrow _{x}^{a_{i}}\right\} _{i=1}^{l}$ to a
parallelepiped of volume $1.$ Using Inequality \autoref{area proj 2}, the
analogous statement applies to $\left\{ \Uparrow _{\tilde{x}}^{\tilde{a}%
_{i}}\right\} _{i=1}^{l}$ and $\left\{ \Uparrow _{\tilde{x}}^{\tilde{c}%
_{j}}\right\} _{i=1}^{l}.$

Inequality \autoref{proj of Y} with $\delta =0$ implies that $W$ is in the
span of $\left\{ \Uparrow _{x}^{a_{i}}\right\} _{i=1}^{l}$. Given such a $W,$
there is a $\tilde{W}$ whose coefficients as a combination of $\left\{
\Uparrow _{\tilde{x}}^{\tilde{a}_{i}}\right\} _{i=1}^{l}$ are the same as
those of $W$ as a combination of $\left\{ \Uparrow _{x}^{a_{i}}\right\}
_{i=1}^{l}.$ That is, we get Inequality \autoref{alm same coeffs 1} when $%
\delta =\varepsilon =0.$ Inequality \autoref{ch bas ineq} with $\varepsilon
=0$ implies that the change of basis matrix that carries $\left\{ \Uparrow
_{x}^{a_{i}}\right\} _{i=1}^{l}$ to $\left\{ \Uparrow _{x}^{c_{j}}\right\}
_{i=1}^{l}$ also carries $\left\{ \Uparrow _{\tilde{x}}^{\tilde{a}%
_{i}}\right\} _{i=1}^{l}$ to $\left\{ \Uparrow _{\tilde{x}}^{\tilde{c}%
_{j}}\right\} _{i=1}^{l}.$ Thus Inequality \autoref{alm same coeff 2} with $%
\delta =\varepsilon =0$ follows from the $\delta =\varepsilon =0$ versions
of Inequalities \autoref{ch bas ineq} and \autoref{alm same coeffs 1}. By
continuity, we get the result for all sufficiently small positive $%
\varepsilon $ and $\delta .$
\end{proof}

\addtocounter{algorithm}{1}

\subsection{Spherical Sets and the Join Lemma}

When $x$ is $k$--strained, $\Sigma _{x}$ is Gromov--Hausdorff close to a
space of $\mathrm{curv}\geq 1$ that contains a metrically embedded copy of $%
\mathbb{S}^{k-1}$. The sense in which this embedding preserves metrics\ is
much stronger than for the infinitesimally isometric embeddings of
Definition \ref{Dnf sm and isom}. Specifically,

\begin{definition}
We say that an embedding $\iota :Y\hookrightarrow X$ of a metric space $Y$
into a metric space $X$ is \emph{metric} if and only if%
\begin{equation*}
\mathrm{dist}_{Y}\left( y_{1},y_{2}\right) =\mathrm{dist}_{X}\left( \iota
\left( y_{1}\right) ,\iota \left( y_{2}\right) \right) .
\end{equation*}
\end{definition}

The model space of directions for a point that is $\left( m+1\right) $%
--strained is given by the Join Lemma, which follows.

\begin{lemma}
(Join Lemma, \cite{GrovWilh1})\label{Join Lemma} Let $X$ be an $n$%
--dimensional Alexandrov space with $curv$ $\geq 1.$ If $X$ contains a
metrically embedded copy of the unit $m$--sphere, $\mathbb{S}^{m},$ then $%
E\equiv \left\{ x\in X|\mathrm{dist}\left( \mathbb{S}^{m},x\right) =\frac{%
\pi }{2}\right\} $ is a metrically embedded $\left( n-m-1\right) $%
--dimensional Alexandrov space with $\mathrm{curv}E\geq 1,$ and $X$ is
isometric to the spherical join $\mathbb{S}^{m}\ast E.$
\end{lemma}

See \cite{GrovMark} for the definition of spherical join metrics.

\begin{definition}
As in \cite{BGP} and \cite{Yam2} we say an Alexandrov space $\Sigma $ with $%
\mathrm{curv\,}\Sigma \geq 1$ is globally $(m,\delta )$-strained by pairs of
subsets $\{A_{i},B_{i}\}_{i=1}^{m}$ provided 
\begin{equation*}
\begin{array}{ll}
|\mathrm{dist}(a_{i},b_{j})-\frac{\pi }{2}|<\delta , & \mathrm{dist}%
(a_{i},b_{i})>\pi -\delta , \\ 
|\mathrm{dist}(a_{i},a_{j})-\frac{\pi }{2}|<\delta , & |\mathrm{dist}%
(b_{i},b_{j})-\frac{\pi }{2}|<\delta%
\end{array}%
\end{equation*}%
for all $a_{i}\in A_{i}$ and $b_{i}\in B_{i}$ with $i\neq j$.
\end{definition}

We also consider a generalization of global strainers due to Plaut.

\begin{definition}
(Plaut, \cite{Pla}) A set of $2n$ points $x_{1},y_{1},\ldots ,x_{n},y_{n}$
in a metric space $Y$ is called spherical if $\mathrm{dist}(x_{i},y_{i})=\pi 
$ for all $i$ and $\mathrm{det}[\cos \mathrm{dist}(x_{i},x_{j})]>0$.
\end{definition}

\begin{remarknonum}
If $x_{1},\ldots ,x_{n}$ are points in $\mathbb{S}^{n+k}\subset \mathbb{R}%
^{n+k+1},$ then $\sqrt{\mathrm{det}[\cos \mathrm{dist}(x_{i},x_{j})]}$ is
the $n$--dimensional volume of the parallelepiped spanned by $\left\{
x_{1},\ldots ,x_{n}\right\} .$ So Plaut's condition should be viewed as a
quantification of linear independence.
\end{remarknonum}

\begin{theorem}
\label{Plaut}(Plaut, \cite{Pla}) If $X$ has curvature $\geq 1$ and contains
a spherical set $\Sigma $ of $2(n+1)$ points, then there is a subset $S$ of $%
X$ isometric to $\mathbb{S}^{n}$ such that $\Sigma \subset S$.
\end{theorem}

The following is a natural deformation of Plaut's condition.

\begin{definition}
A set of $2n$ points $x_{1},y_{1},\ldots ,x_{n},y_{n}$ in a metric space $Y$
is called $\left( \delta |d\right) $--almost spherical if $\mathrm{dist}%
(x_{i},y_{i})>\pi -\delta $ for all $i$ and $\mathrm{det}[\cos \mathrm{dist}%
(x_{i},x_{j})]>d>0$.
\end{definition}

Plaut's notion of spherical sets is related to strainers via the following
result.

\begin{proposition}
\label{almost Plauter}Let $X$ have curvature $\geq 1,$ dimension $n,$ and
contain a $\left( \delta |d\right) $--almost spherical set $S$ of $2(m+1)$
points, for $m<n-1.$

There is an $\left( m+1,\tau \left( \delta |d\right) \right) $--global
strainer $\left\{ \left( a_{i},b_{i}\right) \right\} _{i=1}^{m+1}$ for $X$
so that 
\begin{equation*}
\mathrm{dist}\left( a_{i},a_{j}\right) >\frac{\pi }{2}\text{ for }i\neq j.
\end{equation*}%
Moreover, for all $\kappa \in \left( 0,\frac{\pi }{4}\right) ,$ if $\delta $
is sufficiently small compared to $d$ and $\kappa ,$ there is a nonempty set 
$E\subset X$ so that for all $e\in E$ 
\begin{equation*}
\frac{\pi }{2}<\mathrm{dist}\left( e,a_{i}\right) <\frac{\pi }{2}+\kappa ,
\end{equation*}%
and 
\begin{equation*}
\left\vert \mathrm{dist}\left( e,b_{i}\right) -\frac{\pi }{2}\right\vert
<\kappa .
\end{equation*}
\end{proposition}

\begin{proof}
First we consider the rigid case when $X$ contains an isometric copy of $%
\mathbb{S}^{m}.$ Perturbing an orthonormal basis, one sees that $X$ contains
a global $\left( m+1,\delta \right) $--strainer $\left\{ \left(
a_{i},b_{i}\right) \right\} _{i=1}^{m+1}\subset \mathbb{S}^{m}$ so that 
\begin{equation*}
\mathrm{dist}\left( a_{i},a_{j}\right) >\frac{\pi }{2}\text{ for }i\neq j.
\end{equation*}%
We can also find a point $h\in \mathbb{S}^{m}$ with 
\begin{equation*}
\mathrm{dist}\left( a_{i},h\right) >\frac{\pi }{2}\text{ for all }i.
\end{equation*}%
By the Join Lemma, $\tilde{E}\equiv \left\{ x\in X|\mathrm{dist}\left( 
\mathbb{S}^{m},x\right) =\frac{\pi }{2}\right\} $ is a metrically embedded $%
\left( n-m-1\right) $--dimensional Alexandrov space with $\mathrm{curv}%
\tilde{E}\geq 1,$ and $X$ is isometric to the join $\mathbb{S}^{m}\ast 
\tilde{E}.$

Combining this with $\mathrm{dist}\left( a_{i},h\right) >\frac{\pi }{2}$, it
follows that for all $\tilde{e}\in \tilde{E},$ the interior of the segment $%
\tilde{e}h$ is further than $\frac{\pi }{2}$ from all the points $a_{i}.$
For any fixed $\kappa \in \left( 0,\frac{\pi }{4}\right) ,$ we set%
\begin{equation*}
E=\left\{ \left. \tilde{e}h\left( \frac{\kappa }{2}\right) \right\vert 
\tilde{e}\in \tilde{E}\right\} .
\end{equation*}%
This completes the proof in the rigid case. The general case follows from
the rigid case, Theorem \ref{Plaut}, Lemma \ref{Join Lemma}, and a proof by
contradiction.
\end{proof}

\addtocounter{algorithm}{1}

\subsection{Gromov Packing}

We make use a version of Gromov's Packing Lemma. Its closest relative in the
literature, as far as we know, is on page 230 of \cite{Zhu}. Before stating
it we make the following definition.

\begin{definition}
We say that a collection of sets $\mathcal{C}$ has first order $\leq 
\mathfrak{o}$ if and only if each $C\in \mathcal{C}$ intersects no more than 
$\mathfrak{o}-1$ other members of $\mathcal{C}.$
\end{definition}

\begin{lemma}
\label{Gromov Pack}(Gromov's Packing Lemma) Let $X$ be an $n$--dimensional
Alexandrov space with curvature $\geq k$ for some $k\in \mathbb{R}.$ There
are positive constants $\mathfrak{o}\left( n,k\right) $ and $r_{0}\left(
n,k\right) $ with the following property. For all $r\in \left(
0,r_{0}\right) ,$ any compact subset of $A\subset X$ contains a finite
subset $\left\{ a_{i}\right\} _{i\in I}$ so that

\begin{itemize}
\item $A\subset \cup _{i}B\left( a_{i},r\right) ,$ and

\item the first order of the cover $\left\{ B\left( a_{i},3r\right) \right\}
_{i}$ is $\leq \mathfrak{o}.$
\end{itemize}
\end{lemma}

In the Riemannian case, this follows from relative volume comparison, so one
only needs the corresponding lower bound on Ricci curvature. Since relative
volume comparison holds for rough volume in Alexandrov spaces, the proof in 
\cite{Zhu} yields, with minor modifications, Lemma \ref{Gromov Pack}.

\section{Riemannian Submanifolds of Alexandrov Spaces\label{Riem vs Alex}}

Here we establish several results that are relevant to infinitesimally
isometric embeddings of Riemannian manifolds into Alexandrov spaces. In the
first subsection, we show that the unit tangent sphere of each point $p\in S$
metrically embeds into the space of directions of $p$ in $X.$ In the second
subsection, we prove Theorem \ref{magnum cover thm}, which gives local
Alexandrov models of the vector bundles of the TNST.

\addtocounter{algorithm}{1}

\subsection{Riemannian versus Alexandrov Spaces of Directions}

\begin{definition}
\emph{(\cite{BGP}, page 48)} Let $c:\left[ -a,a\right] \longrightarrow 
\mathbb{R}$ be a unit speed curve in an Alexandrov space $X.$ The right and
left derivatives of $c$ at $0$ are%
\begin{equation*}
c_{+}^{\prime }\left( 0\right) \equiv \lim_{t\rightarrow 0^{+}}\Uparrow
_{c\left( 0\right) }^{c\left( t\right) }\text{ and }c_{-}^{\prime }\left(
0\right) \equiv \lim_{t\rightarrow 0^{-}}\Uparrow _{c\left( 0\right)
}^{c\left( t\right) },
\end{equation*}%
provided the limits exist and are single directions.
\end{definition}

\begin{proposition}
\label{isom embedd prop}Let $S$ be a Riemannian manifold that is
infinitesimally isometrically embedded in an Alexandrov space $X.$ For $p\in
S,$ let $T_{p}^{1}S$ be the Riemannian unit tangent sphere to $S$ at $p,$
and for $v\in T_{p}^{1}S,$ let $c_{v}\left( t\right) =\exp _{p}^{S}\left(
tv\right) .$ Then the map 
\begin{eqnarray*}
\iota &:&T_{p}^{1}S\longrightarrow \Sigma _{p}X \\
\iota &:&v\mapsto \left( c_{v}\right) _{+}^{\prime }\left( 0\right)
\end{eqnarray*}%
is a well-defined metric embedding. In particular, $\left( c_{v}\right)
_{+}^{\prime }\left( 0\right) $ exists, and for every geodesic $c$ of $S,$ 
\begin{equation*}
\sphericalangle _{X}\left( c_{+}^{\prime }\left( 0\right) ,c_{-}^{\prime
}\left( 0\right) \right) =\pi .
\end{equation*}
\end{proposition}

\begin{proof}
Let $\left\{ e_{i}\right\} _{i=1}^{\mathrm{\dim }\left( S\right) }\subset
T_{p}S$ be an orthonormal basis. Then 
\begin{equation}
\left\{ \left( c_{e_{i}}\left( r\right) ,\text{ }c_{e_{i}}\left( -r\right)
\right) \right\} _{i=1}^{\mathrm{\dim }\left( S\right) }  \label{S strainer}
\end{equation}%
is a $\left( \mathrm{\dim }\left( S\right) ,\tau \left( r\right) ,r\right) $%
--strainer for $S$ at $p,$ and Definition \ref{Dnf sm and isom} gives us
that for all $v,w\in T_{p}^{1}S,$ 
\begin{equation}
\left\vert \tilde{\sphericalangle}_{S}\left( c_{v}\left( s\right)
,p,c_{w}\left( t\right) \right) -\tilde{\sphericalangle}_{X}\left(
c_{v}\left( s\right) ,p,c_{w}\left( t\right) \right) \right\vert <\tau
\left( s,t\right) .  \label{simialr triangles}
\end{equation}

Thus 
\begin{equation}
\left\{ \left( c_{e_{i}}\left( r\right) ,\text{ }c_{e_{i}}\left( -r\right)
\right) \right\} _{i=1}^{\mathrm{\dim }\left( S\right) }
\end{equation}%
is a $\left( \mathrm{\dim }\left( S\right) ,\tau \left( r\right) ,r\right) $%
--strainer for $X$ at $p.$

Let $\left\{ s_{k}\right\} _{k=1}^{\infty }\subset \left( 0,r\right) $
converge to $0.$ Since angles are larger than comparison angles, 
\begin{equation}
\sphericalangle \left( \left( \Uparrow _{p}^{c_{e_{i}}\left( s_{k}\right)
}\right) _{X},\left( \Uparrow _{p}^{c_{e_{i}}\left( -r\right) }\right)
_{X}\right) \geq \tilde{\sphericalangle}_{X}\left( c_{e_{i}}\left(
s_{k}\right) ,p,c_{e_{i}}\left( -r\right) \right) >\pi -\tau \left(
s_{k},r\right) .  \label{alm antipodal}
\end{equation}

Since $\left\{ \Uparrow _{p}^{c_{e_{i}}\left( s_{k}\right) }\right\}
_{k=1}^{\infty }$ is a sequence of compact subsets of the compact metric
space, $\Sigma _{p}X,$ it has a convergent subsequence. Let $\Uparrow \left(
e_{i}\left( r\right) \right) $ be a limit of such a subsequence. Since 
\textrm{curv}$\left( \Sigma _{p}X\right) \geq 1,$ Inequality (\ref{alm
antipodal}) implies that there is a unique direction $\uparrow _{A\left(
-e_{i}\right) _{r}}$ at maximal distance from $\left( \Uparrow
_{p}^{c_{e_{i}}\left( -r\right) }\right) _{X}.$ What's more, all of the
possible sets $\Uparrow \left( e_{i}\left( r\right) \right) $ lie in the $%
\tau \left( r\right) $--ball around $\uparrow _{A\left( -e_{i}\right) _{r}}.$
Now choose a sequence $r_{k}\rightarrow 0$ so that $\left( \Uparrow
_{p}^{c_{e_{i}}\left( -r_{k}\right) }\right) _{X}$ converges. Then $\left\{
\uparrow _{A\left( -e_{i}\right) _{r_{k}}}\right\} _{k}$ also converges, and 
$\Uparrow \left( e_{i}\left( r_{k}\right) \right) $ converges to a point.

Thus each intrinsic geodesic, $c,$ of $S$ has both a right and left
derivative, $c_{+}^{\prime }\left( 0\right) $ and $c_{-}^{\prime }\left(
0\right) ,$ inside of $X.$ In particular, our map 
\begin{eqnarray*}
\iota &:&T_{p}^{1}S\longrightarrow \Sigma _{p}X \\
\iota &:&v\mapsto \left( c_{v}\right) _{+}^{\prime }\left( 0\right)
\end{eqnarray*}%
is well-defined.

To see that $\iota $ is metric, take $v_{1},w\in T_{p}^{1}S$. Extend $v_{1}$
to an orthonormal basis $\left\{ v_{i}\right\} _{i=1}^{\mathrm{\dim }\left(
S\right) }$ for $T_{p}S$. As above, we have that for $r>0,$ $\left\{ \left(
c_{v_{i}}\left( r\right) ,\text{ }c_{v_{i}}\left( -r\right) \right) \right\}
_{i=1}^{\mathrm{\dim }\left( S\right) }$ is a $\left( \mathrm{\dim }\left(
S\right) ,\tau \left( r\right) ,r\right) $--strainer for both $S$ and $X.$
Thus by Lemma \ref{5.6}, 
\begin{eqnarray*}
\left\vert \sphericalangle _{S}\left( v_{1},w\right) -\tilde{\sphericalangle}%
_{S}\left( c_{v_{1}}\left( r\right) ,p,c_{w}\left( s\right) \right)
\right\vert &<&\tau \left( r\right) +\tau \left( s|r\right) \\
\left\vert \sphericalangle _{X}\left( \iota \left( v_{1}\right) ,\iota
\left( w\right) \right) -\tilde{\sphericalangle}_{X}\left( c_{v_{1}}\left(
r\right) ,p,c_{w}\left( s\right) \right) \right\vert &<&\tau \left( r\right)
+\tau \left( s|r\right) .
\end{eqnarray*}

Combining this with Inequality (\ref{simialr triangles}), we see that 
\begin{eqnarray*}
\left\vert \sphericalangle _{X}\left( \iota \left( v_{1}\right) ,\iota
\left( w\right) \right) -\sphericalangle _{S}\left( v_{1},w\right)
\right\vert &<&\tau \left( r\right) +\tau \left( s|r\right) +\tau \left(
s,r\right) \\
&=&\tau \left( r\right) +\tau \left( s|r\right) .
\end{eqnarray*}%
Since this holds for all small $s$ and $r,$ $\iota $ is a metric embedding.
\end{proof}

The metric embedding $\iota :T_{p}^{1}S\longrightarrow \Sigma _{p}X$ induces
a metric embedding $T_{p}S\longrightarrow T_{p}X.$ From here on, we will
make no notational distinction between $T_{p}^{1}S$ and $T_{p}S$ and their
images under these embeddings. For example, we set 
\begin{equation*}
\iota \left( T_{p}^{1}S\right) =\Sigma _{p}S\subset \Sigma _{p}X,
\end{equation*}%
and for $c_{v}\left( t\right) =\exp _{p}^{S}\left( tv\right) ,$ we have 
\begin{equation}
c_{v}^{\prime }\left( 0\right) =\lim_{t\rightarrow 0^{+}}\left( \Uparrow
_{c\left( 0\right) }^{c\left( t\right) }\right) _{X},
\label{c prime is a direction eqn}
\end{equation}%
where all vectors are directions in $\Sigma _{p}X.$

\addtocounter{algorithm}{1}

\subsection{How to cover $\mathbf{S\hookrightarrow X}\label{cover S subsect}$%
}

In the main result of this subsection, Theorem \ref{magnum cover thm}, we
construct a cover $\mathcal{O}$ of $X$ that decomposes, $\mathcal{O=}%
\dbigcup\limits_{S\in \mathcal{S}^{\mathrm{ext}}}\mathcal{O}^{S},$ into
subcollections $\mathcal{O}^{S}$---one for each element of $\mathcal{S}^{%
\mathrm{ext}}.$ The elements of $\mathcal{O}^{X}$ are $\left( n,\delta
\right) $--strained and are contained in the top stratum. A posteriori,
their union is a Gromov-Hausdorff approximation of the sets $G_{\gamma }$ of
Part \ref{T3} of the TNST. Similarly, the elements of each $\mathcal{O}%
^{S_{i}}$ are $\mathrm{\dim \,}S_{i}$--strained by points of $S_{i},$ and
their union will be Gromov-Hausdorff close to the sets $\mathcal{U}_{\gamma
}^{S_{i}}$ of Part \ref{T2} of the TNST. In fact, the strainers for these
sets will also give us local Alexandrov versions of the diffeomorphism of
Part \ref{T4} and the submersions of Part \ref{T2} of the TNST.

The statement of Theorem \ref{magnum cover thm} is rather technical, so we
prove a series of preliminary results, beginning with the following
application of Equation \autoref{c prime is a direction eqn}.

\begin{lemma}
\label{one strained neigh prop}Let $\left( S,g\right) $ be a Riemannian $k$%
--manifold that is infinitesimally isometrically embedded in an Alexandrov
space $X,$ and let $K$ be a compact subset of $S.$ Given $\varepsilon ,%
\tilde{\delta}>0$ there is an $r_{0}>0$ so that for all $r\in \left(
0,r_{0}\right) $ there is a $\rho >0$ with the following properties.

\noindent 1. For all $p\in K,$ $B(p,3\rho )$ is $\left( k,\tilde{\delta}%
,r\right) $--strained in $X$ by points $\left\{ \left( a_{i},b_{i}\right)
\right\} _{i=1}^{k}$ contained in $S\times S.$

\noindent 2. For all $i,$ and for all $x\in B(p,3\rho )\cap S,$%
\begin{equation*}
\sphericalangle \left( \left( \Uparrow _{x}^{a_{i}}\right)
_{X},T_{x}^{1}S\right) <\varepsilon .
\end{equation*}
\end{lemma}

\begin{proof}
First we prove the existence of $r_{0}$ for a single point $p\in S.$ Take $%
\left\{ v_{i}\right\} _{i=1}^{k}\subset T_{p}S$ to be an orthonormal basis.
Let $c_{v_{i}}$ be the geodesic in $S$ with $\left( c_{v_{i}}\right)
^{\prime }\left( 0\right) =v_{i}.$ Choose $r\in \left( 0,\frac{1}{4}\mathrm{%
inj}_{p}\left( S\right) \right) $, and set $a_{i}=c_{v_{i}}\left( 4r\right) $
and $b_{i}=c_{v_{i}}\left( -4r\right) .$

Since $\left\{ v_{i}\right\} _{i=1}^{k}$ is an orthonormal basis for $T_{p}S$%
, it follows from Proposition \ref{isom embedd prop} that $\left\{
v_{i}\right\} _{i=1}^{k}$ is an orthonormal subset of $\Sigma _{p}X.$ So if $%
r$ is sufficiently small, then 
\begin{equation}
\left\{ \left( a_{i},b_{i}\right) \right\} _{i=1}^{k}\text{ is a }\left( k,%
\tilde{\delta},r\right) \text{--strainer for a neighborhood }N_{X}\text{ of }%
p\text{ in }X\text{,}  \label{k strainer}
\end{equation}%
giving us Property 1 at $p$.

By Equation \autoref{c prime is a direction eqn}, given $\eta \in \left(
0,\varepsilon \right) ,$ 
\begin{equation}
\sphericalangle \left( \left( \Uparrow _{p}^{a_{i}}\right) _{X},v_{i}\right)
<\eta ,  \label{close to S Inequality}
\end{equation}%
if $r$ is small enough.

Inequality \autoref{close to S Inequality} implies 
\begin{equation}
-1\leq D_{v_{i}}\mathrm{dist}_{a_{i}}^{X}\left( \cdot \right) =-\cos \left(
\sphericalangle \left( \left( \Uparrow _{p}^{a_{i}}\right) _{X},v_{i}\right)
\right) \leq -1+\tau \left( \eta \right) .  \label{cloooose
at p}
\end{equation}%
Set 
\begin{equation*}
V_{i}\left( x\right) =\left( \uparrow _{x}^{a_{i}}\right) _{S}.
\end{equation*}%
Part 1 of the definition of an infinitesimally, isometric embedding gives
that $D_{V_{i}\left( x\right) }\mathrm{dist}_{a_{i}}^{X}\left( \cdot \right) 
$ is close to $D_{v_{i}}\mathrm{dist}_{a_{i}}^{X}\left( \cdot \right) $ if $%
x $ is close to $p$. Combining this with Inequality \autoref{cloooose at p}
gives%
\begin{equation*}
D_{V_{i}\left( x\right) }\mathrm{dist}_{a_{i}}^{X}\left( \cdot \right)
=-1+\tau \left( \eta \right)
\end{equation*}%
for all $x$ in a neighborhood of $p.$ A direction $w\in \Sigma _{x}X$ for
which $D_{w}\mathrm{dist}_{a_{i}}^{X}\left( \cdot \right) =-1+\tau \left(
\eta \right) $ must be within $\tau \left( \eta \right) $ of $\left(
\Uparrow _{x}^{a_{i}}\right) _{X}.$ So viewing $V_{i}\left( x\right) \in
\Sigma _{x}X,$ it follows that 
\begin{equation*}
\sphericalangle \left( \left( \Uparrow _{x}^{a_{i}}\right) _{X},V_{i}\left(
x\right) \right) <\tau \left( \eta \right) .
\end{equation*}%
Since we also have $V_{i}\left( x\right) \in T_{x}S$,%
\begin{equation*}
\sphericalangle \left( \left( \Uparrow _{x}^{a_{i}}\right)
_{X},T_{x}^{1}S\right) <\tau \left( \eta \right) ,
\end{equation*}%
giving us Property 2 at $p$.

The existence of an $r_{0}$ that works uniformly throughout a compact subset 
$K$ of $S$ follows from the stability of Properties $1$ and $2.$ Indeed, if $%
\left\{ p_{i}\right\} _{i=1}^{\infty }\subset K$ converges to $p_{\infty
}\in K,$ then we have shown that Properties 1 and 2 hold for $p_{\infty }$.
It follows that they also hold for all but finitely many of the $\left\{
p_{i}\right\} _{i}^{\infty }$s with the corresponding constants divided by $%
2.$ The existence of a uniform $r_{0}$ follows from this and a contradiction
argument.
\end{proof}

Applying Lemmas \ref{Gromov Pack} and \ref{one strained neigh prop} to a
precompact open subset of $S,$ we get the following corollary.

\begin{corollary}
\label{k-strained cover thm}Let $\left( S,g\right) $ be a Riemannian $k$%
--manifold that is infinitesimally isometrically embedded in an Alexandrov
space $X.$ Let $O\subset S$ be a precompact open subset of $S.$ There is an $%
\mathfrak{o}>0$ so that given $\varepsilon ,\tilde{\delta}>0$ there is an $%
r>0,$ a $\rho _{0}\in \left( 0,r\right) ,$ and, for all $\rho \in \left(
0,\rho _{0}\right) ,$ a finite open cover $\mathcal{O}\equiv \left\{
B_{j}(\rho )\right\} _{j}$ of $O$ by $\rho $--balls of $X$ for which the
corresponding $3\rho $--balls have the following properties.

\noindent 1. Each $B_{j}(3\rho )$ is $\left( k,\tilde{\delta},r\right) $%
--strained in $X$ by $\left\{ \left( a_{i}^{j},b_{i}^{j}\right) \right\}
_{i=1}^{k}$ with $a_{i}^{j},b_{i}^{j}\in S.$

\noindent 2. For all $i,j$ and for all $x\in B_{j}(3\rho )\cap S,$%
\begin{equation}
\sphericalangle \left( \left( \Uparrow _{x}^{a_{i}^{j}}\right)
_{X},T_{x}S\right) <\varepsilon .  \label{alm tang to S inequal}
\end{equation}

\noindent 3. The first order of $\left\{ B_{j}(3\rho )\right\} _{j}$ is $%
\leq \mathfrak{o.}$
\end{corollary}

\begin{lemma}
\label{dist S is regular}Let $\left( S,g\right) $ be a Riemannian $k$%
--manifold that is infinitesimally isometrically embedded in an Alexandrov
space $X.$ Given any $p\in S$ and $\varepsilon ,\tilde{\delta}>0,$ let $%
\left\{ \left( a_{i},b_{i}\right) \right\} _{i=1}^{k}$ be as in the previous
lemma. There is an $\eta >0$ so that $\mathrm{dist}^{X}\left( S,\cdot
\right) $ is $\left( 1-\varepsilon \right) $--regular on $B\left( p,2\eta
\right) \setminus S.$

In fact, for all $x\in B\left( p,2\eta \right) \setminus S,$ there is a $%
V^{S}\in \Sigma _{x}$ so that%
\begin{equation}
D_{V^{S}}\mathrm{dist}^{X}\left( S,\cdot \right) >1-\varepsilon ,
\label{dist
S is rgg}
\end{equation}%
and 
\begin{equation}
\left\vert D_{V^{S}}\mathrm{dist}_{a_{i}}^{X}\right\vert \leq \tau \left( 
\tilde{\delta}\right) +\tau \left( \rho |r\right) .
\label{V is alm vertical}
\end{equation}
\end{lemma}

\begin{proof}
By Proposition \ref{isom embedd prop}, for all $p\in S,$ we have that $%
\Sigma _{p}S$ is a metric copy of $\mathbb{S}^{\mathrm{\dim }\left( S\right)
-1}\subset \Sigma _{p}X.$ By the Join Lemma \ref{Join Lemma}, $\Sigma _{p}X$
is isometric to $\mathbb{S}^{\mathrm{\dim }\left( S\right) -1}\ast E,$ where 
$E$ is a compact Alexandrov space of curvature $\geq 1.$ It follows that $%
T_{p}X$ splits orthogonally as 
\begin{equation*}
T_{p}X=T_{p}S\oplus C\left( E\right) .
\end{equation*}%
Under the convergence $\lim_{\lambda \longrightarrow \infty }\left( \lambda
X,p\right) =\left( T_{p}X,\ast \right) ,$ we have $\lim_{\lambda
\longrightarrow \infty }\left( \lambda S,p\right) =\left( T_{p}S,\ast
\right) .$ So the result holds with \thinspace $X,S,$ and $\left\{ \left(
a_{i},b_{i}\right) \right\} _{i=1}^{k}$ replaced by $T_{p}X,$ $T_{p}S,$ and $%
\left\{ \left( \uparrow _{p}^{a_{i}},\uparrow _{p}^{b_{i}}\right) \right\}
_{i=1}^{k}.$ The stability of regular points gives us that for all $x\in
B\left( p,2\eta \right) \setminus S,$ there is a $V^{S}\in \Sigma _{x}$ so
that 
\begin{equation*}
D_{V^{S}}\mathrm{dist}^{X}\left( S,\cdot \right) >1-\varepsilon .
\end{equation*}

Since $\uparrow _{p}^{a_{i}}=\lim_{\lambda \rightarrow \infty }pa_{i}\left( 
\frac{1}{\lambda }\right) $ and $\uparrow _{p}^{b_{i}}=\lim_{\lambda
\rightarrow \infty }pb_{i}\left( \frac{1}{\lambda }\right) ,$ it follows
that \autoref{V is alm vertical} holds with $\left\{ \left(
a_{i},b_{i}\right) \right\} _{i=1}^{k}$ replaced with $\left\{ \left( \tilde{%
a}_{i},\tilde{b}_{i}\right) \right\} _{i=1}^{k},$ where $\tilde{a}_{i}\equiv
pa_{i}\left( \frac{1}{\lambda }\right) $ and $\tilde{b}_{i}\equiv
pb_{i}\left( \frac{1}{\lambda }\right) .$ Since the directional derivatives
of $\mathrm{dist}_{a_{i}}$ and $\mathrm{dist}_{\tilde{a}_{i}}$ are nearly
the same at $p,$ \autoref{V is alm vertical} also holds.
\end{proof}

\begin{lemma}
\label{resp cover lemma}Let $N$ be an element of $\mathcal{S},$ and let $%
S\in \mathcal{S}$ be contained in $\bar{N}$ and not equal to $N.$ Given $%
p\in S$ and $\varepsilon ,\tilde{\delta}>0,$ let $\left\{ \left(
a_{i},b_{i}\right) \right\} _{i=1}^{\mathrm{\dim }\left( S\right) }$ be as
in Lemma \ref{one strained neigh prop}. If $\eta $ is sufficiently small,
then for all $\tilde{p}\in B\left( p,2\eta \right) \cap N,$ the following
hold.

\noindent 1. $\tilde{p}$ is $\left( \dim \left( N\right) ,\tau \left( \tilde{%
\delta},\eta \right) \right) $--strained in $X$ by $\left\{ \left(
a_{i},b_{i}\right) \right\} _{i=1}^{\mathrm{\dim }\left( S\right) }$ and $%
\left( \dim \left( N\right) -\mathrm{\dim }\left( S\right) \right) $ pairs
of points of $N,$ $\left\{ \left( a_{j}^{\tilde{p}},b_{j}^{\tilde{p}}\right)
\right\} _{j=\mathrm{\dim }\left( S\right) +1}^{\dim \left( N\right) }$.

\noindent 2. At every $x\in N$ that is close enough to $\tilde{p},$ 
\begin{equation}
\sphericalangle \left( \left( \Uparrow _{x}^{a_{i}}\right)
_{X},T_{x}N\right) <\varepsilon ,  \label{closer N NEW}
\end{equation}

and%
\begin{equation}
\sphericalangle \left( \left( \Uparrow _{x}^{a_{j}^{\tilde{p}}}\right)
_{X},T_{x}N\right) <\varepsilon .  \label{ext str close to N}
\end{equation}%
\noindent 3. For $V^{S}$ as in Lemma \ref{dist S is regular}, 
\begin{equation}
\sphericalangle \left( \left( \Uparrow _{x}^{a_{\mathrm{\dim }\left(
S\right) +1}^{\tilde{p}}}\right) _{X},V^{S}\right) <\tau \left( \tilde{\delta%
},\varepsilon \right) +\tau \left( \rho |r\right) ,  \label{V
close to N}
\end{equation}%
where $x\in B\left( \tilde{p},\rho \right) ,$ and $B\left( \tilde{p},\rho
\right) $ is $\left( \dim \left( N\right) ,\tau \left( \tilde{\delta},\eta
\right) ,r\right) $--strained in $X$ by 
\begin{equation*}
\left\{ \left\{ \left( a_{i},b_{i}\right) \right\} _{i=1}^{\mathrm{\dim }%
\left( S\right) },\left\{ \left( a_{j}^{\tilde{p}},b_{j}^{\tilde{p}}\right)
\right\} _{j=\mathrm{\dim }\left( S\right) +1}^{\dim \left( N\right)
}\right\} .
\end{equation*}

\noindent 4. If $N$ is the top stratum, that is, if $N=X\setminus \cup
_{S\in \mathcal{S}}S,$ then these same assertions hold except that in
Inequality \autoref{V close to N} we replace $\tau \left( \tilde{\delta}%
,\varepsilon \right) $ with $\tau \left( \delta \right) ,$ and in Part 1, $%
\tilde{p}$ is only $\left( \dim \left( N\right) ,\delta \right) $--strained.
\end{lemma}

\begin{Remark-delta}
The distinction between $\tau \left( \tilde{\delta},\varepsilon \right) ,$ $%
\tau \left( \delta \right) $ and $\delta $ in Part 4 is not merely academic.
In fact, $\tilde{\delta},\varepsilon $ and $\rho $ can be arbitrarily small
in Corollary \ref{k-strained cover thm}, whereas the $\delta $ such that all
points of our top stratum are $\left( n,\delta \right) $--strained is
determined by $X,$ and is therefore fixed.
\end{Remark-delta}

\begin{proof}
Since strainers are stable, every point $\tilde{p}\in B\left( p,\eta \right)
\cap N$ is $\left( \mathrm{\dim }\left( S\right) ,\tau \left( \tilde{\delta}%
,\eta \right) \right) $--strained in $X$ by $\left\{ \left(
a_{i},b_{i}\right) \right\} _{i=1}^{\mathrm{\dim }\left( S\right) }.$
Combining this with Lemma \ref{one strained neigh prop} and the fact that
every point of $N$ is $\left( \mathrm{\dim }\left( N\right) ,0\right) $%
--strained and not $\left( \mathrm{\dim }\left( N\right) +1,\delta \right) $%
--strained gives us Inequality \autoref{closer N NEW}, if we choose $\max
\left\{ \tilde{\delta},\eta ,\varepsilon \right\} <<\delta $.

The existence of $\left\{ \left( a_{j}^{\tilde{p}},b_{j}^{\tilde{p}}\right)
\right\} _{j=1}^{\dim \left( N\right) -\mathrm{\dim }\left( S\right) }$
follows from the fact that every point of $N$ is $\left( \dim \left(
N\right) ,0\right) $--strained, and the proof of Lemma \ref{one strained
neigh prop} gives us Inequality \autoref{ext str close to N}.

It follows from Inequalities \autoref{dist S is rgg} and \autoref{V is alm
vertical} that $\left( a_{\mathrm{\dim }\left( S\right) +1}^{\tilde{p}},b_{%
\mathrm{\dim }\left( S\right) +1}^{\tilde{p}}\right) $ can be chosen so that 
$\sphericalangle \left( \left( \Uparrow _{x}^{a_{\mathrm{\dim }\left(
S\right) +1}^{\tilde{p}}}\right) _{X},V^{S}\right) <\tau \left( \tilde{\delta%
},\varepsilon \right) +\tau \left( \rho |r\right) $, provided is $\eta $
small enough so that Lemma \ref{dist S is regular} holds.
\end{proof}

Let $X$ and $\mathcal{S}$ be as in Theorem \ref{dif stab- dim 4}. Recall
that 
\begin{equation*}
\mathcal{S}^{\mathrm{ext}}\equiv \mathcal{S}\cup \left( X\setminus \cup
_{S\in \mathcal{S}}S\right) .
\end{equation*}%
For an element $S\in \mathcal{S}^{\mathrm{ext}}$ , we write $\bar{S}$ for
the closure of $S$ and set 
\begin{equation*}
Bd\left( S\right) \equiv \bar{S}\setminus S.
\end{equation*}%
Note that $\mathrm{dim}\left( Bd\left( S\right) \right) $ can be $\leq 
\mathrm{dim}\left( S\right) -2;$ in particular, $\bar{S}$ need not be a
manifold with boundary.

From this point we fix a metric on $\left( \amalg _{\alpha }M_{\alpha
}\right) \amalg X$ that realizes the Gromov--Hausdorff convergence, and we
let $B\left( S,\nu \right) $ be the $\nu $--neighborhood of $S$ with respect
to this metric.

\begin{theorem}
\label{magnum cover thm}Let $X,$ $\mathcal{K},$ and $\mathcal{N}$ satisfy
the hypotheses of Theorem \ref{dif stab- dim 4}. Given $\varepsilon ,\tilde{%
\delta}>0,$ there are $\rho _{0}^{X},\rho _{0}^{S_{i}}>0,$ and, for all $%
\rho ^{X}\in \left( 0,\rho _{0}^{X}\right) $ and $\rho ^{S_{i}}\in \left(
0,\rho _{0}^{S_{i}}\right) ,$ there are collections of open sets $\mathcal{O}%
^{X}\equiv \left\{ B_{k}^{X}(\rho ^{X})\right\} _{k}$ and $\left\{ \mathcal{O%
}^{S_{i}}\right\} _{i}\equiv \left\{ \left\{ B_{j}^{S_{i}}(\rho
^{S_{i}})\right\} _{j\in I_{S_{i}}}\right\} _{i}$ where each $B_{k}^{X}(\rho
^{X})$ is a metric $\rho ^{X}$--ball of $X$ and each $B_{j}^{S_{i}}(\rho
^{S_{i}})$ is a metric $\rho ^{S_{i}}$--ball of $X$ with the following
properties.

\noindent 1. Set $O_{i}\equiv \cup _{j\in I_{S_{i}}}B_{j}^{^{S_{i}}}(\rho
^{^{S_{i}}})\cap S_{i}.$ Corollary \ref{k-strained cover thm} holds for each 
$O_{i}.$

\noindent 2. There is an $r>0$ so that each $B_{k}^{X}(3\rho ^{X})$ is $%
\left( n,\delta ,r\right) $--strained.

\noindent 3. If $S_{i}\in \mathcal{K},$ then $\mathcal{O}^{^{S_{i}}}$ is a
cover of $S_{i}$.

\noindent 4. For $N\in \mathcal{S}^{\mathrm{ext}},$ if $Bd\left( N\right)
=\cup _{n_{i}}S_{n_{i}},$ then $\mathcal{O}^{N}$ together with the union of
the $\mathcal{O}^{S_{n_{i}}}$ is a cover of $N.$

\noindent 5. For $S\in \mathcal{S}$ and $j\in I_{S},$ let $\left\{ \left(
B_{j}^{S}\right) _{\alpha }(3\rho ^{S})\right\} _{\alpha }$ be a sequence of
balls so that $\left( B_{j}^{S}\right) _{\alpha }(3\rho ^{S})\longrightarrow
B_{j}^{S}(3\rho ^{S})$ as $\alpha \rightarrow \infty ,$ and set 
\begin{equation*}
\widetilde{\mathcal{U}}_{\alpha }^{S}:=\cup _{j}\left( B_{j}^{S}\right)
_{\alpha }(3\rho ^{S}).
\end{equation*}%
If $\frac{1}{\alpha }$ and $\rho ^{S}$ are sufficiently small, then there is
a $\nu \in \left( 0,\frac{\rho ^{S}}{100}\right) $ and a smooth 
\begin{equation*}
d_{\alpha }^{S}:\widetilde{\mathcal{U}}_{\alpha }^{S}\setminus B\left( S,\nu
\right) \longrightarrow \mathbb{R}
\end{equation*}%
so that 
\begin{equation}
1-\varepsilon <\left\vert \nabla d_{\alpha }^{S}\right\vert <1+\varepsilon
\label{alms dist inequ}
\end{equation}%
and%
\begin{equation}
\left\vert D_{\nabla d_{\alpha }^{S}}\mathrm{dist}_{a_{i}^{\alpha
}}\right\vert <\tau \left( \tilde{\delta}\right) +\tau \left( \rho
^{S}|r\right) ,  \label{alm vert gradient}
\end{equation}%
where $a_{i}^{\alpha }\rightarrow a_{i},$ and $a_{i}$ is part of a strainer
for $S$ as in Corollary \ref{k-strained cover thm}.

\noindent 6. Let $S$ and $N$ be elements of $\mathcal{S}^{\mathrm{ext}}$,
and let $S$ be a subset of $Bd\left( N\right) .$ Then there is a $\nu \in
\left( 0,\frac{\rho ^{S}}{100}\right) $ so that for all $x\in \left( \cup 
\mathcal{O}^{N}\right) \cap \left( \cup \mathcal{O}^{S}\setminus B\left(
S,\nu \right) \right) ,$ there is a $B_{k}^{N}\left( \rho ^{N}\right) \in 
\mathcal{O}^{N}$ and a $B_{j\left( k\right) }^{S}\left( \rho ^{S}\right) \in 
\mathcal{O}^{S}$ so that%
\begin{eqnarray}
\text{ }x &\in &B_{k}^{N}\left( \rho ^{N}\right) ,  \notag \\
B_{k}^{N}\left( 3\rho ^{N}\right) &\Subset &B_{j\left( k\right) }^{S}\left(
\rho ^{S}\right) ,  \label{initialresp eqn}
\end{eqnarray}%
and if $B_{\alpha }^{N}\left( \rho ^{N}\right) $ is an approximation of $%
B^{N}\left( \rho ^{N}\right) $, then for all $x_{\alpha }\in $ $\widetilde{%
\mathcal{U}}_{\alpha }^{S}\setminus B\left( S,\nu \right) ,$ 
\begin{equation}
\sphericalangle \left( \Uparrow _{x_{\alpha }}^{a_{\mathrm{\dim }\left(
S\right) +1}^{\alpha }},\nabla d_{\alpha }^{S}\right) <\left\{ 
\begin{array}{ll}
\tau \left( \delta \right) +\tau \left( \rho ^{S}|r\right) , & \text{if }N%
\text{ is the top stratum} \\ 
\tau \left( \tilde{\delta},\varepsilon \right) +\tau \left( \rho
^{S}|r\right) & \text{otherwise},%
\end{array}%
\right.  \label{alm str ineqaul}
\end{equation}%
where $a_{\mathrm{\dim }\left( S\right) +1}^{\alpha }$ is an approximation
of the $\left( \mathrm{\dim }\left( S\right) +1\right) ^{st}$ strainer for $%
B^{N}\left( \rho ^{N}\right) $ constructed from Lemma \ref{resp cover lemma}.

\noindent 7. Let $S$ and $N$ be elements of $\mathcal{S}$, and let $S$ be a
subset of $Bd\left( N\right) .$ There is a $\nu \in \left( 0,\frac{\rho ^{S}%
}{100}\right) $ and a smooth function $d^{S}$ on $\left( \cup \mathcal{O}%
^{N}\right) \cap \left( \cup \mathcal{O}^{S}\setminus B\left( S,\nu \right)
\right) \cap N$ so that for all $x\in \left( \cup \mathcal{O}^{N}\right)
\cap \left( \cup \mathcal{O}^{S}\setminus B\left( S,\nu \right) \right) \cap
N,$ 
\begin{eqnarray*}
1-\varepsilon &<&\left\vert \nabla d^{S}\right\vert <1+\varepsilon , \\
\left\vert D_{\nabla d^{S}}\mathrm{dist}_{a_{i}}\right\vert &<&\tau \left( 
\tilde{\delta}\right) +\tau \left( \rho ^{S}|r\right) ,\text{ and} \\
\sphericalangle \left( \Uparrow _{x}^{a_{\mathrm{\dim }\left( S\right)
+1}},\nabla d^{S}\right) &<&\tau \left( \tilde{\delta},\varepsilon \right)
+\tau \left( \rho ^{S}|r\right) ,
\end{eqnarray*}%
where $a_{i}$ is part of a strainer for $S$ as in Lemma \ref{dist S is
regular} and $a_{\mathrm{\dim }\left( S\right) +1}$ is one of the strainers
for one of the $B^{N}\left( \rho ^{N}\right) $ that is constructed from
Lemma \ref{resp cover lemma}.
\end{theorem}

Next we define the \emph{Generation Number} of each $S\in \mathcal{S}.$ It
is dual to the concept of Ancestor Number that appears on page \pageref*%
{ansc numb page}. Recall that we partially ordered the $S\in \mathcal{S}^{%
\mathrm{ext}}$ by declaring that $S_{1}<S_{2}$ if $S_{1}\subsetneq \bar{S}%
_{2}$, where $\bar{S}_{2}$ is the closure of $S_{2}.$ We call the number$,$ $%
a,$ the \emph{Generation Number} of $S\in \mathcal{S}^{\mathrm{ext}}$ if $a$
is the length of the largest chain 
\begin{equation*}
S_{0}<S_{1}<\cdots <S_{a}
\end{equation*}%
with $S=S_{a}$ and $S_{0}=\bar{S}_{0}.$ Let $\mathcal{S}_{j}$ be the
collection of all $S\in \mathcal{S}^{\mathrm{ext}}$ that have generation
number $j.$

\begin{proof}[Proof of Theorem \protect\ref{magnum cover thm}]
The construction of $\mathcal{O}^{X}$ and the $\mathcal{O}^{^{S_{i}}}$ is by
induction on the Generation Number. If $S\in \mathcal{S}$ has generation
number $0,$ then we get the desired cover $\mathcal{O}^{S}$ from Corollary %
\ref{k-strained cover thm}.

Suppose by induction that we have constructed the desired cover $\mathcal{O}%
\left( k\right) $ of the union of the elements of $\cup _{j=0}^{k}\mathcal{S}%
_{j},$ and $3\mathcal{O}\left( k\right) $ is the corresponding cover by
balls with three times the radius. For $N\in \mathcal{S}_{k+1},$ let 
\begin{equation*}
J_{N}\equiv \left\{ \left. j\in I\text{ }\right\vert \text{ }S_{j}\subset
Bd\left( N\right) \text{ and }S_{j}\in \mathcal{S}\right\} .
\end{equation*}%
Given $\nu >0,$ we apply Lemma \ref{resp cover lemma} to obtain a cover $%
\mathcal{O}^{N,\mathrm{pre}}$ of 
\begin{equation*}
\left\{ \left( \cup 3\mathcal{O}\left( k\right) \right) \setminus \cup
_{j\in J_{N}}B\left( S_{j},\nu \right) \right\} \cap N
\end{equation*}%
that satisfies (\ref{initialresp eqn}). Since $\left\{ \left( \cup 3\mathcal{%
O}\left( k\right) \right) \setminus \cup _{j\in J_{N}}B\left( S_{j},\nu
\right) \right\} \cap N$ is precompact in $N,$ we can take $\mathcal{O}^{N,%
\mathrm{pre}}$ to be a finite cover. We then apply Corollary \ref{k-strained
cover thm} with $O=N\setminus \cup _{j\in J_{N_{l}}}B\left( S_{j},\nu
\right) $ to get the desired cover of $N.$ Since there are only finitely
many $N\in \mathcal{S}_{k+1},$ this completes the induction step, and hence
the proofs of Parts 1--4.

To construct the function $d_{\alpha }^{S}$ that appears in Parts 5 and 6,
let $h_{\alpha }:X\longrightarrow M_{\alpha }$ be a $\tau \left( 1/\alpha
\right) $--homeomorphism constructed via Perelman's Stability Theorem. Since
the conclusion of Lemma \ref{dist S is regular} is Gromov--Hausdorff stable,
given any $\varepsilon >0,$ there is a $\nu >0$ and a unit vector field $%
V_{\alpha }$ on $\widetilde{\mathcal{U}}_{\alpha }^{S}\setminus B\left(
S,\nu \right) $ with 
\begin{equation}
D_{V_{\alpha }}\mathrm{dist}\left( h_{\alpha }\left( S\right) ,\cdot \right)
>1-\varepsilon \text{.}  \label{dist_S reg}
\end{equation}%
Under the hypotheses of Part 6, Parts 3 and 4 of Lemma \ref{resp cover lemma}
give us that 
\begin{equation}
\sphericalangle \left( \Uparrow _{x_{\alpha }}^{a_{\mathrm{\dim }\left(
S\right) +1}^{\alpha }},V_{\alpha }\right) <\left\{ 
\begin{array}{ll}
\tau \left( \delta \right) +\tau \left( \rho ^{S}|r\right) , & \text{if }N%
\text{ is the top stratum} \\ 
\tau \left( \tilde{\delta},\varepsilon \right) +\tau \left( \rho
^{S}|r\right) & \text{otherwise}.%
\end{array}%
\right.  \label{close to strainer}
\end{equation}

We apply the Riemannian convolution method of \cite{GrovShio} to $\mathrm{%
dist}\left( h_{\alpha }\left( S\right) ,\cdot \right) .$ Since Riemannian
convolutions preserve regularity, it follows from \autoref{dist_S reg} and (%
\ref{close to strainer}) that for an appropriate convolution $d_{\alpha },$ 
\begin{eqnarray*}
D_{V_{\alpha }}d_{\alpha }^{S} &>&1-\varepsilon \text{ , } \\
1-\varepsilon &<&\left\vert \nabla d_{\alpha }^{S}\right\vert <1+\varepsilon
,\text{ and } \\
\text{ }\sphericalangle \left( \Uparrow _{x_{\alpha }}^{a_{\mathrm{\dim }%
\left( S\right) +1}^{\alpha }},\nabla d_{\alpha }^{S}\right) &<&\left\{ 
\begin{array}{ll}
\tau \left( \delta \right) +\tau \left( \rho ^{S}|r\right) , & \text{if }N%
\text{ is the top stratum} \\ 
\tau \left( \tilde{\delta},\varepsilon \right) +\tau \left( \rho
^{S}|r\right) & \text{otherwise}%
\end{array}%
\right.
\end{eqnarray*}%
on $\widetilde{\mathcal{U}}_{\alpha }^{S}\setminus B\left( S,\nu \right) ,$
where $d_{\alpha }$ is $C^{\infty }$ and as close as we please to $\mathrm{%
dist}\left( h_{\alpha }\left( S\right) ,\cdot \right) $ in the $C^{0}$%
--topology.

The function $d^{S}$ in Part 7 is constructed through a completely analogous
argument.
\end{proof}

For $a\in X$ and $\eta >0,$ we define $g_{a}:X\longrightarrow \mathbb{R}$ by 
\begin{equation}
g_{a}(y)=\frac{1}{\mathrm{vol}(B(a,\eta ))}\int_{z\in B(a,\eta )}\mathrm{dist%
}(y,z).  \label{average}
\end{equation}

Differentiation under the integral and the directional differentiability of
distance functions gives the following.

\begin{proposition}
\label{Diredtional derivatives}If $X$ is a Riemannian manifold, then $g_{a}$
is $C^{1},$ and in general, for any $v\in T_{y}X,$%
\begin{equation*}
D_{v}\left( g_{a}\right) =\frac{1}{\mathrm{vol}(B(a,\eta ))}\int_{z\in
B(a,\eta )}D_{v}\left( \mathrm{dist}(\cdot ,z)\right) .
\end{equation*}
\end{proposition}

Let $B(x,\rho )$ be $(l,\delta ,r)$--strained by $\left\{ \left(
a_{i},b_{i}\right) \right\} _{i=1}^{l}$. If $B(x,\rho ^{N})=B_{k}^{N}\left(
\rho ^{N}\right) \in \mathcal{O}^{N}$ is very near an $S\in \mathcal{S}$ as
in Part 6 of Theorem \ref{magnum cover thm} and $\left( B_{k}^{N}\right)
_{\alpha }\left( \rho ^{N}\right) \subset M_{\alpha }$ is an approximation
of $B_{k}^{N}\left( \rho ^{N}\right) ,$ we define $p_{\mathrm{rel}}^{\alpha
}:\left( B_{k}^{N}\right) \left( \rho ^{N}\right) \rightarrow \mathbb{R}^{l}$
by 
\begin{equation}
p_{\mathrm{rel}}^{\alpha }(y)\equiv (g_{a_{1}^{\alpha }}(y),\ldots ,g_{a_{%
\mathrm{\dim }\left( S\right) }^{\alpha }}(y),d_{\alpha }^{S},\ldots
,g_{a_{l}^{\alpha }}(y)).  \label{alt dfn of p}
\end{equation}%
Otherwise, we define $p_{\mathrm{conv}}^{\alpha }:B(x,\sigma )\rightarrow 
\mathbb{R}^{l}$ by 
\begin{equation}
p_{\mathrm{conv}}^{\alpha }(y)\equiv (g_{a_{1}^{\alpha }}(y),\ldots
,g_{a_{l}^{\alpha }}(y)).  \label{dfn of p}
\end{equation}

The distinction between $p_{\mathrm{rel}}^{\alpha }$ and $p_{\mathrm{conv}%
}^{\alpha }$ will mostly be irrelevant, and most statements about them will
be true of both types of maps. For such statements, we use a plain $%
``p^{\alpha }$\textquotedblright\ to stand for either map. Note that all of
the $p_{\mathrm{rel}}^{\alpha }$s are $C^{1}$--close to some $p_{\mathrm{conv%
}}^{\alpha }.$

It is of course true that $p^{\alpha }$ depends on $\eta ;$ however, we
adopt the convention that all assertions about the maps $p^{\alpha }$
defined in (\ref{alt dfn of p}) and (\ref{dfn of p}) have the added implicit
assumption that $\eta $ is sufficiently small.

\begin{corollary}
\label{old part 5 cor}For $N\in \mathcal{S},$ let $S$ be a subset of $%
Bd\left( N\right) .$ Let $B_{k}^{N}\left( \rho ^{N}\right) \in \mathcal{O}%
^{N}$ and $B_{j\left( k\right) }^{S}\left( \rho ^{S}\right) \in \mathcal{O}%
^{S}$ be as in (\ref{initialresp eqn}), that is, $B_{k}^{N}\left( 3\rho
^{N}\right) \Subset B_{j\left( k\right) }^{S}\left( \rho ^{S}\right) $. Then 
\begin{equation}
\pi _{\mathrm{\dim }\left( S\right) }\circ \left( p_{k}^{N}\right) ^{\alpha
}=\left( p_{j\left( k\right) }^{S}\right) ^{\alpha },  \label{resp near stra}
\end{equation}%
where $\left( p_{k}^{N}\right) ^{\alpha }:\left( B_{k}\right) _{\alpha
}(\rho ^{N})\longrightarrow \mathbb{R}^{\mathrm{\dim }\left( N\right) }$ and 
$\left( p_{j\left( k\right) }^{S}\right) ^{\alpha }:\left( B_{j\left(
k\right) }^{S}\right) _{\alpha }(\rho ^{S})\longrightarrow \mathbb{R}^{%
\mathrm{\dim }\left( S\right) }$ are defined as in (\ref{alt dfn of p}) and (%
\ref{dfn of p}), and $\pi _{\mathrm{\dim }\left( S\right) }:\mathbb{R}^{\dim
\left( N\right) }\longrightarrow \mathbb{R}^{\mathrm{\dim }\left( S\right) }$
is projection onto the first $\mathrm{\dim }\left( S\right) $ factors.
\end{corollary}

\section{Local Strain and Convex Structure of Alexandrov Spaces\label%
{Perelman Section}}

The{\Huge \ }main result of this section is Theorem \ref{Perel cap thm}. It
provides local versions of the vector bundles of Part \ref{T2} of the TNST
over each member of the open cover of Theorem \ref{magnum cover thm}. In the
next section we show that the projections of our local vector bundles are $%
C^{1}$--close on their intersections, and in Section \ref{Gluing Section} we
state a theorem about gluing together $C^{1}$--close submersions.

Theorem \ref{Perel cap thm} is proven by combining Theorem \ref{magnum cover
thm} with Perelman's remarkable concavity construction. We start with a
review of Perelman Concavity.

\begin{proposition}
\label{Universial Perelman}(Perelman Concavity, \cite{Perel2}) Let $X$ be an 
$n$--dimensional Alexandrov space of curvature $\geq -1$. Suppose $q,p\in X$
satisfy $\mathrm{dist}\left( q,p\right) =d,$ and for some $\eta ,v>0,$ $%
\mathrm{vol}\left( B\left( p,\eta \right) \right) \geq v.$ Then there is a $%
\delta >0$ and a smooth increasing function $\psi :\mathbb{R}\longrightarrow 
\mathbb{R}$ so that 
\begin{equation*}
f_{p}\left( x\right) =\frac{1}{\mathrm{vol}\left( B\left( p,\eta \right)
\right) }\int_{z\in B\left( p,\eta \right) }\psi \circ \mathrm{dist}\left(
x,z\right)
\end{equation*}%
is strictly $-1$--concave on $B\left( q,\delta \right) .$

Moreover, if $\psi $ satisfies $\frac{1}{2}<\psi ^{\prime }\leq 2,$ then $%
f_{p}$ is directionally differentiable and satisfies 
\begin{equation}
\left\vert D_{V}f_{p}\right\vert \leq 2  \label{grad f is bonded}
\end{equation}%
for all directions $v.$
\end{proposition}

\begin{proof}
The idea is to choose $\psi $ to have a very negative second derivative and
so that $\frac{1}{2}<\psi ^{\prime }\leq 2$ on a very small interval around
the number $\mathrm{dist}\left( p,q\right) .$

Indeed, the lower curvature bound gives us a $\lambda >0$ so that for any $%
z\in B\left( p,\eta \right) ,$ $x$ near $q,$ and a direction $w\in \Sigma
_{x}$, 
\begin{equation}
\psi \circ \mathrm{dist}\left( \gamma _{w}\left( t\right) ,z\right) \text{
is }\lambda \text{--concave}.  \label{not so badd}
\end{equation}%
But for most $z\in B\left( p,\eta \right) ,$ we can do much better. In fact,
since $\psi ^{\prime \prime }<<-2,$ 
\begin{equation}
\psi \circ \mathrm{dist}\left( \gamma _{w}\left( t\right) ,z\right) \text{
is }\left( -2\right) \text{--concave},  \label{yeah baby yeah}
\end{equation}%
unless $\left\vert \sphericalangle \left( w,\Uparrow _{x}^{z}\right) -\frac{%
\pi }{2}\right\vert \leq \tau \left( \left. \frac{1}{\left\vert \psi
^{\prime \prime }\right\vert }\right\vert d\right) .$

Next set 
\begin{equation*}
\log _{x}B\left( p,\eta \right) \equiv \left\{ \left. u\in T_{x}X\text{ }%
\right\vert \text{ }\gamma _{u}\text{ is a segment from }x\text{ to }\gamma
_{u}\left( 1\right) \in B\left( p,\eta \right) \right\} .
\end{equation*}%
Then for some $C>0$ (that depends only on $d$), we have 
\begin{equation}
C\cdot \mathrm{vol}\left( \log _{x}B\left( p,\eta \right) \right) \geq 
\mathrm{vol}\left( B\left( p,\eta \right) \right) >v>0.  \label{Bish-Gromov}
\end{equation}%
Given $w\in \Sigma _{x},$ the set of \textquotedblleft bad
directions\textquotedblright\ for $w$, 
\begin{equation*}
\mathrm{B}\left( w\right) \equiv \left\{ \left. u\in \Sigma _{x}\right\vert
\left\vert \sphericalangle \left( w,u\right) -\frac{\pi }{2}\right\vert \leq
\tau \left( \left. \frac{1}{\left\vert \psi ^{\prime \prime }\right\vert }%
\right\vert d\right) \right\} ,
\end{equation*}%
has $\left( n-1\right) $--dimensional volume 
\begin{equation*}
\mathrm{vol}_{n-1}\left( \mathrm{B}\left( w\right) \right) \leq \tau \left(
\left. \frac{1}{\left\vert \psi ^{\prime \prime }\right\vert }\right\vert
d\right) .
\end{equation*}%
So 
\begin{equation*}
\mathrm{vol}_{n}\left( \log _{x}B\left( p,\eta \right) \cap \left\{ \left.
u\in T_{x}X\text{ }\right\vert \text{ }\frac{u}{\left\vert u\right\vert }\in 
\mathrm{B}\left( w\right) \right\} \right) \leq \tau \left( \left. \frac{1}{%
\left\vert \psi ^{\prime \prime }\right\vert }\right\vert d\right) \tau
\left( \eta \right) ,
\end{equation*}%
and using Inequality \autoref{Bish-Gromov}, 
\begin{equation*}
\frac{\mathrm{vol}_{n}\left( \log _{x}B\left( p,\eta \right) \cap \left\{
\left. u\in T_{x}X\text{ }\right\vert \text{ }\frac{u}{\left\vert
u\right\vert }\in \mathrm{B}\left( w\right) \right\} \right) }{\mathrm{vol}%
_{n}\left( \log _{x}B\left( p,\eta \right) \right) }\leq \frac{C}{v}\tau
\left( \left. \frac{1}{\left\vert \psi ^{\prime \prime }\right\vert }%
\right\vert d\right) \tau \left( \eta \right) .
\end{equation*}%
By combining this with \autoref{not so badd} and \autoref{yeah baby yeah},
we can force $f_{p}$ to be strictly $-1$--concave on $B\left( q,\delta
\right) $ with appropriate choices of $\psi $ and $\delta .$

Since $\frac{1}{2}<\psi ^{\prime }\leq 2$ and $\mathrm{dist}\left( \cdot
,z\right) $ is directionally differentiable and $1$--Lipschitz, we apply the
Bounded Convergence Theorem to differentiate under the integral and conclude
that $f_{p}$ is directionally differentiable and satisfies \autoref{grad f
is bonded}.
\end{proof}

A Gram-Schmidt argument as in \cite{Wilh} or \cite{GrovWilh2} gives us the
following.

\begin{lemma}
\label{Gram-Schnidt lemma}Let 
\begin{equation*}
p:U\longrightarrow \mathbb{R}^{k}
\end{equation*}%
be a submersion from an open subset $U$ of a Riemannian manifold. Suppose
that the component functions $g_{i}$ of $p$ are concave down and their
gradients satisfy%
\begin{equation*}
\sphericalangle \left( \nabla g_{i},\nabla g_{j}\right) >\frac{\pi }{2}
\end{equation*}%
for all $i\neq j.$ Let $f:U\longrightarrow \mathbb{R}^{k}$ be a strictly
concave down function so that for all $i,$ 
\begin{equation*}
\sphericalangle \left( \nabla f,\nabla g_{i}\right) >\frac{\pi }{2}.
\end{equation*}%
Then the restrictions of $f$ to the fibers of $p$ are strictly concave down.
\end{lemma}

In the context of a $k$--strained point, we combine the previous two results
to get the following.

\begin{lemma}
\label{k-strained Stability}Let $M_{\alpha }$ be a sequence of Riemannian $n$%
--manifolds with curvature $\geq -1$ that converges to an $n$--dimensional
Alexandrov space $X.$ Suppose $q\in X$ is $\left( k,\delta ,r\right) $%
--strained by $\left\{ \left( a_{i},b_{i}\right) \right\} _{i=1}^{k}$ and $%
q_{\alpha }\in M_{\alpha }$ converge to $q.$

\noindent 1. $\left( \text{cf \cite{GrovWilh2}, \cite{Kap1}}\right) $ There
is a convex neighborhood $C$ of $q$ and, for all but finitely many $\alpha ,$
convex neighborhoods $C^{\alpha }$ of $q_{\alpha }$ so that%
\begin{equation*}
C^{\alpha }\longrightarrow C.
\end{equation*}

\noindent 2. For all but finitely many $\alpha $, there is a $\left( \tau
\left( \delta \right) +\tau \left( 1/\alpha \text{ }|\text{ }r\right)
\right) $--almost Riemannian submersion 
\begin{equation*}
p_{\mathrm{conv}}^{\alpha }:C^{\alpha }\longrightarrow \mathbb{R}^{k}
\end{equation*}%
and a $\left( -1\right) $--concave function 
\begin{equation*}
f_{C^{\alpha }}^{\alpha }:C^{\alpha }\longrightarrow \mathbb{R}
\end{equation*}%
so that the restriction of $f_{C^{\alpha }}^{\alpha }$ to each fiber of $p_{%
\mathrm{conv}}^{\alpha }$ is strictly concave and has a unique interior
maximum. Moreover, $\left( \mathrm{int}\left( C^{\alpha }\right) ,p_{\mathrm{%
conv}}^{\alpha }\right) $ is a vector bundle, and \textrm{int}$\left(
C^{\alpha }\right) $ is diffeomorphic to $\left( 0,1\right) ^{n}$ via a
diffeomorphism $\mu ^{\alpha }$ that coincides with $p_{\mathrm{conv}%
}^{\alpha }$ on the first $k$ factors.
\end{lemma}

\begin{proof}
We apply Proposition \ref{almost Plauter} and conclude that $\Sigma _{q}X$
has a global $\left( k,\tau \left( \delta \right) \right) $--strainer $%
\left\{ \left( v_{i},w_{i}\right) \right\} _{i=1}^{k}$ so that 
\begin{equation}
\frac{\pi }{2}<\mathrm{dist}\left( v_{i},v_{j}\right) \text{ for }i\neq j.
\label{helper-1}
\end{equation}%
Moreover, for all $\kappa \in \left( 0,\frac{\pi }{4}\right) ,$ if $\delta $
is sufficiently small compared to $\kappa ,$ there is a nonempty set $%
E\subset \Sigma _{q}X$ so that for all $e\in E,$ 
\begin{equation*}
\frac{\pi }{2}<\mathrm{dist}\left( e,v_{i}\right) <\frac{\pi }{2}+\kappa
\end{equation*}%
and%
\begin{equation*}
\left\vert \mathrm{dist}\left( e,w_{i}\right) -\frac{\pi }{2}\right\vert
<\kappa .
\end{equation*}%
Take $E\subset \Sigma _{q}X$ to be the set of all directions that satisfy
these inequalities.

By exponentiating approximations of these directions, it follows that there
is a neighborhood $N$ of $q$ that is $\left( k,\tau \left( \delta \right) ,%
\frac{r}{2}\right) $--strained by a strainer $\left\{ \left(
a_{i},b_{i}\right) \right\} _{i=1}^{k}$ that satisfies 
\begin{equation}
\frac{\pi }{2}<\sphericalangle \left( \Uparrow _{x}^{a_{i}},\Uparrow
_{x}^{a_{j}}\right)  \label{<(a_i, a_j)}
\end{equation}%
for all $x\in N$ and $i\neq j.$ Using Lemma $\ref{5.6}$, for some $d>0,$ we
also have 
\begin{equation}
\tilde{\sphericalangle}\left( a_{i},q,\exp _{q}\left( de\right) \right) >%
\frac{\pi }{2},  \label{<(a_i, e)}
\end{equation}%
and 
\begin{equation*}
\left\vert \tilde{\sphericalangle}\left( b_{i},q,\exp _{q}\left( de\right)
\right) -\frac{\pi }{2}\right\vert <\tau \left( \left. \delta ,d,\kappa
\right\vert r\right)
\end{equation*}%
for all $e\in E$ for which $\exp _{q}\left( de\right) $ is defined. Since
the last two inequalities are for comparison angles, $q$ can be replaced by
any $x\in N,$ provided $N$ is sufficiently small.

Let $\left\{ e_{j}\right\} $ be a $\frac{\pi }{4}$--net in $E$ for which $%
\exp _{q}\left( de_{j}\right) $ is defined$.$ Apply the Perelman Concavity
construction to $\exp _{q}\left( de_{j}\right) $ and each of the strainer
points to get strictly $-1$--concave functions $\left\{ f_{e_{j}}\right\}
,\left\{ g_{a_{i}}\right\} ,\left\{ g_{b_{i}}\right\} $ defined in a
possibly smaller neighborhood $U$ of $q,$ and set%
\begin{equation*}
h=\min_{i,j}\left\{ f_{e_{j}},g_{a_{i}},g_{b_{i}}\right\} .
\end{equation*}

For some $\varepsilon >0,$%
\begin{equation}
\left\{ \Uparrow _{q}^{\tilde{a}_{i}}\right\} \cup \left\{ \Uparrow _{q}^{%
\tilde{b}_{i}}\right\} \cup \left\{ \tilde{e}_{j}\right\} \text{ is a }%
\left( \frac{\pi }{2}-\varepsilon \right) \text{--net in }\Sigma _{q}X,
\label{pi over 2 net}
\end{equation}%
provided $\tilde{a}_{i},$ $\tilde{b}_{i},$ and $\tilde{e}_{j}$ are
sufficiently close to $a_{i},$ $b_{i},$ and $e_{j}.$ By adding constants to
the $f_{e_{j}}$s$,$ $g_{a_{i}}$s, and $g_{b_{i}}$s, we can arrange that%
\begin{equation}
f_{e_{j}}\left( q\right) =g_{a_{i}}\left( q\right) =g_{b_{i}}\left( q\right)
\label{all equal}
\end{equation}%
for all $i$ and $j.$ Combining (\ref{pi over 2 net}) and (\ref{all equal})
with the fact that $h$ is strictly $-1$--concave on $U$, it follows that $q$
is the unique maximum of $h$ on $U.$ Let $C$ be a superlevel set of $h$ that
is contained in $U.$

Let $M_{\alpha }$ be sufficiently close to $X.$ The universality of
Perelman's construction implies, in particular, that it is stable under
Gromov--Hausdorff approximation, so each of $h,$ $C,$ and the $f_{e_{j}}$s, $%
g_{a_{i}}$s$,$ and $g_{b_{i}}$s have approximations in $M_{\alpha }.$ Call
these approximations $h^{\alpha },$ $C^{\alpha },$ $f_{e_{j}}^{\alpha }$, $%
g_{a_{i}}^{\alpha },$and $g_{b_{i}}^{\alpha }$. If $\alpha $ is sufficiently
large, the $f_{e_{j}}^{\alpha }$s, $g_{a_{i}}^{\alpha }$s$,$and $%
g_{b_{i}}^{\alpha }$s are strictly $-1$--concave, $C^{\alpha }$ is convex,
and the maximum of $h^{\alpha }$ is in the interior of $C^{\alpha }.$ So $%
C^{\alpha }$ is diffeomorphic to an $n$--disk.

Set 
\begin{eqnarray*}
p_{\mathrm{conv}}^{\alpha } &:&C^{\alpha }\longrightarrow \mathbb{R}^{k} \\
p_{\mathrm{conv}}^{\alpha } &=&\left( g_{a_{1}}^{\alpha },g_{a_{2}}^{\alpha
},\ldots ,g_{a_{k}}^{\alpha }\right) .
\end{eqnarray*}

Since $C^{\alpha }$ is $\left( k,\tau \left( \delta \right) +\tau \left(
1/\alpha \text{ }|\text{ }r\right) ,r\right) $--strained, it follows from
Lemma \ref{angle convergence} that $p_{\mathrm{conv}}^{\alpha }$ is a $%
\left( \tau \left( \delta \right) +\tau \left( 1/\alpha \text{ }|\text{ }%
r\right) \right) $--almost Riemannian submersion.

Proposition \ref{Diredtional derivatives} and inequalities \autoref{<(a_i,
a_j)} and \autoref{<(a_i, e)} give us 
\begin{eqnarray}
\sphericalangle \left( \nabla g_{a_{i}}^{\alpha },\nabla g_{a_{j}}^{\alpha
}\right) &>&\frac{\pi }{2}\,\ \,\text{and}  \notag \\
\sphericalangle \left( \nabla g_{a_{i}}^{\alpha },\nabla f_{e_{j}}^{\alpha
}\right) &>&\frac{\pi }{2},  \label{helper}
\end{eqnarray}%
for $\alpha $ sufficiently large. {\LARGE \ }Combining this with Lemma \ref%
{Gram-Schnidt lemma}, it follows that the restriction of each $%
f_{e_{j}}^{\alpha }$ to the fibers of $p_{\mathrm{conv}}^{\alpha }$ is
concave down. Set 
\begin{equation*}
f_{C^{\alpha }}^{\alpha }\equiv \min_{i}\left\{ f_{e_{i}}^{\alpha }\right\} .
\end{equation*}%
It follows that the restriction of $f_{C}^{\alpha }$ to each fiber $\left(
p_{\mathrm{conv}}^{\alpha }\right) ^{-1}\left( p_{\mathrm{conv}}^{\alpha
}\left( x\right) \right) $ of $p_{\mathrm{conv}}^{\alpha }$ is strictly
concave, and, after possibly restricting the base of $p_{\mathrm{conv}%
}^{\alpha },$ that each $f_{C}^{\alpha }|_{\left( p_{\mathrm{conv}}^{\alpha
}\right) ^{-1}\left( p_{\mathrm{conv}}^{\alpha }\left( x\right) \right) }$
has a unique interior maximum. In particular, each fiber of $p_{\mathrm{conv}%
}^{\alpha }$ is a disk, so there is a diffeomorphism $\mu ^{\alpha
}:C^{\alpha }\longrightarrow I^{n}$ whose first $k$ coordinate functions are 
$p_{\mathrm{conv}}^{\alpha }=\left( g_{a_{1}}^{\alpha },g_{a_{2}}^{\alpha
},\ldots ,g_{a_{k}}^{\alpha }\right) .$

To see that $\left( C^{\alpha },p_{\mathrm{conv}}^{\alpha }\right) $ is a
vector bundle, let $s_{x}^{\alpha }$ be the unique maximum of $f_{C^{\alpha
}}^{\alpha }$ restricted to $\left( p_{\mathrm{conv}}^{\alpha }\right)
^{-1}\left( p_{\mathrm{conv}}^{\alpha }\left( x\right) \right) .$ The
collection 
\begin{equation*}
S^{\alpha }\equiv \left\{ s_{x}^{\alpha }\right\} _{x\in C_{j}^{\alpha }}
\end{equation*}%
forms a $\dim \left( S\right) $--dimensional submanifold of $C^{\alpha }.$
The gradients of $f^{\alpha }$ restricted to the fibers of $p_{\mathrm{conv}%
}^{\alpha }$ allow us to identify the fibers of $p_{\mathrm{conv}}^{\alpha }$
with the normal bundle of $S^{\alpha },$ thus giving $\left( C^{\alpha },p_{%
\mathrm{conv}}^{\alpha }\right) $ the structure of a trivial vector bundle.
\end{proof}

Recall that in Theorem \ref{magnum cover thm} we constructed a cover of $X$
by subcollections, $\mathcal{O}^{X}\equiv \left\{ B_{j}^{X}(\rho
^{X})\right\} _{j}$ and $\left\{ \mathcal{O}^{S_{i}}\right\} _{i}\equiv
\left\{ \left\{ B_{j}^{S_{i}}(\rho ^{S_{i}})\right\} _{j}\right\} _{i}.$ To
simplify notation, we will refer to a $B_{j}^{S_{i}}(\rho ^{S_{i}})$ or to a 
$B_{j}^{X}(\rho ^{X})$ as simply $B_{j}(\rho _{j}),$ and let $p_{j}$ be the
map $B_{j}\left( \rho _{j}\right) \longrightarrow \mathbb{R}^{\mathrm{\dim }%
\left( S_{i}\right) }$ from (\ref{alt dfn of p}) or (\ref{dfn of p}). We
write $S_{j}$ for the element of $\mathcal{S}^{\mathrm{ext}}$ associated to $%
B_{j}(\rho _{j}).$ Thus for $S\in \mathcal{S}$ and $B_{j}(\rho _{j})\in 
\mathcal{O}^{S},$ we have $S_{j}=S.$ Of course, $S_{j}$ might be our top
stratum, $\left( X\setminus \cup _{S\in \mathcal{S}}S\right) ,$ and, with
this notation, many of the $S_{j}$s are likely to be equal to each other.

\begin{theorem}
\label{Perel cap thm}Let $X$ and $\left\{ M_{\alpha }\right\} _{\alpha }$ be
as in the TNST. Given $\varepsilon >0,$ let $\left\{ B_{j}(\rho
_{j})\right\} _{j}$ be the open cover of $X$ from Theorem \ref{magnum cover
thm}. If the $\rho _{j}$s are sufficiently small, then the following hold.

\noindent 1. For all but finitely many $\alpha $ and for all $j$ for which $%
S_{j}$ is not the top stratum, there is a $3\rho _{j}$--ball $B_{j}^{\alpha
}(3\rho _{j})\subset M_{\alpha }$ so that%
\begin{equation*}
B_{j}^{\alpha }(3\rho _{j})\longrightarrow B_{j}(3\rho _{j})
\end{equation*}%
as $\alpha \longrightarrow \infty .$ Moreover, there are $\varepsilon $%
--almost Riemannian submersions 
\begin{eqnarray*}
p_{j}^{\alpha } &:&B_{j}^{\alpha }(3\rho _{j})\longrightarrow \mathbb{R}^{%
\mathrm{\dim S}_{j}} \\
\mu _{j} &:&B_{j}(3\rho _{j})\cap S_{j}\longrightarrow \mathbb{R}^{\mathrm{%
\dim S}_{j}}
\end{eqnarray*}%
so that the $\mu _{j}$s are embeddings, and 
\begin{equation*}
p_{j}^{\alpha }\longrightarrow p_{j}
\end{equation*}%
as $\alpha \rightarrow \infty .$

\noindent 2. If $S_{j}$ is the top stratum, then Part 1 holds except that
the $p_{j}^{\alpha }$s are embeddings that are $\tau \left( \delta \right) $%
--almost Riemannian submersions rather than $\varepsilon $--almost
Riemannian submersions.

\noindent 3. Let $S$ be a subset of $Bd\left( N\right) .$ Let $%
B_{k}^{N}\left( \rho ^{N}\right) \in \mathcal{O}^{N}$ and $B_{j\left(
k\right) }^{S}\left( \rho ^{S}\right) \in \mathcal{O}^{S},$ be as in (\ref%
{initialresp eqn}), that is, $B_{k}^{N}\left( 3\rho ^{N}\right) \Subset
B_{j\left( k\right) }^{S}\left( \rho ^{S}\right) $. Then the $\left( \mathrm{%
dim}\left( S\right) +1\right) ^{\mathrm{st}}$ coordinate functions of $\mu
_{k}^{N}$ and $\left( p_{k}^{N}\right) ^{\alpha }$ are the functions $d^{S}$
and $d_{\alpha }^{S}$ from Parts $5$ and $7$ of Theorem \ref{magnum cover
thm}.
\end{theorem}

\begin{remark}
\label{name change}Since $p_{j}^{\alpha }$ is an embedding when $S_{j}$ is
the top stratum, we will write $\mu _{j}^{\alpha }$ for $p_{j}^{\alpha }$ in
this case.
\end{remark}

\begin{proof}
We apply Lemma \ref{k-strained Stability} to the center of each ball of the
open cover of Theorem \ref{magnum cover thm}. By Lemma \ref{k-strained
Stability}, if $\rho $ is sufficiently small, then each $B_{j}(3\rho _{j})$
is contained in a convex set $C_{j}$ of $X,$ and for each $j$ and all but
finitely many $\alpha ,$ there is a convex set $C_{j}^{\alpha }$ with%
\begin{equation*}
C_{j}^{\alpha }\longrightarrow C_{j}.
\end{equation*}

For each $j$ and all but finitely many $\alpha ,$ Part 2 of Lemma \ref%
{k-strained Stability} and its proof give us%
\begin{equation*}
p_{j}^{\alpha }:C_{j}^{\alpha }\longrightarrow \mathbb{R}^{\mathrm{\dim }%
S_{j}}\text{ and }p_{j}:C_{j}\longrightarrow \mathbb{R}^{\mathrm{\dim }%
S_{j}},\text{ with }p_{j}^{\alpha }\longrightarrow p_{j}\text{ as }\alpha
\rightarrow \infty .
\end{equation*}%
By defining $\mu _{j}\equiv p_{j}|_{S},$ we have the desired maps. If $S_{j}$
is not the top stratum, then it follows from Part 2 of Lemma \ref{k-strained
Stability} that $p_{j}^{\alpha }$ and $\mu _{j}$ are $\tau \left( \tilde{%
\delta}\right) +\tau \left( 1/\alpha \text{ }|r\right) $--almost Riemannian
submersions. Since $\tilde{\delta}$ and $1/\alpha $ can be arbitrarily small$%
,$ we can ensure that $p_{j}^{\alpha }$ and $\mu _{j}$ are $\varepsilon $%
--almost Riemannian submersions. By the proof of Theorem 5.4 of \cite{BGP},
the $\mu _{j}$s are embeddings, provided $\rho _{j}$ is also sufficiently
small, establishing Part 1.

The proof of Part 2 is the same, except that we have not assumed that the
top stratum is a Riemannian manifold. Rather we have only assumed that every
point in the top stratum is $\left( n,\delta \right) $--strained. Thus $%
\delta $ cannot be taken to be arbitrarily small, and we can only conclude,
using Lemma \ref{angle convergence}, that $p_{j}^{\alpha }$ and $\mu _{j}$
are $\tau \left( \delta \right) $--almost Riemannian submersions.

To prove Part 3, simply replace the $\left( \mathrm{dim}\left( S\right)
+1\right) ^{\mathrm{st}}$ coordinate functions of $\mu _{k}^{N}$ and $\left(
p_{k}^{N}\right) ^{\alpha }$ with the functions $d^{S}$ and $d_{\alpha }^{S}$
from Parts $5$ and $7$ of Theorem \ref{magnum cover thm}. Since $d^{S}$ and $%
d_{\alpha }^{S}$ are $C^{1}$ close to the functions that they are replacing,
the statements of Parts 1 and 2 continue to hold.
\end{proof}

\begin{remarknonum}
\label{tau to eps rem}In the proof of Part 1 of the previous result, we
exploited the fact that both $\frac{1}{\alpha }$ and the quantity $\tilde{%
\delta}$ from Corollary \ref{k-strained cover thm} can be arbitrarily small.
Using this we replaced each of $\tau \left( \frac{1}{\alpha }|r\right) ,$ $%
\tau \left( \tilde{\delta}\right) ,$ and $\tau \left( \frac{1}{\alpha }%
|r\right) +\tau \left( \tilde{\delta}\right) $ by an arbitrarily small
positive number $\varepsilon .$ For similar reasons, we replaced $\tau
\left( \frac{1}{\alpha }|\rho ,r\right) +\tau \left( \delta \right) $ with $%
\tau \left( \delta \right) $ in the proof of Part 3. The quantities $\tau
\left( \frac{1}{\alpha }|\rho ,r\right) $ and $\tau \left( \frac{1}{\alpha }%
|r\right) $ will appear in the sequel, but only when they are needed to
clarify a link between results that appear prior to and subsequent to this
remark. Whenever such a clarification is not needed, to simplify notation,
we will make the substitutions of the previous proof, that is,%
\begin{eqnarray*}
&&\tau \left( \frac{1}{\alpha }|r\right) +\tau \left( \tilde{\delta}\right) 
\text{ is replaced by }\varepsilon ,\text{ and} \\
&&\tau \left( \frac{1}{\alpha }|\rho ,r\right) +\tau \left( \delta \right) 
\text{ is replaced by }\tau \left( \delta \right)
\end{eqnarray*}%
For the remainder of the paper, $\varepsilon $ is the number from Theorem %
\ref{magnum cover thm}.
\end{remarknonum}

\section{Submersions of Nearby Convex Sets\label{tools to arrange section}}

In this section, we prove Proposition \ref{overlaps C^1 close prop}, which
says that the submersions of Theorem \ref{Perel cap thm} are $C^{1}$--close
on their overlaps. We then prove the analogous result for the top stratum in
Proposition \ref{generic cover cor} (below). Ultimately, these results will
allow us to glue the locally defined maps together via Theorem \ref%
{submersion gluing}.

We start by showing that the submersions of neighboring balls have nearly
the same horizontal spaces.

\begin{lemma}
\label{line up}Let $X$ and $\left\{ M_{\alpha }\right\} _{\alpha }$ be as in
the TNST. For $S\in \mathcal{S},$ let 
\begin{eqnarray*}
p_{s}^{\alpha } &:&B_{s}^{\alpha }(3\rho )\longrightarrow \mathbb{R}^{%
\mathrm{\dim S}}\text{ and } \\
p_{t}^{\alpha } &:&B_{t}^{\alpha }(3\rho )\longrightarrow \mathbb{R}^{%
\mathrm{\dim S}}
\end{eqnarray*}%
be two of the $\varepsilon $--almost Riemannian submersions from Part 1 of
Theorem \ref{Perel cap thm}. At all points of $B_{s}^{\alpha }(3\rho )\cap
B_{t}^{\alpha }(3\rho ),$ the unit spheres in the horizontal spaces of $%
p_{s}^{\alpha }$ and $p_{t}^{\alpha }$ are within $\tau \left( \varepsilon
\right) $ of each other.
\end{lemma}

\begin{proof}
Let the $\left( \mathrm{\dim }\left( S\right) ,\tilde{\delta},r\right) $%
--strainers of $B_{s}(3\rho )$ and $B_{t}(3\rho )$ be $\left\{ \left(
a_{i},b_{i}\right) \right\} _{i=1}^{\mathrm{\dim S}}$ and $\left\{ \left(
c_{i},d_{i}\right) \right\} _{j=1}^{\mathrm{\dim S}},$ respectively. Let $%
\left\{ \left( a_{i}^{\alpha },b_{i}^{\alpha }\right) \right\} _{i=1}^{%
\mathrm{\dim S}}$ and $\left\{ \left( c_{i}^{\alpha },d_{i}^{\alpha }\right)
\right\} _{j=1}^{\mathrm{\dim S}}$ converge to $\left\{ \left(
a_{i},b_{i}\right) \right\} _{i=1}^{\mathrm{\dim S}}$ and $\left\{ \left(
c_{i},d_{i}\right) \right\} _{j=1}^{\mathrm{\dim S}}.$ By considering the
formula for orthogonal projection with respect to an orthonormal basis, we
see that it suffices to show that for $y^{\alpha }\in B_{s}^{\alpha }(3\rho
)\cap B_{t}^{\alpha }(3\rho ),$%
\begin{equation}
\left\vert \left\vert \det \left( \cos \sphericalangle \left( \Uparrow
_{y^{\alpha }}^{a_{i}^{\alpha }},\Uparrow _{y^{\alpha }}^{c_{j}^{\alpha
}}\right) \right) _{i,j}\right\vert -1\right\vert <\varepsilon .
\label{line up in M_ alpha}
\end{equation}

By Proposition \ref{angggl conv prop}, 
\begin{equation}
\left\vert \sphericalangle \left( \Uparrow _{y^{\alpha }}^{a_{i}^{\alpha
}},\Uparrow _{y^{\alpha }}^{c_{j}^{\alpha }}\right) -\sphericalangle \left(
\Uparrow _{y}^{a_{i}},\Uparrow _{y}^{c_{j}}\right) \right\vert <\varepsilon .
\label{change
of basis inequal}
\end{equation}

On the other hand, by Inequality \autoref{alm tang to S inequal}, both $%
\left\{ \Uparrow _{y}^{a_{i}}\right\} _{i=1}^{\mathrm{\dim }\left( S\right)
} $ and $\left\{ \Uparrow _{y}^{c_{j}}\right\} _{j=1}^{\mathrm{\dim }\left(
S\right) }$ are within $\varepsilon $ of $T_{y}S,$ so 
\begin{equation*}
\left\vert \left\vert \det \left( \cos \sphericalangle \left( \Uparrow
_{y}^{a_{i}},\Uparrow _{y}^{c_{j}}\right) \right) _{i,j}\right\vert
-1\right\vert <\tau \left( \varepsilon \right) .
\end{equation*}

The result follows by combining the previous two displays.
\end{proof}

\begin{proposition}
\label{overlaps C^1 close prop}Let $X$ and $\left\{ M_{\alpha }\right\}
_{\alpha }$ be as in the TNST. For $S\in \mathcal{S},$ let $\mathcal{O}^{S}$
be as in Theorem \ref{magnum cover thm}. Let $B\left( S,2\nu \right) $ be
the $2\nu $--neighborhood of $S$ with respect to a fixed metric on $\left(
\amalg _{\alpha }M_{\alpha }\right) \amalg X$ that realizes the
Gromov--Hausdorff convergence. Let 
\begin{eqnarray*}
p_{j}^{\alpha } &:&B_{j}^{\alpha }(3\rho )\longrightarrow \mathbb{R}^{%
\mathrm{\dim S}} \\
\mu _{j} &:&B_{j}(3\rho )\cap S\longrightarrow \mathbb{R}^{\mathrm{\dim S}}
\end{eqnarray*}%
be the $\varepsilon $--almost Riemannian submersions from Theorem \ref{Perel
cap thm}.

Then on $B_{j}^{\alpha }(3\rho )\cap B_{k}^{\alpha }(3\rho )\cap $ $B\left(
S,2\nu \right) ,$ 
\begin{equation}
\left\vert p_{k}^{\alpha }-\mu _{k}\circ \mu _{j}^{-1}\circ p_{j}^{\alpha
}\right\vert _{C^{0}}\leq \tau \left( \frac{1}{\alpha },\nu \right) ,
\label{C-0 convergence}
\end{equation}%
and 
\begin{equation}
\left\vert p_{k}^{\alpha }-\mu _{k}\circ \mu _{j}^{-1}\circ p_{j}^{\alpha
}\right\vert _{C^{1}}\leq \tau \left( \varepsilon \right) .
\label{C-1 convergence}
\end{equation}
\end{proposition}

\begin{proof}
Suppose $y\in B_{j}(3\rho )\cap B_{k}(3\rho )\cap B\left( S,2\nu \right) $, $%
y^{\alpha }\in B_{j}^{\alpha }(3\rho )\cap B_{k}^{\alpha }(3\rho ),$ and $%
y^{\alpha }\rightarrow y.$ Then%
\begin{equation*}
\mathrm{dist}\left( \mu _{j}^{-1}\circ p_{j}^{\alpha }\left( y^{\alpha
}\right) ,y\right) <\tau \left( \frac{1}{\alpha },\nu \right)
\end{equation*}%
and%
\begin{equation*}
\mathrm{dist}\left( \mu _{k}^{-1}\circ p_{k}^{\alpha }\left( y^{\alpha
}\right) ,y\right) <\tau \left( \frac{1}{\alpha },\nu \right) ,
\end{equation*}%
so 
\begin{equation*}
\mathrm{dist}\left( \mu _{j}^{-1}\circ p_{j}^{\alpha }\left( y^{\alpha
}\right) ,\text{ }\mu _{k}^{-1}\circ p_{k}^{\alpha }\left( y^{\alpha
}\right) \right) <\tau \left( \frac{1}{\alpha },\nu \right) .
\end{equation*}%
Since $\mu _{k}$ is $\left( 1+\varepsilon \right) $--bilipschitz, Inequality %
\autoref{C-0 convergence} follows from the previous display.

To make the proof of Inequality \autoref{C-1 convergence} easier to follow,
we change the indices \textquotedblleft $j$\textquotedblright\ and
\textquotedblleft $k$\textquotedblright\ to \textquotedblleft $a$%
\textquotedblright\ and \textquotedblleft $c$\textquotedblright , and prove %
\autoref{C-1 convergence} for submersions $p_{a}^{\alpha }$ and $%
p_{c}^{\alpha }$ and embeddings $\mu _{a}$ and $\mu _{c},$ whose defining
strainers are $\left\{ \left( a_{i}^{\alpha },b_{i}^{\alpha }\right)
\right\} _{i=1}^{n},$ $\left\{ \left( c_{i}^{\alpha },d_{i}^{\alpha }\right)
\right\} _{i=1}^{n},$ $\left\{ \left( a_{i},b_{i}\right) \right\}
_{i=1}^{n}, $ and $\left\{ \left( c_{i},d_{i}\right) \right\} _{i=1}^{n},$
respectively.

We suppose that for all $i,$ 
\begin{eqnarray*}
\mathrm{dist}\left( a_{i},a_{i}^{\alpha }\right) &<&\varepsilon ,\mathrm{dist%
}\left( b_{i},b_{i}^{\alpha }\right) <\varepsilon , \\
\mathrm{dist}\left( c_{i},c_{i}^{\alpha }\right) &<&\varepsilon ,\text{ and }%
\mathrm{dist}\left( d_{i},d_{i}^{\alpha }\right) <\varepsilon .
\end{eqnarray*}

Let $x^{\alpha }$ be any point in the domains of $p_{a}^{\alpha }$ and $%
p_{c}^{\alpha }.$ Let $x\in X$ satisfy $\mathrm{dist}\left( x,x^{\alpha
}\right) <\varepsilon .$

Inequalities \autoref{line up in M_ alpha} and \autoref{change of basis
inequal} give us the hypotheses of Proposition \ref{sim str cor}. Thus given
a unit%
\begin{equation*}
W^{\alpha }\in \mathrm{span}\left\{ \uparrow _{x^{\alpha }}^{a_{i}^{\alpha
}}\right\} _{i=1}^{\dim S},\text{ }
\end{equation*}%
there is a $Y\in T_{x}S$ so that for all $i,$ 
\begin{equation}
\left\vert \sphericalangle \left( W,\uparrow _{x}^{a_{i}}\right)
-\sphericalangle \left( W^{\alpha },\uparrow _{x^{\alpha }}^{a_{i}^{\alpha
}}\right) \right\vert <\tau \left( \varepsilon \right) \text{\label{same
coord 1}}
\end{equation}

and%
\begin{equation}
\left\vert \sphericalangle \left( W,\uparrow _{x}^{c_{i}}\right)
-\sphericalangle \left( W^{\alpha },\uparrow _{x^{\alpha }}^{c_{i}^{\alpha
}}\right) \right\vert <\tau \left( \varepsilon \right) .  \label{sma coord 2}
\end{equation}

Inequality \autoref{same coord 1} gives us%
\begin{equation}
\text{ }\left\vert D\left( \mu _{a}\right) _{x}\left( W\right) -\text{ }%
D\left( p_{a}^{\alpha }\right) _{x^{\alpha }}\left( W^{\alpha }\right)
\right\vert <\tau \left( \varepsilon \right) ,\text{ \label{Dmu vs dP}}
\end{equation}%
and Inequality \autoref{sma coord 2} gives us 
\begin{equation*}
\left\vert D\left( \mu _{c}\right) _{x}\left( W\right) -\text{ }D\left(
p_{c}^{\alpha }\right) _{x^{\alpha }}\left( W^{\alpha }\right) \right\vert
<\tau \left( \varepsilon \right) .
\end{equation*}

Since $D\left( \mu _{c}\circ \mu _{a}^{-1}\right) $ is $\left( 1+\tau \left( 
\tilde{\delta}\right) \right) $--bilipschitz, Inequality \autoref{Dmu vs dP}
gives us 
\begin{equation*}
\text{ }\left\vert D\left( \mu _{c}\right) _{x}\left( W\right) -\text{ }%
D\left( \mu _{c}\circ \mu _{a}^{-1}\circ p_{a}^{\alpha }\right) _{x^{\alpha
}}\left( W^{\alpha }\right) \right\vert <\tau \left( \varepsilon \right) .
\end{equation*}

Inequality \autoref{C-1 convergence} follows by combining the previous two
displays.
\end{proof}

For the top stratum the analogous result is

\begin{proposition}
\label{generic cover cor}Let $X$ and $\left\{ M_{\alpha }\right\} _{\alpha
\in \mathbb{N}}$ be as in Theorem \ref{dif stab- dim 4}. Let $\mathcal{O}%
^{X}=\left\{ B_{j}(3\rho )\right\} _{j}$ be as in Theorem \ref{magnum cover
thm}. The $\tau \left( \delta \right) $--almost Riemannian submersions 
\begin{equation*}
\mu _{j}^{\alpha }:B_{j}^{\alpha }(3\rho )\longrightarrow \mathbb{R}^{n}
\end{equation*}%
of Part 2 of Theorem \ref{Perel cap thm} have the following property.

For $\beta ,\sigma \in \mathbb{N}$ with $\sigma \leq \beta $ and for all $%
j,k,$%
\begin{equation}
\left\vert \mu _{k}^{\sigma }-\mu _{k}^{\beta }\circ \left( \mu _{j}^{\beta
}\right) ^{-1}\circ \mu _{j}^{\sigma }\right\vert _{C^{1}}\leq \tau \left(
\delta \right)  \label{C^1 close gen inequal}
\end{equation}%
and%
\begin{equation}
\left\vert \mu _{k}^{\sigma }-\mu _{k}^{\beta }\circ \left( \mu _{j}^{\beta
}\right) ^{-1}\circ \mu _{j}^{\sigma }\right\vert _{C^{0}}\leq \tau \left( 
\frac{1}{\sigma }|\text{ }r\right)  \label{c0
close gen inequal}
\end{equation}%
on $B_{j}^{\sigma }(3\rho )\cap B_{k}^{\sigma }(3\rho ).$
\end{proposition}

\begin{proof}
Suppose $y\in B_{j}(3\rho )\cap B_{k}(3\rho )$, $y^{\sigma }\in
B_{j}^{\sigma }(3\rho )\cap B_{k}^{\sigma }(3\rho ),$ $y^{\beta }\in
B_{j}^{\beta }(3\rho )\cap B_{k}^{\beta }(3\rho )$, $\mathrm{dist}\left(
y^{\sigma },y\right) <\tau \left( \frac{1}{\sigma }|\text{ }r\right) ,$ and $%
\mathrm{dist}\left( y^{\beta },y\right) <\tau \left( \frac{1}{\sigma }|\text{
}r\right) .$ Then%
\begin{eqnarray}
\left\vert \mu _{j}^{\beta }\left( y^{\beta }\right) -\mu _{j}^{\sigma
}\left( y^{\sigma }\right) \right\vert &<&\tau \left( \frac{1}{\sigma }|%
\text{ }r\right) \text{ and\label{singa C-0 ineqaul}} \\
\left\vert \mu _{k}^{\beta }\left( y^{\beta }\right) -\mu _{k}^{\sigma
}\left( y^{\sigma }\right) \right\vert &<&\tau \left( \frac{1}{\sigma }|%
\text{ }r\right) .  \label{lambda c-0 inequal}
\end{eqnarray}

Since $\mu _{k}^{\beta }\circ \left( \mu _{j}^{\beta }\right) ^{-1}$ is $%
\left( 1+\tau \left( \delta \right) \right) $--Lipschitz, Inequality \autoref%
{singa C-0 ineqaul} gives 
\begin{equation*}
\left\vert \mu _{k}^{\beta }\left( y^{\beta }\right) -\mu _{k}^{\beta }\circ
\left( \mu _{j}^{\beta }\right) ^{-1}\circ \mu _{j}^{\sigma }\left(
y^{\sigma }\right) \right\vert <\tau \left( \frac{1}{\sigma }|\text{ }%
r\right) ,
\end{equation*}%
which, together with Inequality \autoref{lambda c-0 inequal}, gives
Inequality \autoref{c0 close gen inequal}.

Suppose $M,\tilde{M}\in \left\{ M_{\alpha }\right\} _{\alpha \geq \sigma }.$
To make the proof of Inequality \autoref{C^1 close gen inequal} easier to
follow, we change the indices \textquotedblleft $j$\textquotedblright\ and
\textquotedblleft $k$\textquotedblright\ to \textquotedblleft $a$%
\textquotedblright\ and \textquotedblleft $c$\textquotedblright , and prove %
\autoref{C^1 close gen inequal} for coordinate charts $\mu _{a}$ and $\mu
_{c}$ of $M$ and $\tilde{\mu}_{a}$ and $\tilde{\mu}_{c}$ of $\tilde{M}$,
whose defining strainers are $\left\{ \left( a_{i},b_{i}\right) \right\}
_{i=1}^{n},$ $\left\{ \left( c_{i},d_{i}\right) \right\} _{i=1}^{n},$ $%
\left\{ \left( \tilde{a}_{i},\tilde{b}_{i}\right) \right\} _{i=1}^{n}$ , and 
$\left\{ \left( \tilde{c}_{i},\tilde{d}_{i}\right) \right\} _{i=1}^{n},$
respectively.

Suppose that for all $i,$ 
\begin{eqnarray}
\mathrm{dist}\left( a_{i},\tilde{a}_{i}\right) &<&\tau \left( \frac{1}{%
\sigma }|\text{ }r\right) ,\mathrm{dist}\left( b_{i},\tilde{b}_{i}\right)
<\tau \left( \frac{1}{\sigma }|\text{ }r\right) ,  \notag \\
\mathrm{dist}\left( c_{i},\tilde{c}_{i}\right) &<&\tau \left( \frac{1}{%
\sigma }|\text{ }r\right) ,\text{ and }\mathrm{dist}\left( d_{i},\tilde{d}%
_{i}\right) <\tau \left( \frac{1}{\sigma }|\text{ }r\right) .
\label{dtrs
close}
\end{eqnarray}

Suppose also that $y\in M$ is in the domains of both $\mu _{a}$ and $\mu
_{c},$ that $\tilde{y}\in \tilde{M}$ is in the domains of both $\tilde{\mu}%
_{a}$ and $\tilde{\mu}_{c}$, and that $\mathrm{dist}\left( y,\tilde{y}%
\right) <\tau \left( \frac{1}{\sigma }|\text{ }r\right) .$

Proposition \ref{angggl conv prop} and the inequalities in \autoref{dtrs
close} give us the hypotheses of Proposition \ref{sim str cor}. So given a
unit%
\begin{equation*}
W\in \Sigma _{y},
\end{equation*}%
there is a unit 
\begin{equation*}
\tilde{W}\in \Sigma _{\tilde{y}}
\end{equation*}%
so that for all $i,$%
\begin{equation*}
\left\vert \sphericalangle \left( W,\Uparrow _{y}^{a_{i}}\right)
-\sphericalangle \left( \tilde{W},\Uparrow _{\tilde{y}}^{\tilde{a}%
_{i}}\right) \right\vert <\tau \left( \frac{1}{\sigma }|r\right) +\tau
\left( \delta \right)
\end{equation*}%
and%
\begin{equation*}
\left\vert \sphericalangle \left( W,\Uparrow _{y}^{c_{i}}\right)
-\sphericalangle \left( \tilde{W},\Uparrow _{\tilde{y}}^{\tilde{c}%
_{i}}\right) \right\vert <\tau \left( \frac{1}{\sigma }|r\right) +\tau
\left( \delta \right) .
\end{equation*}%
Combining this with the definitions of the $\mu $s, 
\begin{equation}
\left\vert D\tilde{\mu}_{a}\left( \tilde{W}\right) -D\mu _{a}\left( W\right)
\right\vert \leq \tau \left( \delta \right) +\tau \left( \frac{1}{\sigma }%
\text{ }|\text{ }r\right)  \label{singma
Inequal}
\end{equation}%
and%
\begin{equation}
\left\vert D\tilde{\mu}_{c}\left( \tilde{W}\right) -D\mu _{c}\left( W\right)
\right\vert \leq \tau \left( \delta \right) +\tau \left( \frac{1}{\sigma }%
\text{ }|\text{ }r\right) .  \label{tau inequal}
\end{equation}

Since $D\left( \tilde{\mu}_{c}\circ \tilde{\mu}_{a}^{-1}\right) $ is $\left(
1+\tau \left( \delta \right) \right) $--bilipschitz, Inequality \autoref%
{singma Inequal} gives%
\begin{equation*}
\left\vert \left( D\tilde{\mu}_{c}\right) \left( \tilde{W}\right) -D\left( 
\tilde{\mu}_{c}\circ \tilde{\mu}_{a}^{-1}\circ \mu _{a}\right) \left(
W\right) \right\vert \leq \tau \left( \delta \right) +\tau \left( \frac{1}{%
\sigma }\text{ }|\text{ }r\right) .
\end{equation*}

Combined with Inequality \autoref{tau inequal}, this gives%
\begin{equation*}
\left\vert D\mu _{c}\left( W\right) -D\left( \tilde{\mu}_{c}\circ \tilde{\mu}%
_{a}^{-1}\circ \mu _{a}\right) \left( W\right) \right\vert \leq \tau \left(
\delta \right) +\tau \left( \frac{1}{\sigma }\text{ }|\text{ }r\right) .
\end{equation*}

Inequality \autoref{C^1 close gen inequal} follows by recalling that $\tau
\left( \frac{1}{\sigma }\text{ }|\text{ }r\right) $ can be arbitrarily small.
\end{proof}

\section{Gluing $C^{1}$--Close Submersions\label{Gluing Section}%
{\protect\Huge \ }}

In this section we state Theorem \ref{submersion gluing}, an abstract gluing
theorem for submersions, which, together with Proposition \ref{overlaps C^1
close prop}, will allow us to glue together the locally defined submersions
of Theorem \ref{Perel cap thm}. It is based on the principle that a space of
submersions is locally contractible in the $C^{1}$--topology. Since there
are somewhat similar results elsewhere in the literature (cf \cite{Cheeg2}, 
\cite{Kap2}, \cite{Perel1}), we defer the proof of Theorem \ref{submersion
gluing} to the appendix (\ref{Appendix section}). Before stating Theorem \ref%
{submersion gluing}, we establish some background definitions and hypotheses.

\begin{definition}
We say that two collections of sets $\left\{ C_{i}\right\} _{i\in I}$ and $%
\left\{ T_{i}\right\} _{i\in I}$ have the same intersection pattern provided 
$C_{i}\cap C_{j}\neq \emptyset $ if and only if $T_{i}\cap T_{j}\neq
\emptyset .$
\end{definition}

\begin{definition}
If $\mathcal{C}\equiv \left\{ C_{i}\right\} _{i\in I}$ is a collection of
subsets of a space $X,$ we let $\mathrm{cl}\left( \mathcal{C}\right) \equiv
\left\{ \bar{C}_{i}\right\} _{i\in I}$ be the collection of their closures.
\end{definition}

Throughout this section, we assume the following:

\noindent 1. The collection $\mathcal{\tilde{C}}\equiv \left\{ \tilde{B}%
_{i}\left( 3\rho \right) \right\} _{i=1}^{m_{l}}$ of $3\rho $--balls in the
Riemannian $n$--manifold $M$ has first order $\leq \mathfrak{o}$ and
satisfies $\mathrm{dist}\left( \overline{\tilde{B}_{i}\left( \rho \right) },%
\overline{\tilde{B}_{i}\left( 3\rho \right) \setminus \tilde{B}_{i}\left(
2\rho \right) }\right) =\rho .$

\noindent 2. For $\eta \in \left( 0,1\right) $ and $l\geq 1,$ 
\begin{equation*}
\tilde{p}_{i}:\tilde{B}_{i}\left( 3\rho \right) \longrightarrow \mathbb{R}%
^{l}
\end{equation*}%
are $\eta $--almost Riemannian submersions.

\noindent 3. $\mathcal{C}=\left\{ B_{i}\left( \rho \right) \right\}
_{i=1}^{m_{l}}$ is a collection of $\rho $--balls in a Riemannian $l$%
--manifold $S.$

\noindent 4. There are coordinate charts 
\begin{equation*}
\mu _{i}:B_{i}\left( 3\rho \right) \longrightarrow \mathbb{R}^{l}
\end{equation*}%
that are are $\eta $--almost Riemannian submersions.

\noindent 5.\textrm{\ }The collections $\mathcal{C},$ $\mathcal{\tilde{C}},$ 
$\mathrm{cl}\left( \mathcal{C}\right) ,$ and $\mathrm{cl}\left( \mathcal{%
\tilde{C}}\right) $ have the same intersection pattern.

\begin{theorem}
\label{submersion gluing}(Submersion Gluing Theorem) Assume that $M$ and $S$
satisfy Hypotheses 1--5, above.

There are $\xi _{0}\left( \mathfrak{o},l\right) >0,$ $\eta \left( l\right)
>0,$ and $\varepsilon _{0}(l)>0$ with the following property: Suppose that
for all $i,$%
\begin{equation}
\mathrm{dist}_{\mathrm{Haus}}\left( \tilde{p}_{i}\left( \tilde{B}_{i}\left(
\rho \right) \right) ,\mu _{i}\left( B_{i}\left( \rho \right) \right)
\right) <\xi _{0},  \label{Huas inequl}
\end{equation}%
and, for all pairs $\left( i,j\right) ,$ 
\begin{equation}
\left\vert \tilde{p}_{i}-\mu _{i}\circ \mu _{j}^{-1}\circ \tilde{p}%
_{j}\right\vert _{C^{0}}<\xi \leq \xi _{0}  \label{C-0 close inequal}
\end{equation}%
and%
\begin{equation}
\left\vert \tilde{p}_{i}-\mu _{i}\circ \mu _{j}^{-1}\circ \tilde{p}%
_{j}\right\vert _{C^{1}}<\varepsilon \leq \varepsilon _{0}
\label{C-1 close subm inequality}
\end{equation}%
on $\tilde{B}_{i}\left( 3\rho \right) \cap \tilde{B}_{j}\left( 3\rho \right)
.$

Then there is a submersion $P:\cup _{i=1}^{m_{l}}\tilde{B}_{i}\left( \rho
\right) \longrightarrow P\left( \cup _{i=1}^{m_{l}}\tilde{B}_{i}\left( \rho
\right) \right) \subset S$ so that 
\begin{equation}
P|_{\tilde{B}_{m_{l}}\left( \rho \right) }=\mu _{_{m_{l}}}^{-1}\circ \tilde{p%
}_{m_{l}},  \label{last ball eq}
\end{equation}%
and, on each $\tilde{B}_{i}\left( \rho \right) ,$%
\begin{equation}
\left\vert \mu _{i}\circ P-\tilde{p}_{i}\right\vert _{C^{0}}<\tau \left( \xi
\right)  \label{still C-0 close}
\end{equation}%
and 
\begin{equation}
\left\vert \mu _{i}\circ P-\tilde{p}_{i}\right\vert _{C^{1}}<\tau \left(
\varepsilon \right) +\tau \left( \xi |\rho \right) .  \label{geom inequl}
\end{equation}
\end{theorem}

\begin{remark}
In the proof of Theorem \ref{submersion gluing}, we show that the functions $%
\tau $ on the right hand sides of Inequalities (\ref{still C-0 close}) and (%
\ref{geom inequl}) can be taken to be%
\begin{equation*}
\tau \left( \xi \right) =\left( 1+\eta \right) ^{2\mathfrak{o}}\xi \text{ and%
}
\end{equation*}%
\begin{equation*}
\tau \left( \varepsilon \right) +\tau \left( \xi |\rho \right) =\left(
1+\eta \right) ^{2\left( \mathfrak{o}-1\right) }\varepsilon +\frac{2}{\rho }%
\xi \left( \mathfrak{o}-1\right) \left( 1+\eta \right) ^{2\left( \mathfrak{o}%
-1\right) }.
\end{equation*}%
The reader might be more comfortable calling these functions $\tau \left(
\xi |\eta ,\mathfrak{o}\right) $ and $\tau \left( \varepsilon ,\eta |%
\mathfrak{o}\right) +\tau \left( \xi |\eta ,\mathfrak{o,}\rho \right) .$ In
our applications, $\eta $ is small, $\xi <<\eta ,$ and $\mathfrak{o}$ is a
fixed constant that only depends on $X,$ so for simpler notation, we have
chosen to write them as in Theorem \ref{submersion gluing}.
\end{remark}

While Theorem \ref{submersion gluing} is the main abstract gluing tool used
to construct the bundle maps of the TNST, we will also need the following
corollaries to establish Properties $\ref{inter strata resp T6}$ and \ref{T6}
of the TNST.

\begin{corollary}
\label{respt other cor}Let $M$, $S,$ and $P$ be as in Theorem \ref%
{submersion gluing}. Suppose that for some $I_{R}\subset \left\{ 1,2,\ldots
,m_{l}\right\} $, all $i\in I_{R},$ and some $j\in \left\{ 1,\ldots
,l\right\} ,$ the $j^{th}$--coordinate functions of the functions $\tilde{p}%
_{i}$ and $\mu _{i}$ are each respectively given by 
\begin{equation*}
\tilde{d}:\cup _{i\in I_{R}}\tilde{B}_{i}\left( 3\rho _{R}\right)
\longrightarrow \mathbb{R}\text{ and }d:\cup _{i\in I_{R}}B_{i}\left( 3\rho
_{R}\right) \longrightarrow \mathbb{R}\text{ .}
\end{equation*}%
Then we can choose the submersion $P$ from the conclusion of Theorem \ref%
{submersion gluing} so that for all $i\in I_{R},$ the $j^{th}$--coordinate
function of $\mu _{i}\circ P|_{\tilde{B}_{i}\left( 3\rho _{R}\right) }$ is $%
\tilde{d}$.
\end{corollary}

\begin{corollary}
\label{gluing addendum}Let $M,$ $N,$ and $S$ be compact Riemannian manifolds
of dimensions $n\geq k\geq l,$ respectively. Suppose the hypotheses of
Theorem \ref{submersion gluing} hold for $M$ and $S,$ and that for some $%
\rho _{R}>0,$ $\left\{ B_{i}\left( \rho _{R}\right) \right\} _{i=1}^{m_{R}}$
is a collection of $\rho _{R}$ balls in $M,$ so that 
\begin{eqnarray*}
\mathrm{dist}\left( \overline{B_{i}\left( \rho _{R}\right) },\overline{%
B_{i}\left( 3\rho _{R}\right) \setminus B_{i}\left( 2\rho _{R}\right) }%
\right) &=&\rho _{R}\text{ and } \\
\cup _{i\in I_{R}}B_{i}\left( 3\rho _{R}\right) &\subset &\cup _{i=1}^{m_{l}}%
\tilde{B}_{i}\left( \rho \right) ,
\end{eqnarray*}%
where $I_{R}$ is some subset of $\left\{ 1,2,\ldots ,m_{R}\right\} $ for
which the first order of $\left\{ B_{i}\left( 3\rho _{R}\right) \right\}
_{i\in I_{R}}$ is $\leq \mathfrak{o}.$ Then there are $\xi _{0}\left( l,k,%
\mathfrak{o}\right) >0,$ $\eta \left( l,k\right) >0,$ and $\varepsilon
_{0}(l,k)>0$ with the following property.

Suppose that 
\begin{eqnarray*}
R &:&\cup _{i=1}^{m_{R}}B_{i}\left( 3\rho _{R}\right) \longrightarrow N\text{
and} \\
Q &:&N\longrightarrow S
\end{eqnarray*}%
are $\eta $--almost Riemannian submersions so that for each $i=1,2,\ldots
,m_{l},$ on $\cup _{i\in I_{R}}B_{i}\left( 3\rho _{R}\right) \cap \tilde{B}%
_{i}\left( 3\rho \right) ,$ we have 
\begin{equation}
\left\vert \tilde{p}_{i}-\mu _{i}\circ Q\circ R\right\vert _{C^{0}}<\xi \leq
\xi _{0}  \label{resp c-0 close inequal}
\end{equation}%
and%
\begin{equation}
\left\vert \tilde{p}_{i}-\mu _{i}\circ Q\circ R\right\vert
_{C^{1}}<\varepsilon \leq \varepsilon _{0}.  \label{resp c-1 close ineqaul}
\end{equation}%
Then there is a submersion $P:\cup _{i=1}^{m_{l}}\tilde{B}_{i}\left( \rho
\right) \longrightarrow P\left( \cup _{i=1}^{m_{l}}\tilde{B}_{i}\left( \rho
\right) \right) \subset S$ so that on $\cup _{i\in I_{R}}B_{i}\left( \rho
_{R}\right) $ 
\begin{equation}
P=Q\circ R,  \label{respeqn}
\end{equation}%
and, on each $\tilde{B}_{i}\left( \rho \right) ,$%
\begin{equation}
\left\vert \mu _{i}\circ P-\tilde{p}_{i}\right\vert _{C^{0}}<\tau \left( \xi
\right)  \label{resp still C-0 close inequal}
\end{equation}%
and%
\begin{equation}
\left\vert \mu _{i}\circ P-\tilde{p}_{i}\right\vert _{C^{1}}<\tau \left(
\varepsilon \right) +\tau \left( \xi |\rho \right) .
\label{respp stilll C1 cls}
\end{equation}
\end{corollary}

Since Theorem \ref{submersion gluing} and Corollary \ref{gluing addendum}
are similar to other results in the literature, we defer their proofs to the
appendix (\ref{Appendix section}).

\section{Establishing the Tubular Neighborhood Stability Theorem\label{TNST
Section}}

In this section, we complete the proof of Theorem \ref{dif stab- dim 4} by
proving the TNST. Parts \ref{T1}--\ref{T3}, \ref{inter strata resp T6} and %
\ref{T6} are established in Subsection \ref{vector bundles}. Part \ref{T4}
is proven in Subsection \ref{embedding subsect}.

\addtocounter{algorithm}{1}

\subsection{The Disk Bundles of the TNST\label{vector bundles}}

Part \ref{T2} of the TNST is a consequence of the following result.

\begin{proposition}
\label{submersions are glued prop}Let $X$ and $\left\{ M_{\alpha }\right\}
_{\alpha }$ be as in the TNST. Given $\varepsilon >0$ and $S\in \mathcal{S}$
let $\mathcal{O}^{S}=\left\{ B_{j}(\rho )\right\} _{j}$ be the open sets
from Theorem \ref{magnum cover thm}, and let 
\begin{eqnarray*}
p_{j}^{\alpha } &:&B_{j}^{\alpha }(3\rho )\longrightarrow \mathbb{R}^{%
\mathrm{\dim S}}\text{ and} \\
\mu _{j} &:&B_{j}(3\rho )\cap S\longrightarrow \mathbb{R}^{\mathrm{\dim S}}
\end{eqnarray*}%
be the $\varepsilon $--almost Riemannian submersions from Theorem \ref{Perel
cap thm}. If $\frac{1}{\alpha }$ and $\rho $ are sufficiently small, then

\noindent 1. There is a $\mathcal{U}_{\alpha }^{j}\subset M_{\alpha }$ and a
surjective $C^{1}$--disk bundle 
\begin{equation*}
P_{\alpha }^{S}:\mathcal{U}_{\alpha }^{S}\longrightarrow O\subset S
\end{equation*}%
whose fibers have dimension $n\,-\,\dim \left( S\right) $ and which is also
an $\varepsilon $--almost Riemannian submersion$.$ Here $O$ is as in Part 1
of Theorem \ref{magnum cover thm}.

\noindent 2. For $B_{j}(\rho _{j})\in \mathcal{O}^{S},$ 
\begin{equation}
\left\vert \mu _{j}\circ P_{\alpha }^{S}-p_{j}^{\alpha }\right\vert
_{C^{1}}<\tau \left( \varepsilon \right)
\label{C^1 close in k-th step Hypoth}
\end{equation}%
on $B_{j}^{\alpha }(\rho )\cap \mathcal{U}_{\alpha }^{S}.$
\end{proposition}

\begin{proof}
By combining Proposition \ref{overlaps C^1 close prop} with Theorems \ref%
{Perel cap thm} and \ref{submersion gluing}, we get the existence of $%
\widetilde{\mathcal{U}}_{\alpha }^{S}\subset M_{\alpha },$ with 
\begin{equation*}
\mathrm{dist}_{GH}\left( \widetilde{\mathcal{U}}_{\alpha },O\right) <\tau
\left( \frac{1}{\alpha },\nu \right)
\end{equation*}%
and a $\tau \left( \varepsilon \right) $--almost Riemannian submersion%
\begin{equation*}
P_{\alpha }:\widetilde{\mathcal{U}}_{\alpha }\longrightarrow O\subset S
\end{equation*}%
that satisfies Equation \autoref{C^1 close in k-th step Hypoth}. Here $\nu $
is as in Proposition \ref{overlaps C^1 close prop}.

Let $d_{\alpha }^{S}$ be as in Parts 5 and 6 of Theorem \ref{magnum cover
thm}, and set $\mathcal{U}_{\alpha }\equiv \widetilde{\mathcal{U}}_{\alpha
}\cap \left( d_{\alpha }^{S}\right) ^{-1}\left[ 0,10\nu \right] ,$ where $%
\nu $ is as in Part 5 of Theorem \ref{magnum cover thm}. It follows from (%
\ref{alms dist inequ}), (\ref{alm vert gradient}), and (\ref{C^1 close in
k-th step Hypoth}) that the restriction of $P_{\alpha }$ to $\mathcal{U}%
_{\alpha }$ is a submersion. Since a proper submersion is a fiber bundle, $%
\left. P_{\alpha }\right\vert _{\mathcal{U}_{\alpha }}$ is a fiber bundle.

If the $\rho $s are small enough, then some of our local submersions are the
restriction of the maps $p_{\mathrm{conv}}^{\alpha }$ from Part 2 of Lemma %
\ref{k-strained Stability}. In particular, these local submersions have disk
fibers. Now order the $B_{j}^{\alpha }(3\rho )$ so that for the last ball,
the corresponding submersion $p_{\mathrm{last}}$ has disk fibers. It follows
from Equation \autoref{last ball eq} that the fibers of $P_{\alpha }$ agree
with those of $p_{\mathrm{last}}$ on this last ball. Hence a fiber of $%
\left. P_{\alpha }\right\vert _{\mathcal{U}_{\alpha }}$ is a disk. Thus $%
P_{\alpha }|_{\mathcal{U}_{\alpha }}$ is a fiber bundle with fiber $\mathbb{D%
}^{n-l},$ where $l=\dim \left( S\right) .$
\end{proof}

\begin{proof}[Proof of Part \protect\ref{T1} of the TNST]
Combine the construction of the $\mathcal{U}_{\gamma }^{S}$s with the
hypothesis that the elements of $\mathcal{S}$ are pairwise disjoint and the
fact that Theorem \ref{magnum cover thm} holds for all sufficiently small $%
\rho $.
\end{proof}

\begin{proof}[Proof of Part \protect\ref{T3} of the TNST]
Set 
\begin{equation*}
\mathcal{U}_{\alpha }^{S}\left( t\right) \equiv \left( d_{\alpha
}^{S}\right) ^{-1}\left[ 0,t\nu \right]
\end{equation*}%
and appeal to the proof of Proposition \ref{submersions are glued prop}.
\end{proof}

\begin{proof}[Proof of Part \protect\ref{inter strata resp T6} of the TNST]
Via an argument nearly identical to the proof of Proposition \ref%
{submersions are glued prop}, we construct the submersions 
\begin{equation*}
Q^{S_{j}}:\mathcal{V}^{S_{j}}\setminus S_{j}\longrightarrow S_{j}.
\end{equation*}%
To get Equation \autoref{gen resp EEEEqn}, we combine Corollaries \ref{old
part 5 cor} and \ref{gluing addendum}.
\end{proof}

\begin{proof}[Proof of Part \protect\ref{T6} of the TNST]
Suppose that $S$ has Ancestor Number 2 and is in $Bd\left( N\right) .$ Then
by Part 3 of Theorem \ref{Perel cap thm}, on $\cup \mathcal{O}^{N}\cap
\left( \cup \mathcal{O}^{S}\setminus B\left( S,\nu \right) \right) ,$ the $%
\left( \mathrm{\dim }\left( S\right) +1\right) ^{st}$--coordinate functions
of all of the $\left( p_{j}^{N}\right) ^{\alpha }$ is the function $%
d_{\alpha }^{S}$ from Part 5 of Theorem \ref{magnum cover thm}.
Additionally, the $\left( \mathrm{\dim }\left( S\right) +1\right) ^{st}$%
--coordinate function of all of the $\mu _{k}^{N}$ is the function $d^{S}$
from Part 7 of Theorem \ref{magnum cover thm}. So by Corollary \ref{respt
other cor}, the $\left( \mathrm{\dim }\left( S\right) +1\right) ^{st}$%
--coordinate function of $\mu _{k}^{N}\circ P_{\alpha }^{N}$ is $d_{\alpha
}^{S}$. Part \ref{T6} of the TNST follows from this and the fact that $%
\mathcal{U}_{\alpha }^{S}\left( 3\right) \equiv \left( d_{\alpha
}^{S}\right) ^{-1}\left[ 0,3\nu \right] .$
\end{proof}

\addtocounter{algorithm}{1}

\subsection{The Embeddings of the TNST\label{embedding subsect}}

Part \ref{T4} of the TNST follows from the next result, wherein we construct
the embedding $\Phi _{\beta ,\alpha }:G_{\alpha }\longrightarrow M_{\beta }$
of the Tubular Neighborhood Stability Theorem. The existence of an embedding 
$G_{\alpha }\longrightarrow M_{\beta }$ is a consequence of Theorem 6.1 in 
\cite{KMS}. To prove Part 4 of the TNST, we also need to show that $\Phi
_{\beta ,\alpha }$ satisfies Equations \autoref{top stratum resp eqn}, (\ref%
{containment}), and (\ref{schn 2}). This is achieved via an appeal to
Corollaries \ref{old part 5 cor} and \ref{gluing addendum}.

\begin{proposition}
\label{generic embedding prop}Let $X$ and $\left\{ M_{\alpha }\right\}
_{\alpha }$ be as in the TNST.

\noindent 1. Set 
\begin{equation*}
G_{\alpha }\equiv M_{\alpha }\setminus \cup _{S\in \boldsymbol{S}}\mathcal{U}%
_{\alpha }^{S}\left( 1\right) .
\end{equation*}%
There is a $C^{1},$ $\tau \left( \frac{1}{\alpha },\frac{1}{\beta }\right) $%
--embedding 
\begin{equation*}
\Phi _{\beta ,\alpha }:G_{\alpha }\longrightarrow M_{\beta }
\end{equation*}%
so that for all $S\in \mathcal{S},$%
\begin{equation}
P_{\alpha }^{S}=P_{\beta }^{S}\circ \Phi _{\beta ,\alpha },
\label{respect baby}
\end{equation}%
wherever both expressions are defined.

\noindent 2. In addition, we may choose $\Phi _{\beta ,\alpha }$ so that for
all $S\in \mathcal{S},$ 
\begin{equation*}
\Phi _{\beta ,\alpha }\left( \partial \mathcal{U}_{\alpha }^{S}\left(
3\right) \cap G_{\alpha }\right) =\partial \mathcal{U}_{\beta }^{S}\left(
3\right) \cap G_{\beta }.
\end{equation*}
\end{proposition}

Note that if $N\in \mathcal{S}$ has ancestor number $1,$ then $\partial 
\mathcal{U}_{\alpha }^{N}\left( 3\right) \subset G_{\alpha }.$ Thus (\ref%
{containment}) and (\ref{schn 2}) follow from Part 2 of the previous result.

\begin{proof}
Using Corollaries \ref{old part 5 cor} and \ref{gluing addendum}, we glue
the embeddings $\left( \mu _{j}^{\beta }\right) ^{-1}\circ \mu _{j}^{\alpha
} $ of Proposition \ref{generic cover cor} to get an immersion 
\begin{equation*}
\Phi _{\beta ,\alpha }:G_{\alpha }\longrightarrow \Phi _{\beta ,\alpha
}\left( G_{\alpha }\right) \subset M_{\beta }
\end{equation*}%
so that for all $S\in \mathcal{S},$ 
\begin{equation*}
P_{\alpha }^{S}=P_{\beta }^{S}\circ \Phi _{\beta ,\alpha },
\end{equation*}%
wherever both expressions are defined.

It follows from Inequalities \autoref{c0 close gen inequal} and \autoref%
{resp still C-0 close inequal} that $\Phi _{\beta ,\alpha }$ is also a $\tau
\left( \frac{1}{\alpha },\frac{1}{\beta }\right) $--Hausdorff approximation.
From Inequalities \autoref{C^1 close gen inequal}, \autoref{c0 close gen
inequal}, and \autoref{respp stilll C1 cls}, it follows that on $%
B_{j}^{\alpha }(\rho ),$ 
\begin{equation}
\left\vert \mu _{j}^{\beta }\circ \Phi _{\beta ,\alpha }-\mu _{j}^{\alpha
}\right\vert _{C^{1}}\leq \tau \left( \delta \right) .
\label{not much change}
\end{equation}

Combining this with the fact that $\mu _{j}^{\alpha }$ and $\mu _{j}^{\beta
} $ are $\left( \tau \left( \delta \right) \right) $--almost Riemannian
embeddings, we see that $\left. \Phi _{\beta ,\alpha }\right\vert
_{B_{j}^{\alpha }(\rho )}$ is one-to-one. Because $\Phi _{\beta ,\alpha }$
is also a $\tau \left( \frac{1}{\alpha },\frac{1}{\beta }\right) $%
--Hausdorff approximation, it is one-to-one if $\alpha $ and $\beta $ are
sufficiently large. Since $\dim \left( M^{\alpha }\right) =\dim \left(
M^{\beta }\right) ,$ $\Phi _{\beta ,\alpha }$ is an embedding.

To prove Part 2, for $S\in \mathcal{S}$ and $\gamma =\alpha $ or $\beta ,$
let $d_{\gamma }^{S}$ be the smooth function from Part 5 of Theorem \ref%
{magnum cover thm}. It follows from (\ref{not much change}) that the $k^{th}$%
--coordinate functions of the $\mu _{j}^{\gamma }$s satisfy 
\begin{equation*}
\left\vert d\Phi _{\beta ,\alpha }\left( \nabla \left( \mu _{j}^{\alpha
}\right) ^{k}\right) -\nabla \left( \mu _{j}^{\beta }\right) ^{k}\right\vert
\leq \tau \left( \delta \right) .
\end{equation*}%
Combining this with Part 3 of Theorem \ref{Perel cap thm}, 
\begin{equation*}
\left\vert d\Phi _{\beta ,\alpha }\left( \nabla d_{\alpha }^{S}\right)
-\nabla d_{\beta }^{S}\right\vert \leq \tau \left( \delta \right) .
\end{equation*}%
Together with (\ref{alms dist inequ}), this gives us that $\nabla d_{\alpha
}^{S}$ is gradient-like for $d_{\beta }.$ It follows that there is a
nonvanishing vector field $W$ on 
\begin{equation*}
\mathcal{U}_{\beta }^{S}\left( 10\right) \setminus \mathcal{U}_{\beta
}^{S}\left( 1\right)
\end{equation*}%
so that 
\begin{equation*}
W=\left\{ 
\begin{array}{ll}
d\Phi _{\beta ,\alpha }\left( \nabla d_{\alpha }^{S}\right) & \text{ near
the boundary of }\Phi _{\beta ,\alpha }\left( M_{\alpha }\setminus \left\{
\cup _{i}\mathcal{U}_{\alpha }^{i}\left( 3\right) \right\} \right) \\ 
\nabla d_{\beta }^{S} & \text{near the boundary of }M_{\beta }\setminus
\left\{ \cup _{i}\mathcal{U}_{\beta }^{i}\left( 2\right) \right\} .%
\end{array}%
\right.
\end{equation*}

Since $\mathcal{U}_{\gamma }^{S}\left( t\right) \equiv \left( d_{\gamma
}^{S}\right) ^{-1}\left[ 0,t\nu \right] ,$ $W$ is transverse to the
boundaries of $\Phi _{\beta ,\alpha }\left( M_{\alpha }\setminus \left\{
\cup _{i}\mathcal{U}_{\alpha }^{i}\left( 3\right) \right\} \right) $ and $%
M_{\beta }\setminus \left\{ \cup _{i}\mathcal{U}_{\beta }^{i}\left( 2\right)
\right\} .$ It follows from (\ref{alm str ineqaul}) and (\ref{C^1 close in
k-th step Hypoth}) that $\nabla d_{\alpha }^{S}$ and $\nabla d_{\beta }^{S}$
are nearly vertical for $P_{\alpha }^{S_{i}}$ and $P_{\beta }^{S_{i}}.$
Combined with (\ref{respect baby}) and (\ref{not much change}) it follows
that $d\Phi _{\beta ,\alpha }\left( \nabla d_{\alpha }^{S}\right) $ is
nearly vertical for $P_{\beta }^{S_{i}}.$ Thus $W$ and its $P_{\beta
}^{S_{i}}$--vertical component, $W^{V},$ are nearly the same field. So $%
W^{V} $ is transverse to the boundaries of $\Phi _{\beta ,\alpha }\left(
M_{\alpha }\setminus \left\{ \cup _{i}\mathcal{U}_{\alpha }^{i}\left(
3\right) \right\} \right) $ and $M_{\beta }\setminus \left\{ \cup _{i}%
\mathcal{U}_{\beta }^{i}\left( 2\right) \right\} .$ Since $\Phi _{\beta
,\alpha }$ is a $\tau \left( \frac{1}{\alpha },\frac{1}{\beta }\right) $%
--Hausdorff approximation,%
\begin{equation*}
\Phi _{\beta ,\alpha }\left( M_{\alpha }\setminus \cup _{i}\mathcal{U}%
_{\alpha }^{S_{i}}\left( 3\right) \right) \subset M_{\beta }\setminus \cup
_{i}\mathcal{U}_{\beta }^{S_{i}}\left( 2\right) .
\end{equation*}%
Using a reparameterization of the flow of $W^{V},$ we construct a
diffeomorphism $\Upsilon $ that carries $\Phi _{\beta ,\alpha }\left(
M_{\alpha }\setminus \left\{ \cup _{i}\mathcal{U}_{\alpha }^{i}\left(
3\right) \right\} \right) $ to $M_{\beta }\setminus \left\{ \cup _{i}%
\mathcal{U}_{\beta }^{i}\left( 2\right) \right\} .$ We abuse notation and
call $\Upsilon \circ \Phi _{\beta ,\alpha },$ $\Phi _{\beta ,\alpha }$. It
follows that $\Phi _{\beta ,\alpha }\left( M_{\alpha }\setminus \cup _{i}%
\mathcal{U}_{\alpha }^{S_{i}}\left( 3\right) \right) =M_{\beta }\setminus
\cup _{i}\mathcal{U}_{\beta }^{S_{i}}\left( 2\right) ,$ and, after modifying
the parameterization of our disk bundles, 
\begin{equation*}
\Phi _{\beta ,\alpha }\left( M_{\alpha }\setminus \cup _{i}\mathcal{U}%
_{\alpha }^{S_{i}}\left( 3\right) \right) =M_{\beta }\setminus \cup _{i}%
\mathcal{U}_{\beta }^{S_{i}}\left( 3\right) ,
\end{equation*}%
so Part 2 holds. Since $W^{V}$ is vertical for $P_{\beta }^{S},$ $\Phi
_{\beta ,\alpha }$ continues to satisfy Equation (\ref{respect baby}).
\end{proof}

This completes the proof of Theorem \ref{dif stab- dim 4}, modulo the proofs
of Theorem \ref{submersion gluing} and Corollaries \ref{respt other cor} and %
\ref{gluing addendum}.

\section{Appendix A: How to Glue $C^{1}$--Close Submersions\label{Appendix
section}}

In this section we prove Theorem \ref{submersion gluing} and Corollary \ref%
{gluing addendum}. Before doing so we establish several inductive gluing
tools in Subsubsection \ref{ind gluing tools}, and we prove a result about
stability of intersection patterns in Subsubsection \ref{inter stab subsub}.

\addtocounter{algorithm}{1}

\subsection{Tools to \label{ind gluing tools}Glue $C^{1}$--Close Submersions}

In this subsection we prove Key Lemma \ref{Gluing, Abstract}, the main
inductive gluing lemma that will allow us to prove Theorem \ref{submersion
gluing}. First we establish several preliminary results.

\begin{lemma}
\label{Rank Isotopy Lemma}(Submersion Isotopy Lemma) Let $G\subset M$ be an
open subset of a Riemannian $n$--manifold $M$. Let $\pi :G\rightarrow 
\mathbb{R}^{l}$ be an $\eta $--almost Riemannian submersion, and let $%
p:G\rightarrow \mathbb{R}^{l}$ be any submersion with%
\begin{equation*}
\left\vert p-\pi \right\vert _{C^{1}}<\varepsilon .
\end{equation*}%
There are positive numbers $\eta _{1}$ and $\varepsilon _{1}$ that only
depend on $l$ so that if $\eta \in \left( 0,\eta _{1}\right) $ and $%
\varepsilon \in \left( 0,\varepsilon _{1}\right) ,$ then the homotopy $%
H:G\times \lbrack 0,1]\rightarrow \mathbb{R}^{l},$%
\begin{equation*}
H_{t}\equiv \pi +t(p-\pi ),
\end{equation*}%
from $p$ to $\pi $ has the following properties.

\begin{enumerate}
\item $H_{t}$ is a submersion.

\item $\left\vert H_{t}-\pi \right\vert _{C^{1}}<\varepsilon $ and $%
\left\vert H_{t}-p\right\vert _{C^{1}}<\varepsilon .$

\item $|H_{t}-\pi |_{C^{0}}\leq |p-\pi |_{C^{0}}$ and $|H_{t}-p|_{C^{0}}\leq
|p-\pi |_{C^{0}}.$

\item If $Z\subset G$ is open and $q:Z\rightarrow \mathbb{R}^{l}$ is a
submersion with $\left\vert q-\pi \right\vert _{C^{1}}<\varepsilon $ and $%
\left\vert q-p\right\vert _{C^{1}}<\varepsilon ,$ then $\left\vert
H_{t}-q\right\vert _{C^{1}}<\varepsilon .$

\item If $Z\subset G$ is open and $q:Z\rightarrow \mathbb{R}^{l}$ is a
submersion with $\left\vert q-\pi \right\vert _{C^{0}}<\xi $ and $\left\vert
q-p\right\vert _{C^{0}}<\xi ,$ then $\left\vert H_{t}-q\right\vert
_{C^{0}}<\xi .$

\item $|DH_{\left( x,t\right) }\left( 0,\frac{\partial }{\partial t}\right)
|\leq |p-\pi |_{C^{0}}$.

\item If $F$ is a subset of $G$ with $\pi _{k}\circ p|_{F}=\pi _{k}\circ \pi
|_{F},$ then $\pi _{k}\circ H_{t}|_{F}=\pi _{k}\circ p|_{F}=\pi _{k}\circ
\pi |_{F}$ for all $t.$ Here $\pi _{k}:\mathbb{R}^{l}\rightarrow \mathbb{R}%
^{k}$ is projection to the first $k$--factors.
\end{enumerate}
\end{lemma}

\begin{proof}
There is an $\varepsilon _{\mathrm{Riem}}>0$ so that for any Riemannian
submersion $\pi _{\mathrm{Riem}}:G\rightarrow \mathbb{R}^{l},$ any map $%
h:G\rightarrow \mathbb{R}^{l}$ is a submersion, provided 
\begin{equation*}
\left\vert h-\pi _{\mathrm{Riem}}\right\vert _{C^{1}}<\varepsilon _{\mathrm{%
Riem}}.
\end{equation*}

Take $\eta _{1}=\varepsilon _{1}=\frac{\varepsilon _{\mathrm{Riem}}}{2}.$
Then any map $h:G\rightarrow \mathbb{R}^{l}$ is a submersion provided 
\begin{equation*}
\left\vert h-\pi \right\vert _{C^{1}}<\varepsilon _{1}.
\end{equation*}%
Since 
\begin{equation*}
H_{t}\equiv \pi +t(p-\pi ),
\end{equation*}%
Conclusions 2, 3, 4, and 5 follow from convexity of balls in Euclidean
space, and Conclusion 7 follows from the definition of $H.$ Conclusion 1
follows from Conclusion 2 and our choice of $\varepsilon .$ Conclusion 6
follows from 
\begin{equation*}
DH_{\left( x,t\right) }\left( 0,\frac{\partial }{\partial t}\right) =p\left(
x\right) -\pi \left( x\right) .
\end{equation*}
\end{proof}

\begin{lemma}
\label{Urysohn}For $\zeta >0$, let $W\Subset V\Subset G\subset M$ be three
nonempty, open, pre-compact sets that satisfy 
\begin{equation*}
\mathrm{dist(}\overline{W},\overline{G\setminus V}\mathrm{)}>\zeta .
\end{equation*}%
There is a $C^{\infty }$ function $\omega :G\longrightarrow \left[ 0,1\right]
$ that satisfies

\begin{enumerate}
\item 
\begin{equation*}
\omega \left( x\right) =\left\{ 
\begin{array}{cc}
0 & \text{for }x\in W \\ 
1 & \text{for }x\in G\setminus V%
\end{array}%
\right.
\end{equation*}

\item 
\begin{equation*}
\left\vert \nabla \omega \right\vert \leq \frac{2}{\zeta }.
\end{equation*}
\end{enumerate}
\end{lemma}

\begin{proof}
Approximate $\mathrm{dist}(\overline{W},\cdot )$ and $\mathrm{dist}(%
\overline{G\setminus V},\cdot )$ by smooth functions in the $C^{0}$%
-topology. Choose sublevels $C_{1}$ and $C_{2}$ of these approximations so
that $W\Subset C_{1}$, $G\setminus V\Subset C_{2},$ and $\mathrm{dist}%
(C_{1},C_{2})>\zeta $. Using the techniques of \cite{GrovShio,GrWu2},
approximate $\mathrm{dist}(C_{i},\cdot )$ by smooth functions $f_{C_{i}}$
that satisfy $f_{C_{i}}\geq 0,$ $|\nabla f_{C_{i}}|\leq 2$, and $%
f_{C_{i}}|_{C_{i}}\equiv 0$. Since 
\begin{equation*}
\mathrm{dist}(C_{1},x)+\mathrm{dist}(C_{2},x)\geq \mathrm{dist}%
(C_{1},C_{2})>\zeta ,
\end{equation*}%
and the technique of \cite{GrovShio,GrWu2} allows the approximation to be as
close as we please in the $C^{0}$--topology, we can choose the $f_{C_{i}}$s
so that they also satisfy 
\begin{equation*}
f_{C_{1}}+f_{C_{2}}>\zeta .
\end{equation*}

Then the function 
\begin{equation*}
\omega \equiv \frac{f_{C_{1}}}{f_{C_{1}}+f_{C_{2}}}
\end{equation*}%
satisfies Property 1. Moreover, 
\begin{eqnarray*}
\left\vert \nabla \omega \right\vert &=&\left\vert \frac{\left(
f_{C_{1}}+f_{C_{2}}\right) \nabla f_{C_{1}}-f_{C_{1}}\nabla \left(
f_{C_{1}}+f_{C_{2}}\right) }{\left( f_{C_{1}}+f_{C_{2}}\right) ^{2}}%
\right\vert \\
&=&\left\vert \frac{f_{C_{2}}\nabla f_{C_{1}}-f_{C_{1}}\nabla f_{C_{2}}}{%
\left( f_{C_{1}}+f_{C_{2}}\right) ^{2}}\right\vert \\
&\leq &2\frac{f_{C_{2}}+f_{C_{1}}}{\left( f_{C_{1}}+f_{C_{2}}\right) ^{2}} \\
&\leq &\frac{2}{\zeta },
\end{eqnarray*}%
as claimed.
\end{proof}

\begin{lemma}
\label{const rank def}(Submersion Deformation Lemma) Let $W\Subset V\Subset
G\subset M$ satisfy the hypotheses of Lemma \ref{Urysohn}, and let $\omega
:G\longrightarrow \left[ 0,1\right] $ be as in the conclusion of Lemma \ref%
{Urysohn}. Let $\pi :G\rightarrow \mathbb{R}^{l}$ be an $\eta $--almost
Riemannian submersion, where $\eta $ is as in Lemma \ref{Rank Isotopy Lemma}%
. Let $p:G\rightarrow \mathbb{R}^{l}$ be a submersion satisfying 
\begin{equation*}
\left\vert p-\pi \right\vert _{C^{1}}<\varepsilon \text{ and }\left\vert
p-\pi \right\vert _{C^{0}}<\xi <\varepsilon ,
\end{equation*}%
and let $\varepsilon _{1}$ be as in Lemma \ref{Rank Isotopy Lemma}.

If $0<\varepsilon +\frac{2\left\vert p-\pi \right\vert _{C^{0}}}{\zeta }%
<\varepsilon _{1},$ then the map $\psi :G\rightarrow \mathbb{R}^{l}$%
\begin{equation*}
\psi \left( x\right) =\pi \left( x\right) +\omega (x)\cdot (p-\pi )\left(
x\right)
\end{equation*}%
is a submersion with the following properties.

\begin{enumerate}
\item 
\begin{equation*}
\psi =\left\{ 
\begin{array}{cc}
\pi & \text{on }W \\ 
p & \text{on }G\setminus V%
\end{array}%
\right.
\end{equation*}

\item 
\begin{equation*}
\left\vert \psi -\pi \right\vert _{C^{1}}<\varepsilon +\frac{2\left\vert
p-\pi \right\vert _{C^{0}}}{\zeta }\text{ and }\left\vert \psi -p\right\vert
_{C^{1}}<\varepsilon +\frac{2\left\vert p-\pi \right\vert _{C^{0}}}{\zeta }
\end{equation*}

\item If $U\subset G$ is open and $q:U\rightarrow \mathbb{R}^{l}$ is a
submersion with $\left\vert q-\pi \right\vert _{C^{1}}<\varepsilon $ and $%
\left\vert q-p\right\vert _{C^{1}}<\varepsilon ,$ then $\left\vert \psi
-q\right\vert _{C^{1}}<\varepsilon +\frac{2\left\vert p-\pi \right\vert
_{C^{0}}}{\zeta }.$

\item If $U\subset G$ is open and $q:U\rightarrow \mathbb{R}^{l}$ is a
submersion with $\left\vert q-\pi \right\vert _{C^{0}}<\xi $ and $\left\vert
q-p\right\vert _{C^{0}}<\xi ,$ then $\left\vert \psi -q\right\vert
_{C^{0}}<\xi .$

\item If $F$ is a subset of $G$ with $\pi _{k}\circ p|_{F}=\pi _{k}\circ
\varphi |_{F},$ then $\pi _{k}\circ \psi |_{F}=\pi _{k}\circ p|_{F}=\pi
_{k}\circ \varphi |_{F},$ where $\pi _{k}:\mathbb{R}^{l}\rightarrow \mathbb{R%
}^{k}$ is projection to the first $k$--factors.
\end{enumerate}
\end{lemma}

\begin{proof}
Part 1 is a consequence of the definitions of $\psi $ and $\omega .$

Let $H_{t}:G\rightarrow \mathbb{R}^{l}$ be the isotopy from Lemma \ref{Rank
Isotopy Lemma}. Since $\psi (x)=H_{\omega (x)}(x)$, Parts 4 and 5 follow
from Parts 5 and 7 of Lemma \ref{Rank Isotopy Lemma}.

For any $x\in G$ and any $v\in T_{x}M,$%
\begin{equation}
D\psi _{x}\left( v\right) =D\pi _{x}\left( v\right) +\omega \left( x\right)
D\left( p-\pi \right) _{x}\left( v\right) +\left\langle \nabla \omega
,v\right\rangle \left( p-\pi \right) \left( x\right) .  \label{Dpsi eqn}
\end{equation}%
Since $\left\vert p-\pi \right\vert _{C^{1}}<\varepsilon ,$ $\left\vert
\omega \right\vert \leq 1,$ and $\left\vert \nabla \omega \right\vert \leq 
\frac{2}{\zeta },$ 
\begin{eqnarray*}
\left\vert D\psi _{x}-D\pi _{x}\right\vert &\leq &\varepsilon +\left\vert
\nabla \omega \right\vert |p-\pi |_{C^{0}} \\
&\leq &\varepsilon +\frac{2|p-\pi |_{C^{0}}}{\zeta }.
\end{eqnarray*}%
By rewriting $\psi $ as $\psi =p+\left( 1-\omega \right) \cdot (\pi -p),$ a
similar argument gives 
\begin{equation*}
\left\vert D\psi _{x}-Dp_{x}\right\vert \leq \varepsilon +\frac{2|p-\pi
|_{C^{0}}}{\zeta }.
\end{equation*}%
Combining the previous two displays gives us Part 2.

If $q$ is as in Part 3, then by Part 4 of Lemma \ref{Rank Isotopy Lemma}, 
\begin{equation*}
\left\vert \left( Dq\right) _{x}-\left( D\pi _{x}+\omega \left( x\right)
D\left( p-\pi \right) _{x}\right) \right\vert <\varepsilon .
\end{equation*}%
Combined with Equation \autoref{Dpsi eqn} this gives us Part 3.

Combining Part 2 with our hypothesis that $\varepsilon +\frac{2\left\vert
p-\pi \right\vert _{C^{0}}}{\zeta }<\varepsilon _{1},$ we see that $\psi $
is a submersion.
\end{proof}

\begin{keylemma}
\label{Gluing, Abstract}Let $\tilde{M}$ and $S$ be compact Riemannian
manifolds. Let 
\begin{eqnarray*}
\tilde{W}\Subset \tilde{V}\Subset \tilde{G},\tilde{O} &\subset &\tilde{M}%
\text{ and} \\
G,O &\subset &S
\end{eqnarray*}%
be pre-compact open sets with 
\begin{equation*}
\mathrm{dist(closure}\left( \tilde{W}\right) ,\mathrm{closure}\left( \tilde{G%
}\setminus \tilde{V}\right) \mathrm{)}>\zeta \text{.}
\end{equation*}%
Let $p_{O}:\tilde{O}\longrightarrow O,$ $p_{G}:\tilde{G}\longrightarrow G$
and $\mu :G\longrightarrow \mathbb{R}^{l}$ be $\eta $--almost Riemannian
submersions with $\mu $ a coordinate chart.

Suppose $\tilde{W}\cap \tilde{O}\neq \emptyset $, $p_{G}\left( \tilde{W}%
\right) \cap p_{O}\left( \tilde{O}\right) \neq \emptyset ,$ and the
restrictions of $p_{O}$ and $p_{G}$ to $\tilde{O}\cap \tilde{G}$ satisfy 
\begin{equation}
\left\vert \mu \circ p_{O}-\mu \circ p_{G}\right\vert _{C^{1}}<\varepsilon 
\text{ and\label{pi_l vs P_0}}
\end{equation}%
\begin{equation}
\left\vert \mu \circ p_{O}-\mu \circ p_{G}\right\vert _{C^{0}}<\xi
<\varepsilon ,  \label{pi_i vs P_0--C^0}
\end{equation}%
where $\varepsilon +\frac{2\xi }{\zeta }<\varepsilon _{1},$ and $\varepsilon
_{1}$ is as in Lemma \ref{Rank Isotopy Lemma}.

Then there is a submersion 
\begin{equation*}
P:\tilde{W}\cup \tilde{O}\longrightarrow P\left( \tilde{W}\cup \tilde{O}%
\right) \subset S
\end{equation*}%
so that%
\begin{equation}
P=\left\{ 
\begin{array}{cc}
p_{G} & \mathrm{on}\,\,\,\tilde{W} \\ 
p_{O} & \mathrm{on}\,\,\,\tilde{O}\setminus \tilde{V},%
\end{array}%
\right.  \label{pieces of P eqn}
\end{equation}%
and in addition, the following hold.

\noindent 1. On $\tilde{G}\cap \tilde{O},$%
\begin{equation*}
\left\vert \mu \circ P-\mu \circ p_{G}\right\vert _{C^{1}}<\varepsilon +%
\frac{2\xi }{\zeta }\text{ and }\left\vert \mu \circ P-\mu \circ
p_{O}\right\vert _{C^{1}}<\varepsilon +\frac{2\xi }{\zeta }.
\end{equation*}%
\noindent 2. If $\tilde{U}\subset \tilde{G}\cap \tilde{O}$ is open and $q:%
\tilde{U}\longrightarrow S$ is a submersion with $\left\vert \mu \circ q-\mu
\circ p_{G}\right\vert _{C^{1}}<\varepsilon $ and $\left\vert \mu \circ
q-\mu \circ p_{O}\right\vert _{C^{1}}<\varepsilon ,$ then $\left\vert \mu
\circ P-\mu \circ q\right\vert _{C^{1}}<\varepsilon +\frac{2\xi }{\zeta }.$

\noindent 3. If $\tilde{U}\subset \tilde{G}\cap \tilde{O}$ is open and $q:%
\tilde{U}\longrightarrow S$ is a submersion with $\left\vert \mu \circ q-\mu
\circ p_{G}\right\vert _{C^{0}}<\xi $ and $\left\vert \mu \circ q-\mu \circ
p_{O}\right\vert _{C^{0}}<\xi ,$ then $\left\vert \mu \circ P-\mu \circ
q\right\vert _{C^{0}}<\xi .$

\noindent 4. If $F$ is a subset of $\tilde{O}\cap \tilde{G}$ with $\pi
_{k}\circ \mu \circ p_{O}|_{F}=\pi _{k}\circ \mu \circ p_{G}|_{F},$ then $%
\pi _{k}\circ \mu \circ P|_{F}=\pi _{k}\circ \mu \circ p_{O}|_{F}=\pi
_{k}\circ \mu \circ p_{G}|_{F}.$ Here $\pi _{k}:\mathbb{R}^{l}\rightarrow 
\mathbb{R}^{k}$ is projection to the first $k$--factors.
\end{keylemma}

\begin{proof}
By Lemma \ref{const rank def} there is a submersion $\psi :\tilde{G}\cap 
\tilde{O}\longrightarrow \psi \left( \tilde{G}\cap \tilde{O}\right) \subset 
\mathbb{R}^{l}$ so that%
\begin{equation*}
\psi =\left\{ 
\begin{array}{cc}
\mu \circ p_{G} & \text{on }\tilde{W}\cap \tilde{O} \\ 
\mu \circ p_{O} & \text{on }\left( \tilde{G}\setminus \tilde{V}\right) \cap 
\tilde{O}%
\end{array}%
\right.
\end{equation*}%
and%
\begin{equation*}
\left\vert \psi -\mu \circ p_{G}\right\vert _{C^{1}}<\varepsilon +\frac{2\xi 
}{\zeta }\text{ and }\left\vert \psi -\mu \circ p_{O}\right\vert
_{C^{1}}<\varepsilon +\frac{2\xi }{\zeta }.
\end{equation*}%
Therefore, the map 
\begin{equation*}
P:\tilde{W}\cup \tilde{O}\longrightarrow S
\end{equation*}%
defined by%
\begin{equation*}
P:=\left\{ 
\begin{array}{cc}
p_{G} & \text{on }\,\,\,\tilde{W} \\ 
\mu ^{-1}\circ \psi & \text{on }\,\,\,\tilde{G}\cap \tilde{O} \\ 
p_{O} & \text{on }\,\,\,\tilde{O}\setminus \tilde{V}%
\end{array}%
\right.
\end{equation*}%
is a well defined submersion satisfying Equation \autoref{pieces of P eqn}.
Combining the definition of $P$ with Parts 2, 3, 4, and 5 of the Submersion
Deformation Lemma gives us Parts 1, 2, 3, and 4.
\end{proof}

\addtocounter{algorithm}{1}

\subsection{Stability of Intersection Patterns\label{inter stab subsub}}

\begin{proposition}
\label{cover stab prop}Let $\mathcal{C}$ be an ordered collection of $m$
open subsets of a compact metric space $X.$ Suppose that $\mathcal{C}$ and 
\textrm{cl}$\left( \mathcal{C}\right) $ have the same intersection pattern.
Let $\mathcal{X}$ be the collection of compact subsets of $X$ equipped with
the Hausdorff metric, and let $\mathcal{X}^{m}$ be the $m$--fold product of $%
\mathcal{X}.$

There is a neighborhood $\mathcal{N}$ of \textrm{cl}$\left( \mathcal{C}%
\right) $ in $\mathcal{X}^{m}$ with the following property: If $\mathcal{D}$
is a collection of $m$ open subsets of $X$ with \textrm{cl}$\left( \mathcal{D%
}\right) \in \mathcal{N},$ then $\mathcal{D}$ and $\mathcal{C}$ have the
same intersection pattern.
\end{proposition}

\begin{proof}
Since $\mathcal{C}$ and \textrm{cl}$\left( \mathcal{C}\right) $ have the
same intersection pattern, there is an $\varepsilon >0$ so that if $%
C_{i},C_{j}\in \mathcal{C}$ are disjoint, then $\mathrm{dist}\left(
c_{i},c_{j}\right) >\varepsilon $ for all $c_{i}\in C_{i}$ and $c_{j}\in
C_{j}.$ It follows that if $D_{i},D_{j}\in \mathcal{D}$ are close enough to $%
C_{i}$ and $C_{j},$ then $D_{i}$ and $D_{j}$ are disjoint.

On the other hand, if $x\in C_{i}\cap C_{j}$, then there is an $\eta >0$ so
that $B\left( x,\eta \right) \subset C_{i}\cap C_{j}.$ It follows that $%
D_{i}\cap D_{j}\neq \emptyset $ if the Hausdorff distances satisfy 
\begin{equation*}
\mathrm{dist}_{Haus}\left( C_{i},D_{i}\right) <\frac{\eta }{10}\text{ and }%
\mathrm{dist}_{Haus}\left( C_{j},D_{j}\right) <\frac{\eta }{10}.
\end{equation*}
\end{proof}

\begin{proposition}
\label{Push forw intersection}Adopt the hypotheses of Theorem \ref%
{submersion gluing}, and let 
\begin{equation*}
P_{k}:\cup _{i=1}^{k}\tilde{B}_{i}\left( \rho \right) \longrightarrow
P_{k}\left( \cup _{i=1}^{k}\tilde{B}_{i}\left( \rho \right) \right) \subset S
\end{equation*}%
be a submersion with 
\begin{equation}
\left\vert \left. P_{k}-\mu _{i}^{-1}\circ \tilde{p}_{i}\right\vert _{\tilde{%
B}_{i}\left( \rho \right) }\right\vert _{C^{0}}<\xi
\label{delta close
ineqaul}
\end{equation}%
for all $i.$ If $\xi $ is sufficiently small, then 
\begin{equation*}
P_{k}\left( \cup _{i=1}^{k}\tilde{B}_{i}\left( \rho \right) \right) \cap
B_{k+1}\left( \rho \right) \neq \emptyset
\end{equation*}%
if and only if 
\begin{equation*}
\cup _{i=1}^{k}B_{i}\left( \rho \right) \cap B_{k+1}\left( \rho \right) \neq
\emptyset .
\end{equation*}
\end{proposition}

\begin{proof}
We have $P_{k}\left( \cup _{i=1}^{k}\tilde{B}_{i}\left( \rho \right) \right)
=\cup _{i=1}^{k}P_{k}\left( \tilde{B}_{i}\left( \rho \right) \right) ,$ and
Inequalities \autoref{delta close ineqaul} and \autoref{Huas inequl} give us
that $\cup _{i=1}^{k}P_{k}\left( \tilde{B}_{i}\left( \rho \right) \right) $
is Hausdorff close to $\cup _{i=1}^{k}B_{i}\left( \rho \right) .$ So by
Proposition \ref{cover stab prop}, 
\begin{equation*}
\cup _{i=1}^{k}B_{i}\left( \rho \right) \dbigcap B_{k+1}\left( \rho \right)
\neq \emptyset
\end{equation*}%
if and only if%
\begin{equation*}
P_{k}\left( \cup _{i=1}^{k}\tilde{B}_{i}\left( \rho \right) \right) \dbigcap
B_{k+1}\left( \rho \right) \neq \emptyset .
\end{equation*}
\end{proof}

\addtocounter{algorithm}{1}

\subsection{Proofs of Theorem \protect\ref{submersion gluing}, Corollary 
\protect\ref{respt other cor}, and Corollary \protect\ref{gluing addendum}}

\begin{proof}[Proof of Theorem \protect\ref{submersion gluing}]
Choose $\varepsilon _{0}>0$ so that 
\begin{equation*}
\varepsilon _{0}<\frac{\varepsilon _{1}}{2},
\end{equation*}%
where $\varepsilon _{1}$ is as in Lemma \ref{Rank Isotopy Lemma}. Choose $%
\xi _{0},\eta >0$ so that the conclusion of Proposition \ref{Push forw
intersection} holds with $\xi =\xi _{0}$ and so that%
\begin{equation*}
\left( 1+\eta \right) ^{2\left( \mathfrak{o}-1\right) }\varepsilon _{0}+%
\frac{2}{\rho }\xi _{0}\left( \mathfrak{o}-1\right) \left( 1+\eta \right)
^{2\left( \mathfrak{o}-1\right) }<\frac{\varepsilon _{1}}{2}.
\end{equation*}

Next we partition $\left\{ \tilde{B}_{i}\left( 3\rho \right) \right\}
_{i=1}^{m_{l}}$ into $\mathfrak{o}$ subcollections of pairwise disjoint
balls, where $\mathfrak{o}$ is the first order of $\left\{ \tilde{B}%
_{i}\left( 3\rho \right) \right\} _{i=1}^{m_{l}}.$ To begin, we take $%
\mathcal{\tilde{B}}_{1}\left( 3\rho \right) $ to be a maximal subcollection
of $\left\{ \tilde{B}_{i}\left( 3\rho \right) \right\} _{i=1}^{m_{l}}$ that
is pairwise disjoint, and in general, for $j\in \left\{ 2,\ldots ,\mathfrak{o%
}\right\} ,$ we take $\mathcal{\tilde{B}}_{j}\left( 3\rho \right) $ to be a
maximal pairwise disjoint subcollection of $\left\{ \tilde{B}_{i}\left(
3\rho \right) \right\} _{i=1}^{m_{l}}\setminus \left\{ \mathcal{\tilde{B}}%
_{1}\left( 3\rho \right) \cup \cdots \cup \mathcal{\tilde{B}}_{j-1}\left(
3\rho \right) \right\} .$ Then every element of $\mathcal{\tilde{B}}%
_{j}\left( 3\rho \right) $ intersects at least one element from each of $%
\mathcal{\tilde{B}}_{1}\left( 3\rho \right) ,\cdots ,\mathcal{\tilde{B}}%
_{j-1}\left( 3\rho \right) ,$ so the first order of the collection $\mathcal{%
\tilde{B}}_{1}\left( 3\rho \right) \cup \cdots \cup \mathcal{\tilde{B}}%
_{j}\left( 3\rho \right) $ is at least $j.$ Therefore for $j\geq \mathfrak{o}%
+1,$ $\mathcal{\tilde{B}}_{j}\left( 3\rho \right) =\emptyset ,$ and $%
\mathcal{\tilde{B}}_{1}\left( 3\rho \right) \cup \cdots \cup \mathcal{\tilde{%
B}}_{\mathfrak{o}}\left( 3\rho \right) =\left\{ \tilde{B}_{i}\left( 3\rho
\right) \right\} _{i=1}^{m_{l}}.$

We let $\mathcal{\tilde{B}}_{j}\left( \rho \right) $ be the $\rho $--balls
that have the same centers as the $\mathcal{\tilde{B}}_{j}\left( 3\rho
\right) $s, and we let $\mathcal{B}_{j}\left( 3\rho \right) $ and $\mathcal{B%
}_{j}\left( \rho \right) $ be the corresponding subcollections of $\left\{
B_{j}\left( 3\rho \right) \right\} _{j=1}^{m_{l}}$ and $\left\{ B_{j}\left(
\rho \right) \right\} _{j=1}^{m_{l}}.$ We use the superscript $^{u}$ to
denote the union of one of these subcollections. Thus for example, $\mathcal{%
\tilde{B}}_{1}^{u}\left( 3\rho \right) $ is the subset of $M$ obtained by
taking the union of each ball in $\mathcal{\tilde{B}}_{1}\left( 3\rho
\right) .$

For each $j\in \left\{ 1,2,\ldots ,\mathfrak{o}\right\} $ and each $i$ with $%
\tilde{B}_{i}\left( 3\rho \right) \in \mathcal{\tilde{B}}_{j}\left( 3\rho
\right) ,$ we let 
\begin{equation*}
\hat{p}_{j}:\mathcal{\tilde{B}}_{j}^{u}\left( 3\rho \right) \longrightarrow 
\mathbb{R}^{l}\text{ }
\end{equation*}%
be given by 
\begin{equation*}
\hat{p}_{j}|_{B_{i}\left( 3\rho \right) }=\tilde{p}_{i},
\end{equation*}%
and 
\begin{equation*}
\hat{\mu}_{j}:\mathcal{B}_{j}^{u}\left( 3\rho \right) \longrightarrow 
\mathbb{R}^{l}
\end{equation*}%
be given by%
\begin{equation*}
\hat{\mu}_{j}|_{B_{i}\left( 3\rho \right) }=\mu _{i}.
\end{equation*}

The proof is by induction on the index $j$ of the $\mathcal{\tilde{B}}%
_{j}\left( 3\rho \right) $s. To formulate our induction statement for $k\in
\left\{ 1,\ldots ,\mathfrak{o}\right\} ,$ we set 
\begin{equation}
\mathcal{E}_{k}=\left( 1+\eta \right) ^{2\left( k-1\right) }\varepsilon +%
\frac{2}{\rho }\xi \left( k-1\right) \left( 1+\eta \right) ^{2\left(
k-1\right) }.  \label{dfn of E_k}
\end{equation}

Our $k^{th}$ statement asserts the existence of a submersion 
\begin{equation*}
P_{k}:\cup _{j=1}^{k}\mathcal{\tilde{B}}_{j}^{u}\left( \rho \right)
\longrightarrow P_{k}\left( \cup _{j=1}^{k}\mathcal{\tilde{B}}_{j}^{u}\left(
\rho \right) \right) \subset S
\end{equation*}%
so that for all $s\in \left\{ 1,2,\ldots ,\mathfrak{o}\right\} $ on $\cup
_{j=1}^{k}\mathcal{\tilde{B}}_{j}^{u}\left( \rho \right) \cap \mathcal{%
\tilde{B}}_{s}^{u}\left( 3\rho \right) ,$%
\begin{equation}
\left\vert \hat{\mu}_{k}\circ P_{k}-\hat{\mu}_{k}\circ \hat{\mu}%
_{s}^{-1}\circ \hat{p}_{s}\right\vert _{C^{0}}<\left( 1+\eta \right)
^{2k}\xi \text{ and}  \label{C-0 induct
hyp}
\end{equation}%
\begin{equation}
\left\vert \hat{\mu}_{k}\circ P_{k}-\hat{\mu}_{k}\circ \hat{\mu}%
_{s}^{-1}\circ \hat{p}_{s}\right\vert _{C^{1}}<\mathcal{E}_{k}.
\label{C^1 Ind hyp
Inequal}
\end{equation}

Setting $P_{1}=\hat{\mu}_{1}^{-1}\circ \hat{p}_{1}$ and appealing to
Equations \autoref{C-0 close inequal} and \autoref{C-1 close subm inequality}
anchors the induction.

Since the collection $\left\{ \mathcal{\tilde{B}}_{j}^{u}\left( \rho \right)
\right\} _{j=1}^{\mathfrak{o}}$ has first order $\mathfrak{o,}\left( \cup
_{j=1}^{k}\mathcal{\tilde{B}}_{j}^{u}\left( \rho \right) \right) \cap 
\mathcal{\tilde{B}}_{k+1}^{u}\left( \rho \right) \neq \emptyset .$ Combining
this with $\mathcal{E}_{k}<\mathcal{E}_{\mathfrak{o}}<\varepsilon _{1}$
allows us to apply Key Lemma \ref{Gluing, Abstract} with $p_{O}=P_{k}$ and $%
p_{G}=\hat{\mu}_{k+1}^{-1}\circ \hat{p}_{k+1}$ to get a new submersion 
\begin{equation*}
P_{k+1}:\cup _{j=1}^{k+1}\mathcal{\tilde{B}}_{j}^{u}\left( \rho \right)
\longrightarrow P_{k+1}\left( \cup _{j=1}^{k+1}\mathcal{\tilde{B}}%
_{j}^{u}\left( \rho \right) \right) \subset S.
\end{equation*}

It remains to verify hypotheses $(\ref{C-0 induct hyp})_{k+1}$ and $\left( %
\ref{C^1 Ind hyp Inequal}\right) _{k+1}$. The induction hypothesis, $(\ref%
{C-0 induct hyp})_{k}$, combined with our hypothesis that the differentials
of the $\hat{\mu}_{i}$s are $\left( 1+\eta \right) $--bi-lipshitz gives%
\begin{eqnarray*}
\left\vert \hat{\mu}_{k+1}\circ P_{k}-\hat{\mu}_{k+1}\circ \hat{\mu}%
_{s}^{-1}\circ \hat{p}_{s}\right\vert _{C^{0}} &=&\left\vert \left( \hat{\mu}%
_{k+1}\circ \hat{\mu}_{k}^{-1}\right) \circ \left( \hat{\mu}_{k}\circ P_{k}-%
\hat{\mu}_{k}\circ \hat{\mu}_{s}^{-1}\circ \hat{p}_{s}\right) \right\vert
_{C^{0}} \\
&<&\left( 1+\eta \right) ^{2}\left( 1+\eta \right) ^{2k}\xi \\
&=&\left( 1+\eta \right) ^{2\left( k+1\right) }\xi .
\end{eqnarray*}%
So by Part 3 of the Key Lemma \ref{Gluing, Abstract},%
\begin{equation*}
\left\vert \hat{\mu}_{k+1}\circ P_{k+1}-\hat{\mu}_{k+1}\circ \left( \hat{\mu}%
_{s}^{-1}\circ \hat{p}_{s}\right) \right\vert _{C^{0}}<\left( 1+\eta \right)
^{2\left( k+1\right) }\xi ,
\end{equation*}%
and $\left( \ref{C-0 induct hyp}\right) _{k+1}$ holds.

Combining $\left( \ref{C^1 Ind hyp Inequal}\right) _{k}$ with the fact that
the differentials of the $\hat{\mu}_{i}$s are $\left( 1+\eta \right) $%
--bi-lipshitz gives%
\begin{eqnarray*}
\left\vert \hat{\mu}_{k+1}\circ P_{k}-\hat{\mu}_{k+1}\circ \hat{\mu}%
_{s}^{-1}\circ \hat{p}_{s}\right\vert _{C^{1}} &=&\left\vert \hat{\mu}%
_{k+1}\circ \hat{\mu}_{k}^{-1}\circ \left( \hat{\mu}_{k}\circ P_{k}-\hat{\mu}%
_{k}\circ \hat{\mu}_{s}^{-1}\circ \hat{p}_{s}\right) \right\vert _{C^{1}} \\
&<&\left( 1+\eta \right) ^{2}\left( \mathcal{E}_{k}\right) .
\end{eqnarray*}

So by Part 2 of Key Lemma \ref{Gluing, Abstract} and $\left( \ref{C-0 induct
hyp}\right) _{k},$ 
\begin{eqnarray*}
\left\vert \hat{\mu}_{k+1}\circ \hat{P}_{k+1}-\hat{\mu}_{k+1}\circ \hat{\mu}%
_{s}^{-1}\circ \hat{p}_{s}\right\vert _{C^{1}} &<&\left( 1+\eta \right)
^{2}\left( \mathcal{E}_{k}\right) +\frac{2}{\rho }\left( 1+\eta \right)
^{2k}\xi \\
&=&\left( 1+\eta \right) ^{2}\left( \left( 1+\eta \right) ^{2\left(
k-1\right) }\varepsilon +\frac{2}{\rho }\xi \left( k-1\right) \left( 1+\eta
\right) ^{2\left( k-1\right) }\right) \\
&&+\frac{2}{\rho }\left( 1+\eta \right) ^{2k}\xi \\
&=&\left( 1+\eta \right) ^{2k}\varepsilon +\frac{2}{\rho }\xi k\left( 1+\eta
\right) ^{2k} \\
&=&\mathcal{E}_{k+1}.
\end{eqnarray*}

To complete the proof, we need to establish Equation \autoref{last ball eq}.
To do so, we re-index so that $\tilde{B}_{m_{l}}\left( \rho \right) \subset $
$\mathcal{\tilde{B}}_{\mathfrak{o}}^{u}\left( 3\rho \right) $ and notice
that 
\begin{equation*}
P|_{\mathcal{\tilde{B}}_{\mathfrak{o}}^{u}\left( \rho \right) }=P_{\mathfrak{%
o}}|_{\tilde{B}_{m_{l}}\left( \rho \right) }=\hat{\mu}_{\mathfrak{o}%
}^{-1}\circ \hat{p}_{\mathfrak{o}}
\end{equation*}%
by Equation \autoref{pieces of P eqn}.
\end{proof}

\begin{proof}[Proof of Corollary \protect\ref{respt other cor}]
This is a consequence of Part 4 of Key Lemma \ref{Gluing, Abstract} and the
observation that at the $k^{th}$--stage of the induction, we glue $%
p_{O}=P_{k}$ to $p_{G}=\hat{\mu}_{k+1}^{-1}\circ \hat{p}_{k+1}.$
\end{proof}

\begin{proof}[Proof of Corollary \protect\ref{gluing addendum}]
First apply Theorem \ref{submersion gluing} to construct a submersion $%
\tilde{P}:\cup _{i=1}^{m_{l}}\tilde{B}_{i}\left( \rho \right)
\longrightarrow S$ that is close to the $\tilde{p}_{i}$s in the sense that
Inequalities \autoref{still C-0 close} and \autoref{geom inequl} hold.

Since the first order of $\left\{ B_{i}\left( 3\rho _{R}\right) \right\}
_{i\in I_{R}}$ is $\mathfrak{o},$ as in the proof of Theorem \ref{submersion
gluing}, for each $j\in \left\{ 1,2,\ldots ,\mathfrak{o}\right\} ,$ we
construct a subcollection $\mathcal{B}_{j}\left( 3\rho _{R}\right) $ of $%
\left\{ B_{i}\left( 3\rho _{R}\right) \right\} _{i\in I_{R}}$ so that the
balls of $\mathcal{B}_{j}\left( 3\rho _{R}\right) $ are pairwise disjoint,
and the collection $\mathcal{B}_{1}\left( 3\rho _{R}\right) \cup \cdots \cup 
\mathcal{B}_{j}\left( 3\rho _{R}\right) $ has first order at least $j.$

For each $j\in \left\{ 1,2,\ldots ,\mathfrak{o}\right\} ,$ we set 
\begin{equation*}
p_{j}\equiv \mu _{j}\circ Q\circ R:\mathcal{B}_{j}^{u}\left( 3\rho
_{R}\right) \longrightarrow \mathbb{R}^{l},
\end{equation*}%
and note that since the $p_{j}$s are all coordinate representations of the
same submersion, $Q\circ R,$ 
\begin{equation}
\text{Inequalities }\autoref{C-0 close inequal}\text{ and }\autoref{C-1
close subm inequality}\text{ hold with }\xi =\varepsilon =0
\label{C0 and C1
hold}
\end{equation}%
and the $p_{i}$s playing the role of the $\tilde{p}_{i}$s. Using this, for
each $j\in \left\{ 1,2,\ldots ,\mathfrak{o}\right\} ,$ we successively apply
the proof of Theorem \ref{submersion gluing} to deform $\tilde{P}$ on each $%
\mathcal{B}_{j}\left( 3\rho _{R}\right) $ so that it ultimately equals $%
Q\circ R$ on $\cup _{j=1}^{\mathfrak{o}}\mathcal{B}_{i}^{u}\left( \rho
_{R}\right) .$ For the first deformation, this is possible because
Inequalities (\ref{resp c-0 close inequal}), (\ref{resp c-1 close ineqaul}),
and (\ref{C0 and C1 hold}) tell us that the $p_{j}$s are close to the $%
\tilde{p}_{j}$s. Via (\ref{still C-0 close}) and (\ref{geom inequl}) it
follows that the $p_{j}$s are close to local representations of $\tilde{P}.$
In other words, we have that Inequalities (\ref{pi_l vs P_0}) and (\ref{pi_i
vs P_0--C^0}) hold with $p_{O}=\tilde{P}$ and $p_{G}=Q\circ R.$ This
continues to be possible for subsequent deformations because Parts 2 and 3
of Key Lemma \ref{Gluing, Abstract} tell us our deformations preserve
Inequalities (\ref{pi_l vs P_0}) and (\ref{pi_i vs P_0--C^0}), provided $\xi 
$ and $\varepsilon $ are sufficiently small.

To explain why $P=Q\circ R$ on $\cup _{j=1}^{\mathfrak{o}}\mathcal{B}%
_{j}^{u}\left( \rho _{R}\right) $, we let $\tilde{P}_{0},$ $\tilde{P}%
_{1},\ldots ,\tilde{P}_{\mathfrak{o}}$ be the deformations of $\tilde{P}=%
\tilde{P}_{0}.$ By combining Equation \autoref{pieces of P eqn} with the
fact that $p_{1}=\mu _{1}\circ Q\circ R,$ it follows that 
\begin{equation*}
\tilde{P}_{1}\equiv Q\circ R
\end{equation*}%
on $\mathcal{B}_{1}^{u}\left( \rho _{R}\right) .$ By the same reasoning, we
have 
\begin{equation*}
\tilde{P}_{k}\equiv Q\circ R
\end{equation*}%
on $\mathcal{B}_{k}^{u}\left( \rho _{R}\right) $, and Part 4 of Lemma \ref%
{Gluing, Abstract} gives, via induction, that after the $k^{th}$
deformation, we have 
\begin{equation*}
\tilde{P}_{k}\equiv Q\circ R
\end{equation*}%
on $\cup _{j=1}^{k}\mathcal{B}_{j}^{u}\left( \rho _{R}\right) .$ So setting $%
P\equiv \tilde{P}_{\mathfrak{o}},$ we see that $P=Q\circ R$ on $\cup _{j=1}^{%
\mathfrak{o}}\mathcal{B}_{j}^{u}\left( \rho _{R}\right) $.
\end{proof}

\section{Appendix B: Conventions and Notations\label{Not and con}}

We assume throughout that all metric spaces are complete, and the reader has
a basic familiarity with Alexandrov spaces, including but not limited to the
seminal paper by Burago, Gromov, and Perelman (\cite{BGP}). Let $X$, $%
\mathcal{S}=\left\{ S_{i}\right\} _{i\in I},$ $\mathcal{N},$ and $\mathcal{K}
$ be as in Theorem \ref{dif stab- dim 4}, and let $p,x,$ and $y$ be points
of $X.$

\begin{enumerate}
\item We call minimal geodesics in $X$ \emph{segments}.

\item We denote comparison angles with $\tilde{\sphericalangle}.$

\item We let $\Sigma _{p}X$ and $T_{p}X$ denote the space of directions and
tangent cone at $p$, respectively, and we let $\ast $ denote the cone point.

\item For a geodesic direction $v\in T_{p}X,$ we let $\gamma _{v}$ be the
segment whose initial direction is $v.$

\item Following \cite{Petr}, given a subset $A\subset X$, $\Uparrow
_{x}^{A}\subset \Sigma _{x}$ denotes the set of directions of segments from $%
x$ to $A,$ and $\uparrow _{x}^{A}\in $ $\Uparrow _{x}^{A}$ denotes the
direction of a single segment from $x$ to $A.$ For $x\in S_{i}\subset X$ and 
$A\subset S_{i},$ we write $\left( \uparrow _{x}^{A}\right) _{S_{i}}$ or $%
\left( \Uparrow _{x}^{A}\right) _{S_{i}}$ if we are referring to intrinsic
segments of $S$ and $\left( \uparrow _{x}^{A}\right) _{X}$ or $\left(
\Uparrow _{x}^{A}\right) _{X}$ if we are referring to extrinsic segments of $%
X.$

\item For a differentiable map $\Phi $ we write $D\Phi $ for the
differential of $\Phi .$ If $\Phi $ is real valued, we write $D_{v}\left(
\Phi \right) $ for the derivative of $\Phi $ in the $v$ direction.

\item Given a subset $A\subset X,$ we say that $\mathrm{dist}_{A}\left(
\cdot \right) $ is $\left( 1-\varepsilon \right) $--regular at $x$ if there
is a $v\in \Sigma _{x}$ so that the derivative of $\mathrm{dist}_{A}\left(
\cdot \right) $ in the direction $v$ satisfies 
\begin{equation*}
D_{v}\mathrm{dist}_{A}>1-\varepsilon .
\end{equation*}

\item We let $px$ denote a segment from $p$ to $x.$

\item We let $\sphericalangle (x,p,y)$ denote the angle of a hinge formed by
segments $px$ and $py$ and $\tilde{\sphericalangle}(x,p,y)$ denote the
corresponding comparison angle.

\item Following \cite{OSY}, we let $\tau :\mathbb{R}^{k}\rightarrow \mathbb{R%
}_{+}$ be any function that satisfies 
\begin{equation*}
\lim_{x_{1},\ldots ,x_{k}\rightarrow 0}\tau \left( x_{1},\ldots
,x_{k}\right) =0,
\end{equation*}%
and, abusing notation, we let $\tau :\mathbb{R}^{k}\times \mathbb{R}%
^{n}\rightarrow \mathbb{R}$ be any function that satisfies 
\begin{equation*}
\lim_{x_{1},\ldots ,x_{k}\rightarrow 0}\tau \left( x_{1},\ldots
,x_{k}|y_{1},\ldots ,y_{n}\right) =0,
\end{equation*}%
provided $y_{1},\ldots ,y_{n}$ remain fixed. When making an estimate with a
function $\tau ,$ we implicitly assert the existence of such a function for
which the estimate holds. $\tau $ often depends on the limit space $X$
and/or its dimension, but we make no other mention of this.

\item We identify $\mathbb{R}^{l}$ with $\mathbb{R}^{l}\times \left\{
0\right\} ,$ and we let $\pi _{l}:\mathbb{R}^{l}\times \mathbb{R}%
^{n-l}\longrightarrow \mathbb{R}^{l}$ be orthogonal projection to the first $%
l$ factors of $\mathbb{R}^{n}.$

\item For $\lambda \in \mathbb{R},$ we call a function $f:\mathbb{R}%
\longrightarrow \mathbb{R}$ (strictly) $\lambda $--concave if and only if
the function $g(t)=f(t)-\lambda t^{2}/2$ is (strictly) concave.

\item If $U$ is an open subset of an Alexandrov space $X,$ we call $%
f:U\longrightarrow \mathbb{R}$, (strictly) $\lambda $--concave if and only
if its restriction to every geodesic is (strictly) $\lambda $--concave$.$

\item We abbreviate the statement \textquotedblleft $\left\{ M_{\alpha
}\right\} _{\alpha =1}^{\infty }$ converges to $X$ in the Gromov--Hausdorff
topology\textquotedblright\ with the symbols, $M_{\alpha }\overset{GH}{%
\longrightarrow }X.$ Similarly, if $f_{\alpha }:M\longrightarrow \mathbb{R}$
and $f:X\longrightarrow \mathbb{R},$ we abbreviate \textquotedblleft $%
\left\{ f_{\alpha }\right\} _{\alpha =1}^{\infty }$ converges to $f$ in the
Gromov--Hausdorff topology\textquotedblright\ with the symbols,write $%
f_{\alpha }\overset{GH}{\longrightarrow }f.$

\item Let $V$ and $W$ be normed vector spaces. For a linear map $%
L:V\longrightarrow W,$ we set $\left\vert L\right\vert =\max \left\{ \left.
\left\vert L\left( \frac{v}{\left\vert v\right\vert }\right) \right\vert 
\text{ }\right\vert \text{ }v\in V\setminus \left\{ 0\right\} \right\} .$

\item Let $U\subset M$ be open and $\Phi :U\longrightarrow \mathbb{R}^{n}$
be $C^{1}$. We write 
\begin{eqnarray*}
\left\vert \Phi \right\vert _{C^{0}} &\equiv &\sup_{x\in U}\left\{
\left\vert \Phi (x)\right\vert \right\} \text{ and } \\
\left\vert \Phi \right\vert _{C^{1}} &\equiv &\max \left\{ \left\vert \Phi
\right\vert _{C^{0}},\sup_{x\in U}\left\{ \left\vert D\Phi _{x}\right\vert
\right\} \right\}
\end{eqnarray*}

\item We call a submersion, $\pi ,$ $\eta $--almost Riemannian if and only
if for all unit horizontal vectors,%
\begin{equation*}
\left\vert D\pi \left( v\right) -1\right\vert <\eta .
\end{equation*}

\item An $\eta $--embedding ($\eta $--homeomorphism) is an embedding
(homeomorphism) that is also an $\eta $--Gromov-Hausdorff approximation.

\item Volume of subsets of Alexandrov spaces means rough volume as defined
in \cite{BGP}.

\item For $\lambda >0,$ we write 
\begin{equation*}
\lambda X
\end{equation*}%
for the metric spaces obtained from $X$ by rescaling all distances by $%
\lambda .$

\item We write $N$ or $N_{i}$ for an element of $\mathcal{N}$; $K$ or $K_{i}$
for an element of $\mathcal{K};$ and $S$ or $S_{i}$ for an element of $%
\mathcal{S}.$ Thus we redundantly write 
\begin{eqnarray*}
\mathcal{S} &=&\left\{ S_{i}\right\} _{i} \\
&=&\left\{ \mathcal{K}_{k}\right\} _{k}\cup \left\{ N_{n}\right\} _{n}.
\end{eqnarray*}

\item We set 
\begin{equation*}
\mathcal{S}^{\mathrm{ext}}\equiv \mathcal{S}\cup \left( X\setminus \cup
_{S\in \mathcal{S}}S\right) .
\end{equation*}

\item We use superscripts to denote components of vectors in subspaces. So,
for example, if $V$ is a subspace of $W,$ then $U^{V}$ is the component of $%
U $ in $V.$

\item We write $\mathbb{S}^{n}$ for the unit sphere in $\mathbb{R}^{n+1}.$

\item We set 
\begin{equation*}
B\left( p,r\right) \equiv \left\{ \left. x\in X\text{ }\right\vert \text{ 
\textrm{dist}}\left( x,p\right) <r\right\} .
\end{equation*}

\item We use $A\Subset B$ to mean that the closure of $A$ is contained in
the interior of $B.$

\item We say that a collection of sets $\mathcal{C}$ has first order $\leq 
\mathfrak{o}$ if and only if each $C\in \mathcal{C}$ intersects no more than 
$\mathfrak{o}-1$ other members of $\mathcal{C}.$
\end{enumerate}

\end{document}